%% file: triple4.tex
\documentclass[10pt]{amsart}
\input{amssymbol}

\DeclareMathOperator{\lcm}{lcm}
\usepackage{color}

\def\rmd{{\rm d}}

\def\ulQ{{\ul{Q}}}

\def\Om{\boldsymbol \omega}
\def\Qx{{Q_1}}
\def\Qy{{Q_2}}
\def\Qz{{Q_3}}

\def\bal{{\rm bal}}

\def\Eta{\boldsymbol\eta}
\def\addchar{\boldsymbol\psi}

\def\bfa{\mathbf a}

\def\bdsf{{\boldsymbol f}}
\def\bdsg{{\boldsymbol g}}
\def\bdsh{{\boldsymbol h}}
\def\cyc{\boldsymbol\varepsilon_{\rm cyc}}

\def\pmq{\ell}

\def\val{v}
\def\ord{{\rm ord}}

\def\condf{N_1}
\def\condg{N_2}
\def\condh{N_3}

\def\bdsF{\boldsymbol F}

\def\divides{\mid}
\def\rmd{d}
\def\itPhi{{\mathit \Phi}}

\def\rmH{{\rm H}}
\def\eord{e_\mathrm{ord}}
\def\sS{\cS}
\def\Qbarp{\ol{\Q}_p}

\def\ari{+}
\def\rmd{\mathrm{d}}

\def\Mat{\mathrm{M}}

\def\gap{\lambda}
\def\hh{T}
\def\WD{\mathrm{WD}}
\def\bigO{\mathrm{o}}
\def\U{\mathrm{U}}
\def\O{\mathrm{O}}
\def\Mat{\mathrm{M}}
\def\adj{\mathrm{adj}}
\def\oo{\hbox{\bf 0}}
\def\ono{\hbox{\bf 1}}
\def\bsl{\backslash}
\def\d{\mathrm{d}}
\def\St{\mathrm{St}}
\def\bdsE{\boldsymbol{E}}

\def\alp{\alpha}
\def\bet{\beta}
\def\gam{\gamma}
\def\del{\delta}
\def\eps{\epsilon}
\def\zet{\zeta}
\def\tht{\theta}
\def\kap{\kappa}
\def\lam{\lambda}
\def\sig{\sigma}
\def\ome{\omega}
\def\vep{\varepsilon}

\def\vpi{\varpi}

\def\vph{\varphi}
\def\Gam{\Gamma}
\def\Del{\Delta}

\def\Lam{\Lambda}

\def\Ome{\Omega}

\def\vPh{\varPhi}

\def\LDiff{\boldsymbol\lambda}

\def\frkf{{\mathfrak f}}    
    
    \def\frkH{{\mathfrak H}}

\def\frko{{\mathfrak o}}    
\def\frkp{{\mathfrak p}}

\def\cal{\fam2}
\def\cala{{\cal A}}

\def\cald{{\cal D}}
\def\cale{{\cal E}}
\def\calf{{\cal F}}

\def\calh{{\cal H}}

\def\calk{{\cal K}}
\def\call{{\cal L}}
\def\calm{{\cal M}}
\def\caln{{\cal N}}
\def\calo{{\cal O}}
\def\calp{{\cal P}}

\def\cals{{\cal S}}

\def\calv{{\cal V}}
\def\calw{{\cal W}}

\def\caly{{\cal Y}}

\def\II{{\mathbb I}}

\def\bfa{{\mathbf a}}    \def\bfA{{\mathbf A}}
    
\def\bfc{{\mathbf c}}    \def\bfC{{\mathbf C}}
\def\bfd{{\mathbf d}}    
    \def\bfE{{\mathbf E}}
\def\bff{{\mathbf f}}

\def\bfi{{\mathbf i}}    \def\bfI{{\mathbf I}}
    
\def\bfk{{\mathbf k}}    \def\bfK{{\mathbf K}}
\def\bfl{{\mathbf l}}    
\def\bfm{{\mathbf m}}    \def\bfM{{\mathbf M}}
\def\bfn{{\mathbf n}}    
    
\def\bfp{{\mathbf p}}    
    \def\bfQ{{\mathbf Q}}
    \def\bfR{{\mathbf R}}
\def\bfs{{\mathbf s}}    \def\bfS{{\mathbf S}}
\def\bft{{\mathbf t}}    \def\bfT{{\mathbf T}}
\def\bfu{{\mathbf u}}    \def\bfU{{\mathbf U}}
    \def\bfV{{\mathbf V}}
    \def\bfW{{\mathbf W}}
    \def\bfX{{\mathbf X}}
    
    \def\bfZ{{\mathbf Z}}

\def\scrd{{\mathscr D}}

\def\scrk{{\mathscr K}}

\def\scrw{{\mathscr W}}

\def\flat{\dagger}

\def\opcpt{U}
\title[Four-variable $p$-adic triple product $L$-functions]{Four-variable $p$-adic triple product $L$-functions and the trivial zero conjecture}
\author{Ming-Lun Hsieh}
\author{Shunsuke Yamana}
\date{\today}
\address{Department of Mathematics and National Center for Theoretic Science, National Taiwan University, No. 1, Sec. 4, Roosevelt Road, Taipei 10617, Taiwan}
\email{mlhsieh@math.ntu.edu.tw}
\address{Department of Mathematics, Graduate School of Science, Osaka Metropolitan University, 3-3-138 Sugimoto, Sumiyoshi-ku, Osaka 558-8585, Japan}
\email{yamana@omu.ac.jp}

\thanks{During the preparation of this article, Hsieh was partially supported by NSTC grants 108-2628-M-001-009-MY4 and 110-2628-M-001-004-. 
Yamana is partially supported by JSPS Grant-in-Aid for Scientific Research {(C)18K03210, (C)23K03055 and (B)19H01778}. }
\begin{document}
\begin{abstract}
We construct the four-variable primitive $p$-adic $L$-functions associated with the triple product of Hida families and prove the explicit interpolation formulae at all critical points in the balanced range. Our construction is to carry out the $p$-adic interpolation of Garrett's integral representation of triple product $L$-functions via the $p$-adic Rankin-Selberg convolution method. The main novelty in this paper is the construction and the patching of four $p$-adic families of the pull-back of nearly holomorphic Siegel Eisenstein series on $\GSp(6)$.  As an application, we obtain the cyclotomic $p$-adic $L$-function for the motive associated with the triple product of $p$-ordinary elliptic curves and prove the trivial zero conjecture for this motive. In particular, this proves the first cases of the Greenberg-Benois trivial zero conjecture where multiple trivial zeros are present and the Galois representation is not of $\GL(2)$-type.\end{abstract}

\maketitle

\tableofcontents

\section{Introduction}
The aim of this paper is to construct the four-variable $p$-adic triple product $L$-functions for the triple product of Hida families of elliptic newforms with explicit interpolation formulae at all critical specializations in the balanced region. This extends the construction of  three-variable $p$-adic $L$-functions in \cite{GS16} and \cite{Hsieh_triple} by incorporating the cyclotomic variable. As an application, we 
prove the trivial zero conjecture for triple product of elliptic curves with semi-stable reduction at $p$. Let {$p$ be an arbitrary prime number} and fix a valuation ring $\cO$ finite flat over $\Zp$. Let $\bfI$ be a normal domain finite flat over the Iwasawa algebra $\Lam=\cO\powerseries{\Gamma}$ of the topological group {$\Gamma=1+\bfp\Zp$, where $\bfp=p$ or $\bfp=4$ according as $p$ is odd or even.} 
Let
\[\bdsF=(\bdsf,\bdsg,\bdsh)\]
be a triplet of primitive Hida families of tame conductor $(\condf,\condg,\condh)$ and nebentypus $(\chi_1,\chi_2,\chi_3)$ with coefficients in $\bfI$. 
We will construct a four-variable Iwasawa function that interpolates the algebraic part of critical values of the triple product $L$-function attached to $\bdsF$ at all balanced critical specializations twisted by Dirichlet characters. 
Our formulae completely comply with the conjectural form described in \cite{CP89}, \cite{Coates89Bourbaki} and \cite{Coates89II}. In order to state our result precisely, we begin with some notation in Hida theory for elliptic modular forms and technical items such as the modified Euler factors at $p$.

\subsection{Galois representations attached to Hida families}\label{SS:1.2}

Given a field $F$, we denote its separable closure by $\overline{F}$ and put $G_F=\Gal(\overline{F}/F)$. 
If $\cF=\sum_{n=1}^\infty\bfa(n,\cF)q^n\in\bfI\powerseries{q}$ is a primitive cuspidal Hida family of tame conductor $N_\cF$ and nebentypus $\chi_\cF$, let $\rho_\cF:G_\Q\to\GL_2(\Frac\bfI)$ be the associated big Galois representation such that $\Tr\rho_\cF(\Frob_\ell)=\bfa(\ell,\cF)$ for primes $\ell\ndivides N_\cF p$, where $\Frob_\ell$ is the geometric Frobenius at $\ell$. Let $V_\cF$ be the natural realization of $\rho_\cF$ inside the \etale cohomology groups of modular curves, so $V_\cF$ is a lattice in $(\Frac\bfI)^2$ with the continuous Galois action via $\rho_\cF$, and the $G_{\Qp}$-invariant subspace $\Fil^0V_\cF:= V_\cF^{I_p}$ fixed by the inertia group $I_p$ at $p$ is free of rank one over $\bfI$ {at least if $p\geq 5$} (\cite[Corollary, page 558]{Ohta00}). 
A point $Q\in\Spec\bfI(\Qbarp)$ is called an arithmetic point if $Q|_{\Gamma}\colon \Gamma\hookrightarrow \Lam^\x{\stackrel Q\longto}\Qbarp^\x$ is given by $Q(x)=x^{k_Q}\ep_Q(x)$ for some integer $k_Q\geq 2$ and a finite order character $\ep_Q:\Gamma\to\Qbarp^\x$. 
Let $\frakX_\bfI^+$ be the set of arithmetic points of $\bfI$. 
For each arithmetic point $Q\in\frakX_\bfI^+$, the specialization $V_{\cF_Q}:=V_{\cF}\ot_{\bfI,Q}\Qbarp$ is the geometric \padic Galois representation associated with the $p$-stabilized newform $\cF_Q=\sum_{n=1}^\infty Q(\bfa(n,\cF))q^n$.

\subsection{Triple product $L$-functions}\label{S:tri.1} 

The $p$-adic cyclotomic character $\cyc:G_\Q\to\Z_p^\times$ is defined by $\sig(\zet)=\zet^{\cyc(\sig)}$ for every $p$-power root $\zet$ of unity and $\sig\in G_\Q$. 
We denote  by $\Q_\infty$ the cyclotomic $\Zp$-extension of $\Q$, by $\Om:G_\Q\to\mu_{p-1}\hookrightarrow\Z_p^\x$ the Teichm\"{u}ller character, and by $\Dmd{\cyc}_T:G_\Q\twoheadrightarrow\Gal(\Q_\infty/\Q)\hookrightarrow \Zp\powerseries{\Gal(\Q_\infty/\Q)}^\x$ the universal cyclotomic character. 
Let
 \begin{align*}
 \bfI_3&=\bfI\wh\ot_\cO\bfI\wh\ot_\cO\bfI, & 
 \bfI_4&=\bfI_3\powerseries{\Gal(\Q_\infty/\Q)}
 \end{align*}
be finite extensions of the three and four-variable Iwasawa algebras. 

Fix $a\in\Z/l\bfZ$ with $l=2\bigl\lceil\frac{p}{2}\bigl\rceil$. The main object of this paper is a construction of the $p$-adic $L$-function for the triple tensor product Galois representation 
 \begin{align*}
 \calv&=V_\bdsf\wh\ot_\cO V_\bdsg\wh\ot_\cO V_\bdsh, & 
 \bfV&=\calv\wh\ot_\cO\Om^a\Dmd{\cyc}_T
 \end{align*}
 of rank eight over $\bfI_4$.
If $(k_1,k_2,k_3)$ is a triplet of positive integers, we say $(k_1,k_2,k_3)$ is balanced if $k_1+k_2+k_3>2k^*$ with $k^*:=\max\stt{k_1,k_2,k_3}$. 
Let $\frakX_{\bfI_3}^\bal$ denote the set of balanced arithmetic points of $(\frakX_\bfI^+)^3$. 
An integer $k$ is said to be critical for $(k_1,k_2,k_3)$ if 
\[k^*\leq k\leq k_1+k_2+k_3-k^*-2.\]
We define the weight space $\frakX^\bal_{\bfI_4}\subset\Spec\bfI_4(\Qbarp)$ to be the set of balanced critical points of $\bfI_4$ given by 
\[\frakX^\bal_{\bfI_4}=\{(\Qx,\Qy,\Qz,P)\in \frakX_{\bfI_3}^\bal\times\frakX_\Lambda^+\mid k_P\text{ is critical for }(k_{Q_1},k_{Q_2},k_{Q_3})\}. \]
For each point $(\ulQ,P)=(\Qx,\Qy,\Qz,P)\in\frakX^\bal_{\bfI_4}$, 
the specialization $\bfV_{(\ulQ,P)}=\calv_\ulQ\ot\cyc^{k_P}\ep_P\Om^{a-k_P}$ is a $p$-adic geometric Galois representation, where $\calv_\ulQ=V_{\bdsf_\Qx}\ot V_{\bdsg_\Qy}\ot V_{\bdsh_\Qz}$ and $\ep_P$ is regarded as a Galois character via $\ep_P\circ\cyc$. 

Next we briefly recall the motivic $L$-function associated with the specialization $\bfV_{(\ulQ,P)}$. 
To the geometric $p$-adic Galois representation $\bfV_{(\ulQ,P)}$, we can associate the Weil-Deligne representation $\WD_\ell(\bfV_{(\ulQ,P)})$ of the Weil-Deligne group of $\Q_\ell$ over $\Qbarp$ (See \cite[(4.2.1)]{Tate79Corvallis} for $\ell\not =p$ and \cite[(4.2.3)]{F94} for $\ell=p$). 
Fixing an isomorphism $\iota_p:\Qbarp\iso\C$ once and for all, we define the {motivic} $L$-function of $\bfV_{(\ulQ,P)}$ by the Euler product \[L(\bfV_{(\ulQ,P)},s)=\prod_{\ell<\infty}L_\ell(\bfV_{(\ulQ,P)},s)\] of the local $L$-factors $L_\ell(\bfV_{(\ulQ,P)},s)$ attached to 
$\WD_\ell(\bfV_{(\ulQ,P)})\ot_{\Qbarp,\iota_p}\C$ (\cf \cite[(1.2.2)]{Deligne79}, \cite[page 85]{Taylor04ICM}). 
On the other hand, we denote by $\pi_{\bdsf_\Qx}$ (resp. $\pi_{\bdsg_\Qy},\pi_{\bdsh_\Qz}$) the irreducible unitary cuspidal automorphic representation of $\GL_2(\A)$ associated with $\bdsf_\Qx$ (resp. $\bdsg_\Qy,\bdsh_\Qz$). 
Let $L(s,\pi_{\bdsf_\Qx}\times\pi_{\bdsg_\Qy}\times\pi_{\bdsh_\Qz}\otimes\eps_P\Om^{a-k_P})$ be the automorphic $L$-function attached to the triple product of $\pi_{\bdsf_\Qx}$, $\pi_{\bdsg_\Qy}$, and $\pi_{\bdsh_\Qz}\ot\eps_P\Om^{a-k_P}$, as constructed by Garrett \cite{Garrett} in the classical setting and by Piatetski-Shapiro and Rallis \cite{PSR87} in the ad\`{e}lic setting. 
The analytic theory of $L(s,\pi_{\bdsf_\Qx}\times\pi_{\bdsg_\Qy}\times\pi_{\bdsh_\Qz}\otimes\eps_P\Om^{a-k_P})$ such as meromorphic continuation and a functional equation has been explored extensively in the literatures (\cf\cite{PSR87,Ikeda89,Ikeda2}), and thanks to \cite[Theorem 4.4.1]{Dinakar00}, we have
\[L(s+k_P-w_\ulQ/2,\pi_{\bdsf_\Qx}\times\pi_{\bdsg_\Qy}\times\pi_{\bdsh_\Qz}\otimes\eps_P\Om^{a-k_P})=\Gamma_{\bfV_{(\ulQ,P)}}(s)\cdot L(\bfV_{(\ulQ,P)},s),\]
where $w_\ulQ:=k_\Qx+k_\Qy+k_\Qz-3$ and $\Gamma_{\bfV_{(\ulQ,P)}}(s)$ is the Gamma factor of $\bfV_{(\ulQ,P)}$ as given by 
 \[\Gamma_{\bfV_{(\ulQ,P)}}(s):=\Gamma_\C(s+k_P)\Gamma_\C(s+1+k_P-k_\Qx)\Gamma_\C(s+1+k_P-k_\Qy)\Gamma_\C(s+1+k_P-k_\Qz). \]
Here $\Gamma_\C(s)=2(2\pi)^{-s}\Gamma(s)$. 
Hence we have a good understanding of the analytic properties of the motivic $L$-function  $L(\bfV_{(\ulQ,P)},s)$. 
The rationality of its critical $L$-values in the balanced region was proved in \cite{Orl87} and \cite{GH93AJM}, where the authors verify that the Deligne's period for $\bfV_{(\ulQ,P)}$ is the product of Petersson norms of $\bdsf_\Qx$, $\bdsg_\Qy$, $\bdsh_\Qz$. In this article we shall investigate the arithmetic of critical values $L(\bfV_{(\ulQ,P)},0)$ for $(\ulQ,P)\in\frakX^\bal_{\bfI_4}$ and study the $p$-adic analytic behavior of its algebraic part viewed as a function on the weight space $\frakX^\bal_{\bfI_4}$.  

\subsection{The modified Euler factors at $p$ and $\infty$}\label{S:mod.1}

Let $G_{\Qp}$ denote the decomposition group at $p$. 
Define the rank four $G_{\Qp}$-invariant subspace of $\bfV$ by
\[\Fil^+\bfV:=\Fil^+\calv\ot \Om^a\Dmd{\cyc}_T, \]
where 
\[\Fil^+\calv:=\Fil^0V_{\bdsf}\ot \Fil^0V_{\bdsg}\ot V_{\bdsh}+V_{\bdsf}\ot \Fil^0V_{\bdsg}\ot \Fil^0V_{\bdsh}+\Fil^0V_{\bdsf}\ot V_{\bdsg}\ot \Fil^0V_{\bdsh}. \]
The pair $(\Fil^+\bfV,\frakX^\bal_{\bfI_4})$ satisfies the \emph{Panchishkin condition} in \cite[page 217]{Greenberg94} in the sense that for each arithmetic point $(\ulQ,P)\in\frakX^\bal_{\bfI_4}$, the Hodge-Tate numbers of $\Fil^+\bfV_{(\ulQ,P)}$ are all positive, while none of the Hodge-Tate numbers of $\bfV_{(\ulQ,P)}/\Fil^+\bfV_{(\ulQ,P)}$ is positive. Here the Hodge-Tate number of $\Qp(1)$ is one in our convention. 
Now we can define the modified $p$-Euler factor by 
\[\cE_p(\Fil^+\bfV_{(\ulQ,P)}):=\frac{L_p(\Fil^+\bfV_{(\ulQ,P)},0)}{\varepsilon(\mathrm{WD}_p(\Fil^+\bfV_{(\ulQ,P)}))\cdot L_p((\Fil^+\bfV_{(\ulQ,P)})^\vee,1)}\cdot\frac{1}{L_p(\bfV_{(\ulQ,P)},0)}. \]
We note that this modified $p$-Euler factor is precisely the ratio between the factor $\cL_p^{(\rho)}(\bfV_{(\ulQ,P)})$ in \cite[page 109, (18)]{Coates89II} and the local $L$-factor $L_p(\bfV_{(\ulQ,P)},0)$.
 
 In the theory of $p$-adic $L$-functions, we also need the modified Euler factor $\cE_\infty(\bfV_{(\ulQ,P)})$ at the archimedean place observed by Deligne.  It is defined to be the ratio between the factor $\cL_\infty^{(\sqrt{-1})}(\bfV_{(\ulQ,P)})$ in \cite[page 103 (4)]{Coates89II} and the Gamma factor $\Gamma_{\bfV_{(\ulQ,P)}}(0)$. 
In our current case it is explicitly given by \[ \cE_\infty(\bfV_{(\ulQ,P)})=(\sqrt{-1})^{k_\Qx+k_\Qy+k_\Qz-3}.\]

\subsection{The $p$-modified periods}\label{S:period.1}
To give the precise definition of periods for the motive $\bfV_{(\ulQ,P)}$, we introduce the $p$-modified period of an $\bfI$-adic primitive cuspidal Hida family $\cF$ of tame conductor $N_\cF$. 
We denote by $\cF_Q^\circ$ the normalized newform of weight $k_Q$, conductor $N_Q=N_\cF p^{n_Q}$ with nebentypus $\chi_Q$ corresponding to $\cF_Q$. 
There is a unique decomposition $\chi_Q=\chi'_Q\chi_{Q,(p)}$, where  $\chi_Q'$ and $\chi_{Q,(p)}$ are Dirichlet characters modulo $N_\cF$ and $p^{n_Q}$ respectively. 
Let $\al_Q=\bfa(p,\cF_Q)$. 
Define the modified Euler factor $\cE_p(\cF_Q,\Ad)$ for the adjoint motive of $\cF_Q$ by
\[\cE_p(\cF_Q,\Ad)=\begin{cases}
(1-\al_Q^{-2}\chi_Q(p)p^{k_Q-1})(1-\al_Q^{-2}\chi_Q(p)p^{k_Q-2})&\text{ if }n_Q=0,\\
-1&\text{ if }n_Q=1,\chi_{Q,(p)}=1\,(\text{so }k_Q=2),\\
(\al_Q^{-2}p^{(k_Q-2)})^{n_Q}\frakg(\chi_{Q,(p)})\chi_{Q,(p)}(-1)&\text{ if }n_Q>0,\,\chi_{Q,(p)}\not=1.
\end{cases}\]
Here $\frakg(\chi_{Q,(p)})$ is the usual Gauss sum. 
Fixing the choice of a generator $\eta_\cF$ and letting $\norm{\cF_Q^\circ}^2_{\Gamma_0(N_Q)}$ be the usual Petersson norm of $\cF_Q^\circ$, we define the \emph{$p$-modified period} $\Omega^\flat_{\cF_Q}$ of $\cF$ at $Q$ by
 \[\Omega^\flat_{\cF_Q}:=(-2\sqrt{-1})^{k_Q+1}\cdot \norm{\cF_Q^\circ}^2_{\Gamma_0(N_Q)}\cdot \cE_p(\cF_Q,\Ad)\in\C^\x. \]
By \cite[Corollary 6.24, Theorem 6.28]{Hida16Pune}, one can show that for each arithmetic point $Q$, up to a $p$-adic unit, the product of the normalized period $\Omega^\flat_{\cF_Q}$ and the congruence number of $\cF_Q$ is equal to the product of the plus/minus canonical periods $\Omega(+\,;\cF_Q^\circ)\Omega(-\,;\cF_Q^\circ)$ introduced in \cite[page 488]{Hida94Duke}. 

 \subsection{Statement of the main result}
  We impose the following technical assumption:
\beqcd{sf} \text {$N_i$ is square-free and $\chi_i=\Om^{a_i}$ is a power of the Teichm\"{u}ller character for $i=1,2,3$}.\eeqcd
  Our main result is a construction of the balanced $p$-adic triple product $L$-functions with the the interpolation formulae at all critical points in the precise form conjectured by Coates and Perrin-Riou: 
  
  \begin{thmA}\label{thm:main}
{Suppose that \eqref{sf} holds. 
Put $l=2\bigl\lceil\frac{p}{2}\bigl\rceil$. 
Then for each $a\in\Z/l\Z$, }there exists an element
\[\cL_{\bdsF,(a)}\in\bfI_3\powerseries{\Gal(\Q_\infty/\Q)}\otimes_{\bfI_3}(\Frac\bfI\otimes\Frac\bfI\otimes\Frac\bfI)\] 
such that 
\begin{itemize}
\item for each balanced critical $(\ulQ,P)=(\Qx,\Qy,\Qz,P)\in\frakX^\bal_{\bfI_4}$,
\begin{align*} 
\cL_{\bdsF,(a)}(\ulQ,P)=&\frac{\Gamma_{\bfV_{(\ulQ,P)}}(0)L(\bfV_{(\ulQ,P)},0)}{\Omega^\flat_{\bdsf_\Qx}\Omega^\flat_{\bdsg_\Qy}\Omega^\flat_{\bdsh_\Qz}}\cdot (\sqrt{-1})^{k_\Qx+k_\Qy+k_\Qz-3}\cdot \cE_p(\Fil^+\bfV_{(\ulQ,P)});
\end{align*}
\item for any $H_1,H_2$ and $H_3$ in the congruence ideals of $\bdsf,\bdsg$ and $\bdsh$, 
 \[H_1H_2H_3\cdot \cL_{\bdsF,(a)}\in \bfI_4.\]
 \end{itemize}\end{thmA}
  

In the literature, the three weight variable $p$-adic $L$-function for the triple product of Hida families in the balanced case has been extensively studied by Greenberg-Seveso \cite{GS16}, the first author \cite{Hsieh_triple} and so on. These works, based on Ichino's formula \cite{Ichino08Duke}, focus on the $p$-adic interpolation of central values and hence the cyclotomic variable is excluded. Our four-variable $p$-adic $L$-function $\cL_{\bdsF,(a)}$ specializes to this three variable $p$-adic $L$-function along the central critical line (see \remref{rem:11}). {To the best of our knowledge,} the first attempt to construct the cyclotomic $p$-adic triple product $L$-functions was made by B\"{o}cherer and Panchishkin \cite{BP06Triple, BP2}, where 
they constructed one-variable $p$-adic $L$-functions associated with three primitive elliptic newforms by using the pull-back of Siegel Eisenstein series on $\GSp(6)$ and Garrett's integral representation of triple product $L$-functions. Their construction is not restricted to the ordinary case, but the $p$-integrality of the $p$-adic $L$-function is not discussed; the interpolation formula is less complete, for example the interpolation at the trivial character is not covered. To obtain an explicit and complete interpolation formula, we introduce new $p$-adic sections and archimedean sections in the construction of Siegel Eisenstein series on $\GSp(6)$.  
\begin{Remark}\noindent We put a few words about our assumption. \begin{enumerate}\item The assumption \eqref{sf} on the conductors of modular forms and characters is imposed due to the complexity of choosing nice test vectors in local zeta integrals arising from Garrett's integral representation of triple product $L$-functions.
\item The $p$-adic $L$-function $\cL_{\bdsF,(a)}$ may have some poles outside {critical} points. Under a Gorensteiness hypothesis of Hecke algebras {(the Hypothesis (CR)  in \S 7.3)}, we can construct an optimal four-variable $p$-adic triple product $L$-function without denominators by multiplying generators of the congruence ideals of $\bdsf$, $\bdsg$, and $\bdsh$ (see \propref{P:interpolation}).\end{enumerate}
\end{Remark}
 \subsection{Application to the trivial zero conjecture}
 
Let $E_i$ be a $p$-ordinary semi-stable elliptic curve over the rationals $\Q$ for $i=1,2,3$. We write $L(\bdsE,s)$ for the degree eight motivic $L$-function for the triple product 
\beq
\bfV_{\bdsE}=\rmH^1_{\et}(E_{1/\Qbar},\Qp)\ot\rmH^1_{\et}(E_{2/\Qbar},\Qp)\ot\rmH^1_{\et}(E_{3/\Qbar},\Qp) \label{tag:8dimGalois}
\eeq
realized in the middle cohomology of the abelian variety $\bdsE=E_1\times E_2\times E_3$ by the K\"{u}nneth formula. 
Hence   
\[L(\rmH^3_{\et}(\bdsE_{/\Qbar},\Qp),s)=L(\bdsE,s)\prod_{i=1}^3L(E_i,s-1)^2. \]
 
Our four-variable $p$-adic $L$-function yields a cyclotomic $p$-adic $L$-function 
\[L_p(\bdsE)\in\Zp\powerseries{\Gal(\Q_\infty/\Q)}\ot\Qp,\]
which roughly interpolates the algebraic part of central values $\frac{L(\bdsE\ot\chi,2)}{\Omega}$ with a fixed period $\Ome$ for all finite order characters $\chi$ of $\Gal(\Q_\infty/\Q)$.  
Define an analytic function $L_p(\bdsE,s):=\cyc^{s-2}(L_p(\bdsE))$ for $s\in\Zp$ (See \propref{P:cycpadicL.7} for the precise statement). 
The Euler-like factor $\cale_p(\Fil^+\bfV_{\bdsE}(2))$ can possibly vanish. 
In this case the interpolation formula forces $L_p(\bdsE,2)$ to be zero. 
Such a zero is called a trivial zero. 
For example, a trivial zero appears if all $E_i$ have split multiplicative reduction at $p$ (see Remark \ref{R:82}). 
In this particular case, the trivial zero conjecture predicts that the leading coefficient of $L_p(\bdsE,s)$ is the product of the $\mathscr L$-invariants for $E_i$ and the algebraic part of the complex central value  $L(\bdsE,2)$ (\cf \cite[(25), p.~166]{Greenberg94Trivial} and \cite[p.~1579]{Benois11AJM}). We recall that if $E/\Q$ is an elliptic curve with split multiplicative reduction at $p$, denote by $\mathscr L_p(E)=\frac{\log_pq_{E}}{\ord_pq_{E}}$ the $\mathscr L$-invariant of $E$ with Tate's $p$-adic period $q_{E}$ attached to $E$.
We construct several improved $p$-adic triple product $L$-functions and apply the idea of Greenberg-Stevens \cite{GS93} and \cite{BDJ17arXiv} to establish the trivial zero conjecture for the triple product of elliptic curves. The following result is a special case of our more general result (see \thmref{T:trivial}).

\begin{thmA}Suppose that $E_1$ has split multiplicative reduction at $p$. 
\begin{enumerate}
\item If $E_2$ and $E_3$ are both split multiplicative at $p$, then $L_p(\bdsE,s)$ has at least a triple zero at $s=2$, and 
\[\lim_{s\to 2}\frac{L_p(\bdsE,s)}{(s-2)^3}=-p\prod_{i=1}^3\mathscr L_p(E_i)\cdot\frac{
L(\bdsE,2)}{\Omega}. \]
\item If $E_2$ and $E_3$ are both good ordinary at $p$ with $a_p(E_2)=a_p(E_3)$, where \[a_p(E_i)=1+p-\#E_i(\Fp),\] then $L_p(\bdsE,s)$ has at least a double zero at $s=2$ and 
\[\lim_{s\to 2}\frac{L_p(\bdsE,s)}{(s-2)^2}=(-p\al_2^{-2})(1-\al_2^{-2})^2\mathscr L_p(E_1)\cdot \frac{L(\bdsE,2)}{\Omega}, \]
where $\al_2$ is the unit root of the Hecke polynomial $X^2-a_p(E_2)X+p$ of $E_2$.
\end{enumerate}
\end{thmA}

In the case of the $p$-adic $L$-function $L_p(E,s)$ of an elliptic curve $E$ over $\Q$ the trivial zero arises if and only if $E$ is split multiplicative at $p$. 
An analogus formula for $L'_p(E,1)$ was experimentally discovered in \cite{MTT86} and proved in \cite{GS93}, and for Hilbert modular forms in \cite{Mok90Comp}, \cite{Spiess14Inv} and \cite{BDJ17arXiv}. Later Greenberg formulated a more general trivial zero conjecture for ordinary motives \cite{Greenberg94Trivial}, and Benois further extended this conjecture to semistable representations \cite{Benois11AJM}. The non-vanishing of $\mathscr L$-invarians $\sL(E_i)$ is known thanks to \cite{BDGP96Inv}. Our result thus proves the first cases of the Greenberg-Benois trivial zero conjecture where multiple trivial zeros are present and the Galois representation is not of $\GL(2)$-type.
{\begin{Remark}In this paper, we focus on the $p$-adic split triple product $L$-functions in the $p$-ordinary setting. In a forthcoming work, we manage to remove the assumption \eqref{sf} and further extend our construction to more general $p$-adic triple product $L$-functions, including the twisted triple product of modular forms of finite slopes. \end{Remark}}
\subsection{The construction of $\cL_{\bdsF,(a)}$}
We give a sketch of the construction of $\cL_{\bdsF,(a)}$. Our method is the combination of Garrett's integral representation of the triple product $L$-function, an integrality result of critical $L$-values for triple products in \cite{Mizumoto90} and Hida's $p$-adic Rankin-Selberg method. We begin with a {construction} of the four-variable $p$-adic family of the pull-back of Siegel Eisenstein series.  For each point $x=(\Qx,\Qy,\Qz,P)\in \frakX^\bal_{\bfI_4}$, we reorder the weights $\stt{k_\Qx,k_\Qy,k_\Qz}=\stt{k_x,l_x,m_x}$ so that $k_x\geq l_x\geq m_x$. 
For each $\nu_1,\nu_2\in \stt{0,1}$, we put 
\[\frakX^\bal_{(\nu_1,\nu_2)}=\stt{x\in \frakX^\bal_{\bfI_4}\mid k_x\con l_x+\nu_1\con m_x+\nu_2\pmod{2}}.\]
Hence we have the partition of the weight space
 \[\frakX^\bal_{\bfI_4}=\coprod_{\nu_1,\nu_2\in\stt{0,1}}\frakX^\bal_{(\nu_1,\nu_2)}.\] 
 
Let $N=\lcm(N_1,N_2,N_3)$. 
When $\chi$ is a Dirichlet character, we write $c(\chi)$ for the exponent of the $p$-part of its conductor.  
For each $x=(\Qx,\Qy,\Qz,P)\in \frakX^\bal_{(\nu_1,\nu_2)}$ we shall construct a nearly holomorphic Siegel Eisenstein series $\bfE_x^{(\nu_1,\nu_2)}(Z,s)$ of degree three whose restriction to the three-fold product  $\frkH_1^3$ of the upper half-plane $\frkH_1$ is a modular form of weight $(k_x,k_x-\nu_1,k_x-\nu_2)$ and level $\Gamma_1(Np^n)^3$, where $n=\max\{1,c(\eps_{Q_1}),c(\eps_{Q_2}),c(\eps_{Q_3}),c(\eps_P)\}$. 
We consider the pull-back given by 
\[G^{(\nu_1,\nu_2)}_x(z_1,z_2,z_3):=e_\Ord\Hol\biggl( \LDiff_{z_2}^\frac{k_x-l_x-\nu_1}{2}\LDiff_{z_3}^\frac{k_x-m_x-\nu_2}{2}\bfE_x^{(\nu_1,\nu_2)}\biggl(\diag{z_1,z_2,z_3},k_P-\frac{w_{\ul{Q}}+1}{2}\biggl)\biggl),\]
where $\LDiff_z:=-\frac{1}{2\pi\sqrt{-1}}(\Im z)^2\frac{\partial}{\partial\ol{z}}$ is the weight-lowering differential operator, $\Hol$ is the holomorphic projection and $e_\Ord$ is Hida's ordinary projector. 
Then we show that $G^{(\nu_1,\nu_2)}_x$ is a $p$-ordinary cusp form of weight $(k_x,l_x,m_x)$ on $\frkH_1^3$ the product of three copies of the upper half plane. 
\begin{Remark}
The idea of using the weight-lowering differential operator is inspired by \cite{Mizumoto90} and is different from the use of weight-raising differential operators in \cite{BP2}, where the authors applied differential operators of Ibukiyama-type to study the algebraicity of critical values of the triple product $L$-functions in the balanced case, and the $p$-adic interpolation is outlined without details.
\end{Remark}
Perhaps the most crucial and surprising point in our construction of $p$-adic families of Siegel Eisenstein series is that the four classes of Siegel Eisenstein series $\bfE_x^{(\nu_1,\nu_2)}$ can be constructed so that the pull-backs $G^{(\nu_1,\nu_2)}_x$ can be glued into a single four-variable Hida family of triple product modular form due to a miraculous effect of the ordinary projector on the Fourier expansions of the pull-back of our Siegel Eisenstein series (See the proof of \propref{P:1.E}). More precisely, let $\bfS^\Ord(N,\chi)$ denote the space of ordinary $\Lambda$-adic modular forms of tame level $N$ and character $\chi$.  
In the following we associate to {$a\in\Z/l\Z$} and $\ul{\chi}=(\chi_1,\chi_2,\chi_3)$ an explicit triple product ordinary $\Lambda$-adic form 
\[\cG_{\ul{\chi}}^{(a)}\in \bfS^\Ord(N,\chi_1,\Z_p\powerseries{X_1})\wh\ot_{\Z_p} \bfS^\Ord(N,\chi_2,\Z_p\powerseries{X_2})\wh\ot_{\Z_p}\bfS^\Ord(N,\chi_3,\Z_p\powerseries{X_3})\wh\otimes_{\Z_p}\Z_p\powerseries{\Gal(\Q_\infty/\Q)}. \] 
Let $T_3^+$ be the set of positive definite half-integral matrices of size $3$. 
The Siegel series attached to $B\in T_3^+$ and a rational prime $\ell$ is defined by
\[b_\ell(B,s)=\sum_{z\in\Sym_3(\Q_\ell)/\Sym_3(\Z_\ell)} \addchar(-\tr(Bz))\nu[z]^{-s},\]
where $\addchar$ is an arbitrarily fixed additive character on $\Q_\ell$ of order $0$ and 
$\nu[z]$ is the product of denominators of elementary divisors of $z$. 
There exists a polynomial $F_{B,\ell}(X)\in\Z[X]$ such that 
\[{b_\ell(B,s)}=(1-\ell^{-s})(1-\ell^{2-2s})F_{B,\ell}(\ell^{-s}). \]
{Let $z\mapsto[z]$ denote the inclusion of group-like elements $1+\bfp\Z_p\hookrightarrow\Z_p\powerseries{1+\bfp\Z_p}^\x$. 
Fix a topological generator $\bfu\in 1+\bfp\Zp$ and identify $\Z_p\powerseries{1+\bfp\Z_p}$ with $\Zp\powerseries{X}$, where $X=[\bfu]-1$. 
Define a character $\Dmd{\cdot}:\Z_p^\x\to 1+\bfp\Z_p$} by $\Dmd{x}=x\Om(x)^{-1}$ and write $\Dmd{x}_X=[\Dmd{x}]=(1+X)^{\log_p z/\log_p\bfu}\in\Z_p\powerseries{X}$. 
Let $\Xi_p$ be the set of symmetric matrices of size $3$ over {$\Q_p$ whose off-diagonal entries times 2 are $p$-units but whose diagonal entries belong to $p\Z_p$}. 
Now the seven-variable formal power series is presented by 
\begin{align*}
\cG_{\ul{\chi}}^{(a)}&=\sum_{B=(b_{ij})\in T_3^+\cap\Xi_p}\cQ_B^{(a)}(X_1,X_2,X_3,T)\cdot \cF_B^{(a)}(X_1,X_2,X_3,T)\cdot q_1^{b_{11}}q_2^{b_{22}}q^{b_{33}}_3, 
\end{align*}
where $\cQ_B^{(a)},\cF_B^{(a)}\in\Z_p\powerseries{X_1,X_2,X_3,T}$ are given by 
\begin{align*}
\cQ_B^{(a)}(X_1,X_2,X_3,T)&=\frac{\Om^a(8b_{23}b_{31}b_{12})\Dmd{8b_{23}b_{31}b_{12}}_T}{\chi_1(2b_{23})\chi_2(2b_{31})\chi_3(2b_{12})\Dmd{2b_{23}}_{X_1}\Dmd{2b_{31}}_{X_2}\Dmd{2b_{12}}_{X_3}}, \\
\cF_B^{(a)}(X_1,X_2,X_3,T)&=\prod_{\ell\ndivides pN}F_{B,\ell}(\Dmd{\ell}_T^{-2}(\Om^{-2a}\chi_1\chi_2\chi_3)(\ell)\Dmd{\ell}_{X_1}\Dmd{\ell}_{X_2}\Dmd{\ell}_{X_3}\ell^{-4}). 
\end{align*} 
By an explicit calculation of Fourier coefficients of $G_x^{(\nu_1,\nu_2)}$, we prove in \propref{P:GLam.E} that the specialization $\cG_{\ul{\chi}}^{(a)}(x)$ at every $x\in\frakX^\bal_{\bfI_4}$ is the $q$-expansion of the cusp form $G_x^{(\nu_1,\nu_2)}$. 

 

Now we apply the $p$-adic Rankin-Selberg method to define the $p$-adic $L$-function. Denote by $\bfT(N,\bfI)$ the $\bfI$-algebras  generated by Hecke operators on the space of ordinary $\Lam$-adic cusp forms of level $N$. For each $?\in\stt{\bdsf,\bdsg,\bdsh}$ we write $\bfone_?\in \bfT(N,\bfI)\otimes_\bfI\mathrm{Frac}\bfI$ for the idempotent corresponding to $?$.
We define \[\breve\cL_{\bdsF,(a)}:=\text{the first Fourier coefficient of }\bfone_\bdsf\ot\bfone_\bdsg\ot\bfone_\bdsh(\Tr_{N/N_1}\ot\Tr_{N/N_2}\ot\Tr_{N/N_3}(\cG^{(a)}_{\ul{\chi}})\in\bfI_3\powerseries{T},\] 
where $\Tr_{N/N_i}:\bfS^\ord(N,\chi_i,\bfI)\to \bfS^\ord(N_i,\chi_i,\bfI)$ is the usual trace map, and then 
the $p$-adic triple product $L$-function is defined to be $\cL_{\bdsF,(a)}=\breve \cL_{\bdsF,(a)}^{}\cdot\frkf_{\ul{\chi},a,N_1,N_2,N_3}^{-1}$, where $\frkf_{\ul{\chi},a,N_1,N_2,N_3}\in\bfI_4^\times$ is a fudge factor which is essentially a product of epsilon factors at prime-to-$p$ finite places. The $p$-adic Rankin-Selberg method tells us that the interpolation formula for the value $\breve\cL_{\bdsF,(a)}(x)$ at $x\in\frakX^\bal_{\bfI_4}$ is roughly given by 
\[\lim_{s\to k_P-\frac{w_{\ul{Q}}+1}{2}}{\frac{\pair{\bdsf_\Qx\ot\del_{l_x}^\frac{k_x-l_x-\nu_1}{2}\bdsg_\Qy\ot \del_{m_x}^\frac{k_x-m_x-\nu_2}{2}\bdsh_\Qz}{\bfE_x^{(\nu_1,\nu_2)}(s)}}{\norm{\bdsf_\Qx}^2\norm{\bdsg_\Qy}^2\norm{\bdsh_\Qz}^2}}\]
(\cf Lemma \ref{lem:Hidafunct}), where {$\del_k^m$ is the Maass-Shimura differential operator} and $\pairing$ is the Petersson pairing on $\frkH_1^3$ and $\norm{\cdot}$ is the Petersson norm on $\frkH_1$. 
The series $\bfE_x^{(\nu_1,\nu_2)}(Z,s)$ is constructed from a factorizable section of a certain family of induced representations. 
By means of the generalization of Garrett's work, carried out in \cite{PSR87,Ikeda89} (see Lemma \ref{lem:Garrett}) 
the pairing can be unfolded and written as a product of $L\bigl(s+\frac{1}{2},\pi_{\bdsf_\Qx}\times\pi_{\bdsg_\Qy}\times\pi_{\bdsh_\Qz}\otimes\eps_P\Om^{a-k_P}\bigl)$ and the normalized local zeta integrals at primes dividing $pN$. 
It turns out that these local zeta integrals are essentially given by the modified Euler factor $\cE_p(\Fil^+\bfV_{(\ulQ,P)})$ at $p$ and the local epsilon factors $\frkf_{\ul{\chi},a,N_1,N_2,N_3}$ at primes $\ell|N$.  
In both calculations the key ingredients are Lemma \ref{lem:11} and the local functional equations for $\GL_1$ and $\GL_2$, by which we can generalize Proposition 4.2 of \cite{GK92} without brute force calculations  (see Remark \ref{rem:21}). 

This paper is organized as follows. In the local part of this paper \S2, \S3 and \S4, we prepare the local ingredients in the construction of Siegel Eisenstein series $\bfE_x^{(\nu_1,\nu_2)}(Z,s)$ and carry out the explicit computations of {degenerate Whittaker functions, which appear in the Fourier coefficients of $\bfE_x^{(\nu_1,\nu_2)}(Z,s)$, } and local zeta integrals that appear in Garrett's integral representation of triple product $L$-functions. The non-archimdean case is treated in \S2 and \S3, and the archimdean case is carried out in \S4. The global part of this paper consists of \S5, \S6 and \S7. After recalling basic materials in Hida theory in \S5, we show that the Fourier expansion of $G^{(\nu_1,\nu_2)}_x$ can be $p$-adically interpolated by the power series $\cG_{\ul{\chi}}^{(a)}$ in \S6. We remark that the most crucial ingredient in this section is \propref{P:1.E} about the computation of Fourier coefficients of $G^{(\nu_1,\nu_2)}_x$. In \S7, we put together the local computations in \S2, 3, and 4 and prove the main interpolation formulae in \thmref{P:interpolation}. Finally, in \S8 we construct several improved $p$-adic $L$-functions in Lemmas \ref{L:01.t} and \ref{L:02} by modifying the construction of four variable and three variable $p$-adic $L$-functions in \S6 and \cite{Hsieh_triple} respectively, and prove \thmref{T:trivial} the trivial zero conjecture for the triple product of elliptic curves in \S8.4 and \S8.5 by the Greenberg-Stevens method.


\subsection*{Notation}

The following notations will be used frequently throughout the paper. 
For an associative ring $R$ with identity element, we denote by $R^\times$ the group of all its invertible elements, and by $\Mat_{m,n}(R)$ the module of all $m\times n$ matrices with entries in $R$. 
Put $\Mat_n(R)=\Mat_{n,n}(R)$ and $\GL_n(R)=\Mat_n(R)^\times$ particularly when we view the set as a ring. 
The identity and zero elements of the ring $\Mat_n(R)$ are denoted by $\ono_n$ and $\oo_n$ (when $n$ needs to be stressed) respectively. 
The transpose of a matrix $x$ is denoted by $x^{\rm t}$. 
Let $\Sym_n(R)=\{z\in \Mat_n(R)\;|\;z^{\rm t}=z\}$ be the space of symmetric matrices of size $n$ over $R$. 
For any set $X$ we denote by $\II_X$ the characteristic function of $X$. 
When $X$ is a finite set, we denote by $\sharp X$ the number of elements in $X$. 
When $X$ is a totally disconnected locally compact topological space or a smooth real manifold, we write $\cals(X)$ for the space of Schwartz-Bruhat functions on $X$.
If $x$ is a real number, then we put $\lceil x\rceil=\max\{i\in\Z\;|\;i\leq x\}$. 

If $R$ is a commutative ring and $G=\GL_2(R)$, we denote by $\rho$ the right translation of $G$ on the space of $\C$-valued functions on $G$. 
Thus $(\rho(g)f)(g')=f(g'g)$. 
We write $\bfone:G\to\C$ for the constant function $\bfone(g)=1$.  For a function $f:G\to\C$ and a character $\chi:R^\x\to\C^\x$, let $f\ot\chi:G\to\C$ denote the function $f\ot\chi(g)=f(g)\chi(\det g)$.

\subsection*{Measures}\label{SS:l_measure}
{Let $F$ be either a non-archimedean local field or $F=\R$. We shall fix the normalization of the Haar measures on some groups over $F$ through the paper. If $F$ is non-archimedean with $\frko$ the ring of integers of $F$,  we normalize the Haar measures on $F$ and $F^\times$ so that $\mathrm{vol}(\frko,\d x)=\mathrm{vol}(\frko^\times,\d^\times a)=1$. If $F=\R$, then $\d x$ denotes the usual Lebesgue measure on $\bfR$ and $\d^\times x=\frac{\d x}{\abs{x}}$. Define the compact subgroups $\bfK$ of $\GL_2(F)$ and $\bfK'$ of $\SL_2(F)$ as follows: If $F$ is non-archimedean, put
\[\bfK =\GL_2(\frko), \quad \bfK'=\SL_2(\frko). \]
If $F=\R$, put 
\[
\bfK=\O(2,\bfR), \quad
\bfK'=\SO(2,\bfR).
\]
Let $\d k$ and $\d k'$ be the Haar measures on $\bfK$ and $\bfK'$ which have total volume $1$. 
Define the Haar measures on $\PGL_2(F)$ and $\SL_2(F)$ by $\d g=\d x\frac{\d^\x y}{|y|}\d k$ and $\d g'=\d x\frac{\d^\x y}{|y|^2}\d k'$ for $g=\begin{pmatrix} y & x \\ 0 & 1 \end{pmatrix}k$ and $g'=\begin{pmatrix} y & xy^{-1} \\ 0 & y^{-1} \end{pmatrix}k'$ with $y\in F^\x$, $x\in F$, $k\in\bfK$ and $k\in\bfK'$.}

\section{Computation of the local zeta integral: the $p$-adic case}


\subsection{The local zeta integral}

Let $T_n$ be the subgroup of diagonal matrices in $\GL_n$, $U_n$ the subgroup of upper triangular unipotent matrices in $\GL_n$, $Z_n$ the subgroup of scalar matrices in $\GL_n$ and $B_n=T_nU_n$ the standard Borel subgroup of $\GL_n$.  
The symplectic similitude group of degree $n$ is defined by 
\begin{align*}
\GSp_{2n}&=\{g\in\GL_{2n}\;|\;gJ_ng^{\rm t}=\nu_n(g)J_n,\;\nu_n(g)\in\GL_1\}, & 
J_n&=\begin{pmatrix} 0 & -\ono_n \\ \ono_n & 0\end{pmatrix}.
\end{align*}

We define the homomorphisms 
\begin{align*}
\bfm&:\GL_n\times\GL_1\to \GSp_{2n}, & 
\bfn,\;\bfn^-&:\Sym_n\to \GSp_{2n}
\end{align*}
by 
\begin{align*}
\bfm(A,\nu)&=\begin{pmatrix} A & 0 \\ 0 & \nu(A^{\rm t})^{-1}\end{pmatrix}, & 
\bfn(z)&=\begin{pmatrix} \ono_n & z \\ 0 & \ono_n\end{pmatrix}, & 
\bfn^-(z)&=\begin{pmatrix} \ono_n & 0 \\ z & \ono_n\end{pmatrix}. 
\end{align*}
We write
\begin{align*}
\bfm(A)&=\bfm(A,1), & 
\bfd(\nu)&=\bfm(\ono_n,\nu). 
\end{align*}
A maximal parabolic subgroup $\calp_n=\calm_nN_n$ of $\GSp_{2n}$ is defined by 
\begin{align*}
\calm_n&=\bfm(\GL_n\times\GL_1), & 
N_n&=\bfn(\Sym_n). 
\end{align*}

Define algebraic groups of $U^0\subset U\subset H$ by 
\begin{align*}
H&=\{(g_1,g_2,g_3)\in (\GL_2)^3\;|\;\det g_1=\det g_2=\det g_3\}, \\
U&=\{(\bfn(x_1),\bfn(x_2),\bfn(x_3))\;|\;x_1,x_2,x_3\in\Mat_1\}, \\
U^0&=\{(\bfn(x_1),\bfn(x_2),\bfn(x_3))\;|\;x_1+x_2+x_3=0\}.  
\end{align*}
We define the embedding $\iota:H\hookrightarrow\GSp_6$ by 
\[\iota\left(\begin{pmatrix} a_1 & b_1 \\ c_1 & d_1 \end{pmatrix},\begin{pmatrix} a_2 & b_2 \\ c_2 & d_2 \end{pmatrix},\begin{pmatrix} a_3 & b_3 \\ c_3 & d_3 \end{pmatrix} \right)=\left(\begin{array}{ccc|ccc} 
a_1 & & & b_1 & & \\
& a_2 & & & b_2 & \\
& & a_3 & & & b_3 \\ \hline 
c_1 & & & d_1 & & \\
& c_2 & & & d_2 & \\
& & c_3 & & & d_3 \end{array}\right). \]
We identity $Z=Z_6$ with the center of $\GSp_6$. 
{Once and for all we choose a representative $\Eta$ for the open orbit of $H$ in $\calp_3\bsl\GSp_6$ (\cf \cite[Lemma 1.1]{Ikeda89} and Lemma \ref{lem:11} below). 
The dependance on the choice of $\Eta$ is explicated in \cite[\S\S 4.1--4.2]{Chen21}. 
Let 
\[\Eta=\left(\begin{array}{ccc|ccc} 
 0 & 0 & 0 & -1 & 0 & 0 \\ 0 & 1 & 0 & 0 & 0 & 0 \\ 0 & 0 & 1 & 0 & 0 & 0 \\ \hline 1 & 1 & 1 & 0 & 0 & 0 \\ 0 & 0 & 0 & -1 & 1 & 0 \\ 0 & 0 & 0 & -1 & 0 & 1  \end{array}\right). \]
 This particular choice is made in \cite[p.~206]{GS93}, \cite[p.~293]{Ichino08Duke} and \cite[p.~762]{Chen21}.}

Let $F$ be a local field of characteristic zero. 
In the nonarchimedean case, $F$ contains a ring $\frko$ of integers having a single prime ideal $\frkp$, and the absolute value $\Abs_F=|\!\cdot\!|$ on $F$ is normalized via $|\vpi|=q^{-1}$ for any generator $\vpi$ of $\frkp$, where $q$ denotes the order of the residue field $\frko/\frkp$. 
Fix an additive character $\addchar$ on $F$ which is trivial on $\frko$ but non-trivial on $\frkp^{-1}$.  
When $F=\R$, we define $\addchar(x)=e^{2\pi\sqrt{-1}x}$ for $x\in\R$. 

We remind the readers that in this and the next sections $\chi$ and $\hat\omega$ stand for quasi-characters of $F^\times$, but on the other hand, in the global setting \secref{ssec:61} $\chi$ and $\hat\omega$ denote Dirichlet characters of $p$-power conductor, which we sometimes view as finite order characters of $\Z_p^\times$ (see \subsecref{SS:5.1} for our convention in the global setting).

Let $K$ be a standard maximal compact subgroup of $\GSp_6(F)$. For quasi-characters $\hat\ome,\chi:F^\times\to\C^\times$ we let $I_3(\hat\ome,\chi):=\Ind_{\calp_3}^{\GSp_6(F)}\chi^2\hat\ome\boxtimes\chi^{-3}\hat\ome^{-1}$ be the space of all right $K$-finite functions $f$ on $\GSp_6(F)$ which satisfy 
\[f(\bfm(A,\lam)\bfn(z)g)=\hat\ome(\lam^{-2}\det A)\chi(\lam^{-3}(\det A)^2)|\lam^{-3}(\det A)^2|f(g)\]
for $A\in\GL_3(F)$, $\lam\in F^\times$, $z\in\Sym_3(F)$ and $g\in\GSp_6(F)$. 
The group $\GSp_6(F)$ acts on $I_3(\hat\ome,\chi)$ by right translation $\rho_3$.  
It is important to note that for $t=\diag{a,d}\in T_2$ 
\beq
f(\Eta\iota(tg_1,tg_2,tg_3))=\hat\ome(d)^{-1}\chi(ad^{-1})|ad^{-1}|f(\Eta\iota(g_1,g_2,g_3)) \label{tag:12}
\eeq
{(\cf \cite[(4.9.11)]{GH93AJM}).} It is well worthy of notice that 
\beq
I_3(\hat\ome,\chi)\otimes\mu\circ\nu_3=I_3(\hat\ome\mu^{-2},\chi\mu). \label{tag:twist}
\eeq
We call a $K$-finite function $(s,g)\mapsto f_s(g)$ on $\C\x\GSp_6(F)$ a holomorphic section of $I_3(\hat\ome,\chi\Abs_F^s)$ if $f_s(g)$ is holomorphic in $s$ for each $g\in \GSp_6(F)$ and $f_s\in I_3(\hat\ome,\chi\Abs_F^s)$ for each $s \in\C$. We associate to a non-degenerate symmetric matrix $B$ of size $3$ and $f_s\in I_3(\hat\ome,\chi\Abs_F^s)$ the degenerate Whittaker function
\beq\label{E:DegW.2}\calw_B(g,f_s)=\int_{\Sym_3(F)}f_s(J_3\bfn(z)g)\addchar(-\tr(Bz))\,\d z\quad (g\in \GSp_6(F)),\eeq
where $\d z$ is the Haar measure on $\Sym_n(F)$ self-dual with respect to the pairing $\addchar(\tr(zw))$. The integral converges if $\Re s$ is sufficiently large and can be continued to an entire function. By definition, it is clear that 
\[\calw_B(\bfn(z)g,f_s)=\addchar(\tr(Bz))\calw_B(g,f_s).\]
This yields the degenerate Whittaker functional given by 
\begin{align*}
\calw_B&:I_3(\hat\ome,\chi\Abs_F^s)\to\C, & 
\calw_B(f_s)&:=\calw_B(\bfone_3,f_s).
\end{align*}

The study of the triple product $L$-function began with Garrett in \cite{Garrett}. 
Piatetski-Shapiro and Rallis modified this construction to give a twisted triple product $L$-function associated with a cuspidal automorphic form on $\GL(2)$ over a cubic algebra in \cite{PSR87}. 
Ikeda precisely described the poles of this $L$-function in \cite{Ikeda2}. 
When it is associated with holomorphic cusp forms, the algebraicity of its special values has been deeply investigated in {\cite{Orl87,GH93AJM,GK92,BP2}} and so on. 
We will focus on the $p$-adic aspects in this section. 

Given an irreducible admissible infinite dimensional representation $\pi$ of $\GL_2(F)$, we denote by $\scrw(\pi)$ the Whittaker model of $\pi$ with respect to $\addchar$. 
Let $\pi_1,\pi_2,\pi_3$ be a triplet of irreducible admissible infinite dimensional representations of $\GL_2(F)$. 
We denote the central character of $\pi_i$ by $\ome_i$. 
Set $\hat\ome=\ome_1\ome_2\ome_3$.  
We associate to a holomorphic section $f_s$ of $I(\hat\ome,\chi\Abs_F^s)$ and Whittaker functions $W_i\in\scrw(\pi_i)$ the local zeta integral of {Garrett, Piateski-Shapiro and Rallis (\cf\cite[(3-1), p.48]{PSR87})}
\beq\label{E:GPSR}Z(W_1,W_2,W_3,f_s)=\int_{U^0Z\bsl H}W_1(g_1)W_2(g_2)W_3(g_3)f_s(\Eta\iota(g_1,g_2,g_3))\,\d g_1^{}\d g_2'\d g_3',\eeq
{where the Haar measure $\d g_1^{}\d g_2'\d g_3'$ is obtained via $Z\bsl H\simeq\PGL_2(F)\times\SL_2(F)\times\SL_2(F)$. It is proved in \cite{PSR87,Ikeda89} that $Z(W_1,W_2,W_3,f_s)$ converges absolutely for $\Re s\gg 0$ and admits meromorphic continuation to $s\in\C$ and the functional equation. 
Let $\sig_i$ be the representation of the Weil-Deligne group $W_F$ of $F$ associated to $\pi_i$ by the local Langlands correspondence. 
Then $L\bigl(s+\frac{1}{2},\pi_1\times\pi_2\times\pi_3\ot\chi\bigl):=L\bigl(s+\frac{1}{2},\sig_1\otimes\sig_2\otimes\sig_3\otimes\chi\bigl)$ is precisely the GCD for the family of $Z(W_1,W_2,W_3,f_s)$ associated to `good' sections $f_s$, provided that $\pi_1,\pi_2,\pi_3$ arise as local components of cuspidal automorphic representations (see \cite{Ikeda99,Dinakar00}). }

{Put $T=\stt{(t,t,t)\in H\mid t\in T_2}$. }We define a map $\iota_0:H\hookrightarrow\GSp_6$ by 
{\beq
\textcolor{black}{\iota_0(g_1,g_2,g_3)=\Eta\iota(g_1,g_2J_1,g_3J_1).} \label{tag:iota}
\eeq}
As a preliminary step, we choose a coordinate system on an open dense subset of $U^0{T}\bsl H$ {(\cf (\ref{tag:12}))}. 

\begin{lm}\label{lem:11}
If $(x_1,u_1,u_2,u_3,a_2,a_3)\in F^4\oplus F^{\times 2}$, then 
\[\iota_0(\bfn^-(u_1)\bfn(x_1),\bfm(a_2)\bfn^-(u_2),\bfm(a_3)\bfn^-(u_3))=\begin{pmatrix} A & B \\ \oo_3 & (A^{\rm t})^{-1}\end{pmatrix}J_3\bfn(-z), \]
where 
\begin{align*}
A&=\begin{pmatrix}
1 & a_2u_1 & a_3u_1 \\ 
0 & a_2 & 0 \\ 
0 & 0 & a_3 
\end{pmatrix}, \quad\quad\quad
B=\begin{pmatrix} 
-u_1 & 0 & 0 \\ 0 & 0 & 0 \\ 0 & 0 & 0 
\end{pmatrix}, &
z&=\begin{pmatrix}
-x_1 & a_2 & a_3 \\
a_2 & u_2+a^2_2u_1 & a_2a_3u_1 \\
a_3 & a_2a_3u_1 &  u_3+a_3^2u_1
\end{pmatrix}. 
\end{align*}
\end{lm}

\begin{proof}
We can prove Lemma \ref{lem:11} by the matrix expression of $\iota_0$. 
\end{proof}


\subsection{The unramified case}\label{ssec:unram}
{From now on we assume that $F$ is nonarchimedean until the end of \secref{sec:3}. Denote by $\zet(s)=(1-q^{-s})^{-1}$ the local zeta function for $F$.} 
When $\pi_i$ is unramified, we write $W_i^0\in\scrw(\pi_i)$ for the unique Whittaker function which takes the value $1$ on $\GL_2(\frko)$. 
Assume that $\hat\ome$ and $\chi$ are unramified. 
Then we define the holomorphic section $f^0_s(\chi)$ of $I_3(\hat\ome,\chi\Abs_F^s)$ by the condition that $f^0_s(k,\chi)=1$ for $k\in\GSp_6(\frko)$. 
Garrett has proved that 
\[Z(W^0_1,W^0_2,W^0_3,f_s^0(\chi))=\frac{L\left(s+\frac{1}{2},\pi_1\times\pi_2\times\pi_3\otimes\chi\right)}{L(2s+2,\chi^2\hat\ome)L(4s+2,\chi^4\hat\ome^2)}. \]
{The reader who has interest in the proof of this formula can consult \cite[Theorem 3.1]{PSR87}, which is better suited for the local calculation. }

We associate to a half-integral symmetric matrix $B$ the series defined by
\[b(B,s)=\sum_{z\in\Sym_3(F)/\Sym_3(\frko)} \addchar(-\tr(Bz))\nu[z]^{-s},\]
where $\addchar$ is an arbitrarily fixed additive character on $F$ of order $0$ and $\nu[z]=[z\frko^3+\frko^3:\frko^3]$. 
If $\det B\neq 0$, then there exists a polynomial $F_B(X)\in\Z[X]$ such that 
\beq\label{E:Siegel}b(B,s)=(1-q^{-s})(1-q^{2-2s})F_B(q^{-s}). \eeq
{The unramified degenerate Whittaker function is a representation theoretic interpretation of the Siegel series.  When $\chi=\hat\ome=1$, we have $\cW_B(f^0_s(1))=b(B,2s+2)$ by \cite[Proposition 19.2, page 158]{Shimura97Euler-P-Eisenstein.}. Therefore}
\beq
\cW_B(f^0_s(\chi))=\frac{F_B(\chi^2\hat\om(\vpi)q^{-2s-2})}{L(2s+2,\chi^2\hat\om)L(4s+2,\chi^4\hat\om^2)}. \label{tag:unramWhittaker}
\eeq


\subsection{The $\frkp$-adic sections}\label{SS:padic}

Let $\St$ stand for the Steinberg representation of $\GL_2(F)$. 
For quasi-characters $\mu,\nu$ of $F^\times$ the representation $I(\mu,\nu)$ is realized on the space of functions $f:\GL_2(F)\to\C$ which satisfy 
\[f\left(\begin{pmatrix} a & b \\ 0 & d \end{pmatrix}g\right)=\mu(a)\nu(d)\left|\frac{a}{d}\right|^{1/2}f(g)\]
for $a,d\in F^\times$, $b\in F$ and $g\in\GL_2(F)$, where $\GL_2(F)$ acts by right translation $\rho$.  
Hereafter we assume that $\pi_i$ are not supercuspidal and are infinite dimensional. 
Then $\pi_i$ is a quotient of a principal series representation $I(\mu_i,\nu_i)$ with quasi-characters $\mu_i,\nu_i$. 
If $\mu_i^{}\nu_i^{-1}\neq\Abs_F^{-1}$, then $\pi_i\simeq I(\mu_i,\nu_i)$. 
If $\mu_i^{}\nu_i^{-1}=\Abs_F^{-1}$, then $\pi_i\simeq \St\otimes\mu_i\Abs_F^{1/2}$.  Given a character $\mu$ of $\frko^\times$, we define $\vph_\mu\in\cals(F)$ by 
\[\vph_\mu(x)=\mu(x)\II_{\frko^\times}(x). \]
We write $c(\mu)$ for the smallest integer $n$ such that $\mu$ is trivial on $\frko^\times\cap(1+\frkp^n)$. 
Define the open compact subgroup $\calk_0^{(m)}(\frkp^n)$ of $\GSp_{2m}(F)$ by 
\begin{align*}
K_0^{(m)}(\frkp^n)
&=\biggl\{\begin{pmatrix} a & b \\ c & d \end{pmatrix}\in\GSp_{2m}(\frko)\;\biggl|\;c\in\Mat_m(\frkp^n)\biggl\}. 
\end{align*}
Provided that $n\geq c(\mu)$, we can define characters $\mu^{\uparrow}$ and $\mu^{\downarrow}$ of $K_0^{(1)}(\frkp^n)$ by 
\begin{align}
\mu^\uparrow\left(\begin{pmatrix} a & b \\ c & d \end{pmatrix}\right)&=\mu(a), &
\mu^{\downarrow}\left(\begin{pmatrix} a & b \\ c & d \end{pmatrix}\right)&=\mu(d). \label{tag:updown}
\end{align}
 
{Our construction of $p$-adic $L$-functions relies on an explicit construction of $p$-adic families of Siegel Eisenstein series of degree three, and hence the main local problem is a choice of special sections in $I_3(\hat\ome,\chi)$. We shall consider the following sections supported in the open cell. These sections have appeared in the study of poles of an intertwining operator (\cite[Lemma 4.1]{PSR87}) and was used in the construction of $p$-adic families of Siegel Eisenstein series on unitary groups (\cite[(3.3.4)]{HLS06Doc}).  \begin{defn}[sections supported in the open cell]\label{def:cell}
Since $J_3^{}N_3^{}J_3^{-1}=\bfn^-(\Sym_3(F))$ is the unipotent radical of the parabolic subgroup opposite to $\calp_3$, if $p,p'\in\calp_3$, $u,u'\in N_3$ and $pJ_3u=p'J_3u'$, then $p=p'$ and $u=u'$. 
Given a Schwartz function $\varPhi\in\cals(\Sym_3(F))$, we can therefore define a function $f_\varPhi(\chi):\calp_3J_3N_3\to\C$ by 
\[f_\varPhi(\bfm(A,\lam)uJ_3\bfn(z),\chi)=\hat\ome(\lam^{-2}\det A)\chi(\lam^{-3}(\det A)^2)|\lam^{-3}(\det A)^2|\varPhi(z)\]
for $A\in\GL_3(F)$, $\lam\in F^\times$, $u\in N_3$ and $z\in\Sym_3(F)$. 
Since $\calp_3J_3N_3=\calp_3J_3\calp_3$ is the cell of the longest Weyl element $J_3$ in the Bruhat decomposition of $\GSp_6(F)$, we can extend $f_\varPhi(\chi)$ by zero to a function on $\GSp_6(F)$, which is an element of $I_3(\hat\ome,\chi)$ (\cf \cite[(B), p. 138]{Cartier79Corvallis}). 
\end{defn}
\begin{Remark}\label{rem:20}
The family $\{f_\varPhi(\chi)\}$ is stable under the action of $\calp_3$ by definition. 
The group $\GL_3(F)\times F^\times$ acts on the space $\cals(\Sym_3(F))$ of functions on symmetric matrices of size 3 by 
\[(r(A,\lam)\varPhi)(z)=\varPhi(\lam A^{-1}z(A^{\rm t})^{-1}). \]
Since 
\[\rho_3(\bfm(A,\lam))f_\varPhi(J_3\bfn(z),\chi)
=f_\varPhi\left(\begin{pmatrix} \lam (A^{\rm t})^{-1} & \\ & A\end{pmatrix}J_3\bfn(\lam A^{-1}z(A^{\rm t})^{-1}),\chi\right) \]
for $z\in\Sym_3(F)$, we see that for $A\in\GL_3(F)$ and $\lam\in F^\times$  
\[\rho_3(\bfm(A,\lam))f_\varPhi(\chi)=\hat\ome(\lam(\det A)^{-1})(\chi\Abs_F)(\lam^3(\det A)^{-2})f_{r(A,\lam)\varPhi}(\chi). \] 
\end{Remark}}

The element $f_\varPhi(\chi)$ has a simple form of degenerate Whittaker coefficients (see (\ref{tag:Fourier}) below), and
\begin{multline} 
f_\varPhi(\iota_0(\bfn^-(u_1)\bfn(x_1),\bfm(a_2)\bfn^-(u_2),\bfm(a_3)\bfn^-(u_3)),\chi)\\
=\hat\ome(a_2a_3)\chi(a_2a_3)^2|a_2a_3|^2\varPhi\left(\begin{pmatrix}
x_1 & -a_2 & -a_3 \\
-a_2 & -u_2-a_2^2u_1 & -a_2a_3u_1 \\
-a_3 & -a_2a_3u_1 &  -u_3-a_3^2u_1
\end{pmatrix}\right) \label{tag:13}
\end{multline}
by \lmref{lem:11}. 
{We may assume that $\vPh\in\cals(\Sym_3(F))$ is of the form
\beq\label{E:Phi}
\vPh\left(\begin{pmatrix}
u_1 & x_3 & x_2 \\
x_3 & u_2 & x_1 \\
x_2 & x_1 & u_3
\end{pmatrix}\right)=\prod_{i=1}^3\phi_i(u_i)\vph_i(x_i) 
\eeq
with $\phi_1,\phi_2,\phi_3,\vph_1,\vph_2,\vph_3\in\cals(F)$. }
{We define a special \BS function $\Phi\in\cals(\Sym_3(F))$ by letting  
\begin{align*}
\phi_1&=\phi_2=\phi_3=\widehat{\II_\frkp}, &
\vph_1&=\widehat{\vph_{\chi\mu_1\nu_2\nu_3}}, & 
\vph_2&=\widehat{\vph_{\chi\nu_1\mu_2\nu_3}}, & 
\vph_3&=\widehat{\vph_{\chi\nu_1\nu_2\mu_3}}. 
\end{align*}
In what follows, we shall show} that $f_\Phi(\chi)$ has an appropriate $K$-type (see Lemma \ref{lem:13} below) and a suitable formula of the local zeta integral (see Proposition \ref{prop:11} below). 
Here we define the Fourier transform of a \BS function $\varPhi$ on the space of symmetric matrices of size $m$ over $F$ with respect to $\addchar$ by 
\[\widehat{\varPhi}(w)=\int_{\Sym_m(F)}\varPhi(z)\addchar(\tr(zw))\,\d z,\]
where we recall that $\d z$ is the self-dual Haar measure on $\Sym_m(F)$. 
{
\begin{remark}
The choice of $\vph_1,\vph_2,\vph_3$ should be related to the twisting operator defined in (1.12) of \cite{BP06Triple}. 
\end{remark}}
The choices of these functions are ad-hoc and the proof is mainly computational, so the reader is advised to skip it on a first reading. 

\begin{lm}\label{lem:13}
If $n\geq\max\{1, c(\chi), c(\mu_i), c(\nu_i)\;|\;i=1,2,3\}$, then 
\begin{align*}
\rho_3(\iota(g_1,g_2,g_3))f_\Phi(\chi)&=f_\Phi(\chi)\prod_{i=1}^3\mu_i^{\uparrow}(g_i)^{-1}\nu_i^{\downarrow}(g_i)^{-1} 
\end{align*}
for $g_1,g_2,g_3\in K_0^{(1)}(\frkp^{2n})$ with $\det g_1=\det g_2=\det g_3$.
\end{lm}

\begin{proof}
One can easily check that 
\begin{align}
\widehat{\vph_\mu}&\in\cals(\frkp^{-n}), &
\widehat{\vph_\mu}(ax)&=\mu(a)^{-1}\widehat{\vph_\mu}(x), &
\widehat{\vph_\mu}(x+b)&=\widehat{\vph_\mu}(x) \label{tag:14}
\end{align} 
for $a\in\frko^\times$, $b\in\frko$ and $x\in F$. 
Simply because $\phi_i=\II_{\frkp^{-1}}$, we see that $\Phi(z+c)=\Phi(z)$ for $c\in\Sym_3(\frko)$, which means that $f_\Phi(\chi)$ is fixed by the action of $\bfn(\Sym_3(\frko))$. 
Put 
\begin{align*}
\chi_1&=\chi\mu_1\nu_2\nu_3, & 
\chi_2&=\chi\nu_1\mu_2\nu_3, & 
\chi_3&=\chi\nu_1\nu_2\mu_3. 
\end{align*}
{If $a_i,d_i\in\frko^\times$ and $\lam=a_1d_1=a_2d_2=a_3d_3$, then 
\[\rho_3(\iota(\diag{a_1,d_1},\diag{a_2,d_2},\diag{a_3,d_3}))f_\Phi(\chi)=\hat\ome(\lam(a_1a_2a_3)^{-1})\chi(\lam^3(a_1a_2a_3)^{-2})f_{r(A,\lam)\Phi}(\chi) \]
by Remark \ref{rem:20}, where we put $A=\diag{a_1,a_2,a_3}$. 
Since 
\[(r(A,\lam)\Phi)\left(\begin{pmatrix}
u_1 & x_3 & x_2 \\
x_3 & u_2 & x_1 \\
x_2 & x_1 & u_3
\end{pmatrix}\right)=\Phi
\left(\begin{pmatrix}
\frac{\lam u_1}{a_1^2} & \frac{\lam x_3}{a_1a_2} & \frac{\lam x_2}{a_1a_3} \\
\frac{\lam x_3}{a_1a_2} & \frac{\lam u_2}{a_2^2} & \frac{\lam x_1}{a_2a_3} \\
\frac{\lam x_2}{a_1a_3} & \frac{\lam x_1}{a_2a_3} & \frac{\lam u_3}{a_3^2}
\end{pmatrix}\right), \]
we have 
\[r(A,\lam)\Phi=\chi_1(\lam^{-1}a_2a_3)\chi_2(\lam^{-1}a_1a_3)\chi_3(\lam^{-1}a_1a_2)\Phi \]
by (\ref{tag:14}). 
Observe that 
\begin{align*}
&\chi_1(\lam^{-1}a_2a_3)\chi_2(\lam^{-1}a_1a_3)\chi_3(\lam^{-1}a_1a_2)\\
=&\chi(\lam^{-3}(a_1a_2a_3)^2)\cdot(\mu_1\nu_2\nu_3)(\lam^{-1}a_2a_3)(\nu_1\mu_2\nu_3)(\lam^{-1}a_1a_3)(\nu_1\nu_2\mu_3)(\lam^{-1}a_1a_2)\\
=&\chi(\lam^{-3}(a_1a_2a_3)^2)\cdot(\hat\ome\nu_1\nu_2\nu_3)(\lam^{-1}a_1a_2a_3)\cdot(\mu_1\nu_2\nu_3)(a_1^{-1})(\nu_1\mu_2\nu_3)(a_2^{-1})(\nu_1\nu_2\mu_3)(a_3^{-1})\\=&\chi(\lam^{-3}(a_1a_2a_3)^2)\cdot\hat\ome(\lam^{-1}a_1a_2a_3)\cdot\nu_1(d_1^{-1})\nu_2(d_2^{-1})\nu_3(d_3^{-1})\mu_1(a_1^{-1})\mu_2(a_2^{-1})\mu_3(a_3^{-1}), 
\end{align*}
from which we conclude that 
\[\rho_3(\iota(\diag{a_1,d_1},\diag{a_2,d_2},\diag{a_3,d_3}))f_\Phi(\chi)
=f_\Phi(\chi)\prod_{i=1}^3\mu_i(a_i)^{-1}\nu_i(d_i)^{-1}.  \]}

Let $w\in\Sym_3(\frkp^{2n})$. 
If $f_\Phi(g\bfn^-(w),\chi)\neq 0$, then since 
$g\bfn^-(w)\in \calp_3J_3\bfn(z)$ with $z\in\Sym_3(\frkp^{-n})$ and since 
\[\bfn(z)\bfn^-(-w)
=\begin{pmatrix} \ono_3-zw & \oo_3 \\ -w & (\ono_3-wz)^{-1} \end{pmatrix}\bfn((\ono_3-zw)^{-1}z), \] 
we have $g\in \calp_3J_3\bfn(\Sym_3(\frkp^{-n}))$. 
We see by the identity above that 
\[f_\Phi(J_3\bfn(z)\bfn^-(w),\chi)=f_\Phi(J_3\bfn((\ono_3+zw)^{-1}z),\chi)=f_\Phi(J_3\bfn(z),\chi)\]
for $z\in\Sym_3(\frkp^{-n})$ and $w\in\Sym_3(\frkp^{2n})$. 
We conclude that $f_\Phi(\chi)$ is fixed by right translation by $\bfn^-(\Sym_3(\frkp^{2n}))$. 
The proof is complete by $K_0^{(1)}(\frkp^m)=\bfn(\frko)\bfd(\frko^\times)\bfm(\frko^\times)\bfn^-(\frkp^m)$.  
\end{proof}


\subsection{Degenerate Whittaker functions at $\frkp$}

Let $\Xi_\frkp$ be a subset of $\Sym_3(F)$ which consists of symmetric matrices whose the diagonal entries belong to $\frkp$ and whose off-diagonal entries belong to $\frac{1}{2}\frko^\times$.

\begin{prop}\label{prop:12}
Let $B=(b_{ij})\in \Sym_3(F)$. 
Put $y_i=b_{jk}$ whenever $\{i,j,k\}=\{1,2,3\}$. 
Then 
\[\calw_B(f_\Phi(\chi))=\chi(8y_1y_2y_3)\prod_{i=1}^3\mu_i(2y_i)\II_{\frko^\times}(2y_i)\II_\frkp(b_{ii})\prod_{j\in\{1,2,3\}\setminus\{i\}}\nu_j(2y_i). \]
In particular, $\calw_B(f_\Phi(\chi))\neq 0$ if and only if $B\in\Xi_\frkp$. 
\end{prop}

\begin{proof}
Observe that
\begin{align}
\calw_B(f_\varPhi(\chi))
=&\int_{\Sym_3(F)}f_\varPhi(J_3\bfn(z),\chi)\addchar(-\tr(Bz))\,\d z 
=\widehat{\varPhi}(-B) \label{tag:Fourier}
\end{align}
for any $\varPhi\in\cals(\Sym_3(F))$. 
We have
\begin{align*}
\widehat{\Phi}\left(-\begin{pmatrix}
b_{11} & y_3 & y_2 \\
y_3 & b_{22} & y_1 \\
y_2 & y_1 & b_{33}
\end{pmatrix}\right)
&=\prod_{i=1}^3\widehat{\vph_i}(-2y_i)\widehat{\phi_i}(-b_{ii})\\
&=\vph_{\chi\mu_1\nu_2\nu_3}(2y_1)\vph_{\chi\nu_1\mu_2\nu_3}(2y_2)\vph_{\chi\nu_1\nu_2\mu_3}(2y_3)\prod_{i=1}^3\II_\frkp(b_{ii})
\end{align*}
by definition.
\end{proof}


\subsection{The $\frkp$-adic zeta integral}{We introduce some special Whittaker functions that will be used in the $\frakp$-adic zeta integral of Garrett, Piateski-Shapiro and Rallis. 
Define the embedding $\bft:F^\times\to\GL_2(F)$ by $\bft(a)=\diag{a,1}$. 
The Kirillov model of $\pi_i$ is given by $\scrk(\pi_i)=\{W\circ\bft\;|\;W\in\scrw(\pi_i)\}$. 
The restriction map $\circ\bft:\scrw(\pi_i)\to\scrk(\pi_i)$ is injective by \cite[Proposition 4.4.7]{Bump97Grey}. 
Let $W_{\nu_i}\in\scrw(\pi_i)$ be the Whittaker function uniquely characterized by  
\[W_{\nu_i}(\bft(a))=\nu_i(a)|a|^{1/2}\II_\frko(a)\]
for $a\in F^\times$. 
This vector belongs to the ordinary line $\scrw(\pi)^\ord(\nu_i)$ with respect to $\nu_i$ in the Whittaker model introduced in \cite[\S2.5, Corollary 2.3]{Hsieh_triple}. Its connection with the classical $p$-ordinary elliptic modular forms can be found in \cite[Remark 2.5]{Hsieh_triple}. Moreover, $W_{\nu_i}$ corresponds to the section $f_i^\dagger\in I(\mu_i,\nu_i)$ supported in the open cell (see (\ref{tag:existence}) below).}
Fix a prime element $\vpi$ of $\frko$. 
For each non-negative integer $n$ we put 
\begin{align*}
m_n&=\bfm(\vpi^n),  &
t_n&=J_1^{-1}m_n, & 
W_i^{(n)}&=\pi_i(t_n)W_{\nu_i}.  
\end{align*}
\begin{prop}\label{prop:11}
If $n\geq\max\{1, c(\chi), c(\mu_i), c(\nu_i)\;|\;i=1,2,3\}$, then 
\begin{multline*}
Z(W_1^{(n)},W_2^{(n)},W_3^{(n)},f_\Phi(\chi))=(1+q^{-1})^{-3}\prod_{j=1}^3\left(\frac{\bet_j}{q\alp_j}\right)^n\\
\times(\chi\nu_1\nu_2\nu_3)(-1)\gam\left(\frac{1}{2},\pi_1\otimes\chi\nu_2\nu_3,\addchar\right)^{-1}\prod_{i=2,3}\gam\left(\frac{1}{2},\chi\nu_1\nu_i\mu_{5-i},\addchar\right)^{-1}, 
\end{multline*}
where $\alp_i=\mu_i(\vpi)$ and $\bet_i=\nu_i(\vpi)$.  
\end{prop}

\begin{proof}
{For a quasi-character $\chi$ we define $\Re \chi$ as the unique real number $\sig$ such that $\chi\Abs_F^{-\sig}$ is unitary. } 
We associate to $f_i\in I(\mu_i,\nu_i)$ a function $W(f_i)\in\scrw(\pi_i)$ by 
\[W(g,f_i)=\int_Ff_i(J_1\bfn(u)g)\addchar(-u)\,\d u=\lim_{k\to\infty}\int_{\frkp^{-k}}f_i(J_1\bfn(u)g)\addchar(-u)\,\d u, \]
{where $g\in\GL_2(F)$ is arbitrarily fixed. 
Here the limit stabilizes and the integral makes sense for any $f_i\in\pi_i$ (see \cite[p.~485]{Bump97Grey}).} 
The integral $W$ factors through the quotient $I(\mu_i,\nu_i)\twoheadrightarrow\St\otimes\mu_i\Abs_F^{1/2}$ when $\mu_i^{}\nu_i^{-1}=\Abs_F^{-1}$. 
{Put $f_i'=\rho(J_1)f_i^{}$. 
Then 
\[Z(W(f_1),W(f_2),W(f_3),f_\vPh(\chi))
=\int_{U^0Z\bsl H}W(g_1,f_1)W(g_2,f_2')W(g_3,f_3')f_\vPh(\iota_0(g_1,g_2,g_3),\chi)\,\d g_1^{}\d g_2'\d g_3'\]
by the definition (\ref{tag:iota}) of $\iota_0$. 
Observe that 
\[W(g_1,f_1)W(g_2,f_2')W(g_3,f_3')=\int_{U^0}\bfW_1(u_0(g_1,g_2,g_3);f_2',f_3')\,\d u_0, \]
where 
\[\bfW_1(g_1,g_2,g_3;f_2',f_3')=W(g_1,f_1)f_2'(J_1g_2)f_3'(J_1g_3). \]
Substituting this expression, we are led to  
\[Z(W(f_1),W(f_2),W(f_3),f_\vPh(\chi))
=\int_{Z\bsl H}\bfW_1(g;f_2',f_3')\,f_\vPh(\iota_0(g),\chi)\,\d g. \]
Define a function $\calf$ on $\SL_2(F)$ by 
\[\calf(g)
=\int_{\SL_2(F)^2}f_2'(J_1g_2)f_3'(J_1g_3)f_\vPh(\iota_0(g,g_2,g_3),\chi)\,\d g_2'\d g_3'. \] 
Let $T'=\bfm(F^{\times})$ be the diagonal torus of $\SL_2(F)$. 
Then 
\begin{align}
Z(W(f_1),W(f_2),W(f_3),f_\vPh(\chi))=\int_{F^\times}\d^\times a\int_{T'\bsl\SL_2(F)}W(\bft(a)g,f_1)\int_{\SL_2(F)^2} \notag\\
f_2'(J_1\bft(a)g_2)f_3'(J_1\bft(a)g_3)f_\vPh(\iota_0(\bft(a)g,\bft(a)g_2,\bft(a)g_3),\chi)\,\d g_2'\d g_3'\d g' \notag \\
=\int_{F^\times}\int_{T'\bsl\SL_2(F)}W(\bft(a)g,f_1)\chi(a)\nu_2(a)\nu_3(a)\calf(g)\,\d g'\d^\times a \label{tag:15}
\end{align}
by (\ref{tag:12}). 
To justify the manipulations we show that the integral 
\[\int_{F^\times}\int_{\SL_2(\frko)}\int_F|W(\bft(a)k,f_1)\chi(a)\nu_2(a)\nu_3(a)\calf(\bfn(x)k)|\,\d x\d k\d^\times a\]
is convergent at least for $\Re\chi\gg 0$ and our choice of test vectors. 
The integral 
\[\int_{F^\times}|W(\bft(a)k,f_1)\chi(a)\nu_2(a)\nu_3(a)|\,\d^\times a\]
is absolutely convergent.} 
Recall that 
\begin{align*}
\bft(a)&=\begin{pmatrix} a & 0 \\ 0 & 1 \end{pmatrix}, & 
\bfm(a)&=\begin{pmatrix} a & 0 \\ 0 & a^{-1} \end{pmatrix}, & 
\bfn(x)&=\begin{pmatrix} 1 & x \\ 0 & 1 \end{pmatrix}, & 
\bfn^-(u)&=\begin{pmatrix} 1 & 0 \\ u & 1 \end{pmatrix}. 
\end{align*}
We frequently use the integration formula 
\[\int_{\SL_2(F)}h(g)\,\d g'=\frac{\zet(2)}{\zet(1)}\int_F\int_F\int_{F^\times}h(\bfm(a)\bfn^-(u)\bfn(x))\,\d^\times a\d u\d x  \]
for an integrable function $h$ on $\SL_2(F)$ {(\cf \cite[3.1.6, p. 206]{MV10}). } 
Observe that 
\begin{align*}
\calf(g)
=&\frac{\zet(2)^2}{\zet(1)^2}\int_{F^2}\d x_2\d x_3\, f_2'(J_1\bfn(x_2))f_3'(J_1\bfn(x_3))\int_{F^{\times 2}}\prod_{i=2,3}(\nu^{}_i\mu_i^{-1})(a_i)\frac{\d^\times a_i}{|a_i|}\\
&\times\int_{F^2} f_\vPh(\iota_0(g,\bfm(a_2)\bfn^-(u_2)\bfn(x_2),\bfm(a_3)\bfn^-(u_3)\bfn(x_3)),\chi)\,\d u_2\d u_3. 
\end{align*}

Let $f^\dagger_i\in I(\mu_i,\nu_i)$ be such that $f^\dagger_i(g)=0$ unless $g\in B_2J_1U_2$ and 
such that $f^\dagger_i(J_1\bfn(x))=\II_\frko(x)$ for $x\in F$ {(\cf Definition \ref{def:cell} and Remark \ref{rem:20})}. 
One can easily check 
\begin{align}
W_{\nu_i}&=W(f_i^\dagger), &
W_i^{(n)}&=W(\rho(t_n)f_i^\dagger). \label{tag:existence}
\end{align}
{Now we let $f_i=\rho(t_n)f_i^\dagger$. 
Since $m_n=J_1t_n=\diag{\vpi^n,\vpi^{-n}}$, $f_i'=\rho(m_n)f^\dagger_i\in I(\mu_i,\nu_i)$. }
Observe that 
\[f_i'(J_1\bfn(x))=f_i^\dagger(J_1m_n\bfn(x\vpi^{-2n}))=\bet_i^n\alp_i^{-n}q^n\II_{\frkp^{2n}}(x). \]
Lemma \ref{lem:13} shows that {if $\vPh=\Phi$ or $n$ is sufficiently large, then}
\begin{align*}
\rho_3(\iota(\ono_2,\bfn(x_2),\bfn(x_3))J_3)f_\vPh(\chi)&=\rho_3(J_3)f_\vPh(\chi) &
(x_2,x_3&\in\frkp^{2n}). 
\end{align*}
It follows that $\calf(g)$ equals the product of $\frac{f'_2(J_1)f'_3(J_1)}{q^{4n}(1+q^{-1})^2}$ and 
\[\int_{F^{\times 2}\oplus F^2} f_\vPh(\iota_0(g,\bfm(a_2)\bfn^-(u_2),\bfm(a_3)\bfn^-(u_3)),\chi)\prod_{i=2,3}\frac{\nu_i(a_i)\d^\times a_i}{\mu_i(a_i)|a_i|}\d u_i. \]
In particular, $\calf(\bfn^-(u)\bfn(x))$ equals the product of $\frac{f'_2(J_1)f'_3(J_1)}{q^{4n}(1+q^{-1})^2}$ and
\begin{align*}
&\int_{F^{\times 2}\oplus F^2}\vPh\left(\begin{pmatrix}
x & -a_2 & -a_3 \\
-a_2 & -u_2 & -a_2a_3u \\
-a_3 & -a_2a_3u &  -u_3
\end{pmatrix}\right)\prod_{i=2,3}(\hat\ome\chi^2\nu^{}_i\mu_i^{-1})(a_i)|a_i|\d^\times a_i\d u_i\\
=&\int_{F^{\times 2}\oplus F^2}\vph_1(-a_2a_3u)\vph_2(-a_2)\vph_3(-a_3)\phi_1(x)\phi_3({-u_2})\phi_2({-u_3})\prod_{i=2,3}(\hat\ome\chi^2\nu^{}_i\mu_i^{-1})(a_i)|a_i|\d^\times a_i\d u_i
\end{align*}
by (\ref{tag:13}) {and \eqref{E:Phi}}. 
Its integral over $x,u\in F$  
 converges absolutely if $\Re\chi$ is large. 
 
Recall the functional equations 
\begin{align}
\gam\left(\frac{1}{2},\pi_1\otimes\chi,\addchar\right)\int_{F^\times}W_1(\bft(a)g)\chi(a)\d^\times a
&=\int_{F^\times}W_1(\bft(a)J_1^{-1}g)(\chi\ome_1)^{-1}(a)\d^\times a, \notag\\
\gamma(s,\chi,\addchar)\int_{F^\times}\vph(a)\chi(a)|a|^s\,\d^\times a&=\int_{F^\times}\widehat{\vph}(a)\chi(a)^{-1}|a|^{1-s}\,\d^\times a \label{tag:fq}
\end{align}
for $W_1\in\scrw(\pi_1)$ and $\vph\in\cals(F)$ {(see \cite[Theorem 4.7.5, Proposition 3.1.5]{Bump97Grey})}. 
It follows from (\ref{tag:15}) that 
\begin{align}
&\gam\left(\frac{1}{2},\pi_1\otimes\chi\nu_2\nu_3,\addchar\right)Z(W_1^{(n)},W_2^{(n)},W_3^{(n)}, f_\vPh(\chi)) \notag\\
=&\int_{F^\times}\int_{T'\bsl\SL_2(F)}W_1^{(n)}(\bft(a)J_1^{-1}g)(\chi\nu_2\nu_3\ome_1)(a)^{-1}\calf(g)\,\d g'\d^\times a\notag\\
=&\int_{F^\times}\int_FW_1^{(n)}(\bft(a)J_1^{-1}\bfn(x))(\chi\nu_2\nu_3\ome_1)(a)^{-1}\calf_{\addchar}(a,x)\,\d x\d^\times a, \label{tag:16} 
\end{align}
where 
\[\calf_{\addchar}(a,x)=(1+q^{-1})^{-1}\int_F\calf(\bfn^-(u)\bfn(x))\addchar(-au)\,\d u. \]

We have seen that 
\begin{align*}
\frac{q^{4n}(1+q^{-1})^3}{f'_2(J_1)f'_3(J_1)}\calf_{\addchar}(a,x)=&\int_{F^{\times 2}}\hat\ome(a_2a_3)\chi(a_2a_3)^2|a_2a_3|^2\vph_3(-a_2)\vph_2(-a_3)\\
&\times\int_F\vph_1(-a_2a_3u)\overline{\addchar(au)}\d u\,\phi_1(x)\prod_{i=2,3}\frac{\nu^{}_i(a_i)\d^\times a_i}{\mu_i(a_i)|a_i|}\int_F\phi_i(u_i)\,\d u_i\\
=&\phi_1(x)\int_{F^{\times 2}}\widehat{\vph_1}\left(\frac{a}{a_2a_3}\right)\prod_{i=2,3}\widehat{\phi_i}(0)(\hat\ome\chi^2\nu_i^{}\mu_i^{-1})(a_i)\vph_{5-i}(-a_i)\,\d^\times a_i. 
\end{align*}
If $xu\neq -1$, then 
\[J_1\bfn(x)\bfn^-(u)=\bfm((1+ux)^{-1})\bfn(-(1+xu)u)J_1\bfn((1+xu)^{-1}x),  \]
which implies that 
\begin{align*}
\rho(\bfn^-(u))f_1^{\dagger}&=f_1^{\dagger}, & 
\pi_1(\bfn^-(u))W_{\nu_1}&=W_{\nu_1}
\end{align*}
 for $u\in\frkp^n$. 
If $\phi_1(x)\neq 0$, then since $x\in\frkp^{-1}$, 
\begin{align*}
W_1^{(n)}(\bft(a)J_1\bfn(x))&=W_{\nu_1}(\bft(a)m_n\bfn^-(-\vpi^{2n}x))=q^{-n}\bet_1^n\alp_1^{-n}\nu_1(a)|a|^{1/2}\II_{\frko}(a\vpi^{2n}).  
\end{align*}
We conclude by (\ref{tag:16}) that 
\begin{align*}
&\gam\left(\frac{1}{2},\pi_1\otimes\chi\nu_2\nu_3,\addchar\right)Z(W_1^{(n)},W_2^{(n)},W_3^{(n)},f_\vPh(\chi))\\
=&\int_{F^\times}\int_FW^{(n)}_1(\bft(a)J_1^{-1}\bfn(x))(\chi\nu_2\nu_3\ome_1)(a)^{-1}\frac{f'_2(J_1)f'_3(J_1)}{q^{4n}(1+q^{-1})^3}\calf_{\addchar}(a,x)\,\d x\d^\times a\\
=&\int_{F^\times}W^{(n)}_1(\bft(a)J_1^{-1})(\chi\nu_2\nu_3\ome_1)(a)^{-1}\frac{f'_2(J_1)f'_3(J_1)}{q^{4n}(1+q^{-1})^3}\int_F\calf_{\addchar}(a,x)\,\d x\d^\times a\\
=&\frac{f'_2(J_1)f'_3(J_1)}{q^{4n}(1+q^{-1})^3}\widehat{\phi_1}(0)\widehat{\phi_2}(0)\widehat{\phi_3}(0)\\
&\times\int_{F^{\times 3}}\frac{W_1^{(n)}(\bft(a)J_1^{-1})}{(\chi\nu_2\nu_3\ome_1)(a)}\widehat{\vph_1}\left(\frac{a}{a_2a_3}\right)\d^\times a\prod_{i=2,3}(\hat\ome\chi^2\nu^{}_i\mu_i^{-1})(a_i)\vph_{5-i}(-a_i)\, \d^\times a_i. 
\end{align*}
The last integral is equal to 
\begin{align*}
&\hat\ome(-1)\int_{F^{\times 3}}\frac{W^{(n)}_1(\bft(aa_2a_3)J_1)}{(\chi\nu_2\nu_3\ome_1)(a)}\widehat{\vph_1}(a)\,\d^\times a\prod_{i=2,3}(\chi\nu_i\mu_{5-i})(a_i)\vph_{5-i}(a_i)\, \d^\times a_i\\
=&\hat\ome(-1)W^{(n)}_1(J_1)\int_{F^{\times 3}}\frac{\nu_1(aa_2a_3)}{(\chi\nu_2\nu_3\ome_1)(a)}|aa_2a_3|^{1/2}\II_\frko(aa_1a_2\vpi^{2n})\widehat{\vph_1}(a)\,\d^\times a\prod_{i=2,3}(\chi\nu_i\mu_{5-i})(a_i)\vph_{5-i}(a_i)\, \d^\times a_i\\
=&\hat\ome(-1)W^{(n)}_1(J_1)\int_{F^{\times 3}}\frac{|a|^{1/2}\widehat{\vph_1}(a)}{(\chi\mu_1\nu_2\nu_3)(a)}\II_\frko(aa_1a_2\vpi^{2n})\,\d^\times a\prod_{i=2,3}\vph_{5-i}(a_i)(\chi\nu_1\nu_i\mu_{5-i})(a_i)|a_i|^{1/2}\, \d^\times a_i.  
\end{align*}

{Now we let $\phi_1=\phi_2=\phi_3=\widehat{\II_\frkp}$, 
$\vph_1=\widehat{\vph_{\chi\mu_1\nu_2\nu_3}}$,  
$\vph_2=\widehat{\vph_{\chi\nu_1\mu_2\nu_3}}$ and  
$\vph_3=\widehat{\vph_{\chi\nu_1\nu_2\mu_3}}$. 
Since if $\vph_{5-i}(a_i)\neq 0$, then $a_i\in\frkp^{-n}$, the integral above coincides with the product \[(\chi\mu_1\nu_2\nu_3)(-1)\prod_{i=2,3}Z(\vph_{5-i},\chi\nu_1\nu_i\mu_{5-i}\Abs_F^{1/2}).\] 
The proof is complete by $f'_i(J_1)=\bet_i^n\alp_i^{-n}q^n$, $W^{(n)}_1(J_1)=\bet_1^n\alp_1^{-n}q^{-n}$ and the functional equation (\ref{tag:fq}). }
\end{proof}

{We specify $\phi_i,\vph_i$ at the final stage of the proof of Proposition \ref{prop:11}. 
For all $\vPh\in\cals(\Sym_3(F))$ the computation is valid if $n$ is sufficiently large. 
We record the following formula. 
\begin{cor}\label{cor:general}
For arbitrarily chosen $\phi_i,\vph_i\in\cals(F)$, if $\ell$ is sufficiently large, then 
\begin{align*}
&\gam\left(\frac{1}{2},\pi_1\otimes\chi\nu_2\nu_3,\addchar\right)\gam\left(\frac{1}{2},\chi\nu_1\nu_2\mu_3,\addchar\right)\gam\left(\frac{1}{2},\chi\nu_1\mu_2\nu_3,\addchar\right)Z(W_1^{(\ell)},W_2^{(\ell)},W_3^{(\ell)},f_\vPh(\chi))\\
=&\hat\ome(-1)\frac{\prod_{j=1}^3(q^{-1}\alp_j^{-1}\bet_j^{})^\ell}{(1+q^{-1})^3}\widehat{\phi_1}(0)\widehat{\phi_2}(0)\widehat{\phi_3}(0)Z(\widehat{\vph_1},(\chi\mu_1\nu_2\nu_3)^{-1}\Abs_F^{1/2})\prod_{i=2,3}Z(\widehat{\vph_{5-i}},(\chi\nu_1\nu_i\mu_{5-i})^{-1}\Abs_F^{1/2}). 
\end{align*}
\end{cor}}


\subsection{Restatements}

We rewrite Propositions \ref{prop:11} and \ref{prop:12} in a form which is more convenient for our later application. Suppose that $\pi_i$ is the irreducible subrepresentation of $I(\mu_i,\nu_i)$ with $\mu_i$ unramified. 
Thus  {the contragredient $\pi_i^\vee\iso \pi_i\ot\om_i^{-1}$ is the irreducible quotient of $I(\mu_i^{-1},\nu_i^{-1})$} and $\om_i=\mu_i\nu_i$ coincides with $\nu_i$ on $\frko^\times$. 
Let $\breve W_i:=W_{\nu_i^{-1}}\in \scrw(\pi_i^\vee)$. By definition
\[\breve W_i(\bft(a))=\nu_i(a)^{-1}|a|^{1/2}\bbI_\frko(a). \]

\begin{defn}\label{def:psection}
We associate to the quadruplet of characters of $\frko^\times$
\[\cD=(\chi,\om_1,\om_2,\om_3)\] 
a holomorphic section $f_{\cD,s}=f_{\Phi_\cD}(\chi\hat\ome\Abs_F^s)$ of $I_3(\hat\om^{-1},\chi^{}\hat\om^{}\Abs_F^{s})$ by
\[\Phi_\cD\left(\begin{pmatrix}
u_1&x_3&x_2\\
x_3&u_2&x_1\\
x_2&x_1&u_3
\end{pmatrix}\right)=\prod_{i=1}^3\widehat{\bbI_\frkp}(u_i)\widehat{\varphi_{\chi\om_i}}(x_i). \]
For each quadruplet $(\chi_0,\chi_1,\chi_2,\chi_3)$ of characters of $\frko^\x$, valued in a commutative ring $R$ we set 
{\beq\label{E:QB}\textcolor{black}{\cQ_B(\chi_0,\chi_1,\chi_2,\chi_3):=\chi_0(8b_{12}b_{23}b_{13})\cdot\chi_1(2b_{23})\chi_2(2b_{13})\chi_3(2b_{12})\bbI_{\Xi_\frkp}(B).}\eeq}
\end{defn}
 Given a section $f_s$ of $I_3(\hat\om^{-1},\chi^{}\hat\om^{}\Abs_F^{s})$, we are interested in the quantity 
\beq\label{E:padiczeta}
Z^*_\frkp(f_s)=\frac{Z(\rho(t_n)\breve W_1,\rho(t_n)\breve W_2,\rho(t_n)\breve W_3,f_s)}{L\bigl(s+\frac{1}{2},\pi_1\times\pi_2\times\pi_3\otimes\chi\bigl)}
\prod_{i=1}^3\frac{\zet(1)}{\zet(2)}\left(\frac{\ome_i(\vpi)q}{\mu_i(\vpi)^2}\right)^n. 
\eeq


\begin{prop}\label{prop:13}
Notations and assumptions being as above, we have 
\begin{align*}
\rho_3(\iota(g_1,g_2,g_3))f_{\cald,s}&=f_{\cald,s}\prod_{i=1}^3\ome_i^{\downarrow}(g_i), & 
g_1,g_2,g_3&\in K_0^{(1)}(\frkp^{2n})
\end{align*}
if $\det g_1=\det g_2=\det g_3$ and {$n\geq\max\{1,c(\chi),c(\ome_1),c(\ome_2),c(\ome_3)\}$.} 
Moreover, 
\begin{align*}
\calw_B(f_{\cD,s})&=\cQ_B(\cald), &
Z^*_\frkp(f_{\cD,s})&=\chi(-1)E_\frkp\left(s+\frac{1}{2},\pi_1\times\pi_2\times\pi_3\otimes\chi\right), 
\end{align*}
where 
\begin{align*}
E_\frkp(s,\pi_1\times \pi_2\times \pi_3\ot\chi)^{-1}
=&L(s,\pi_1\times\pi_2\times\pi_3\ot\chi)\gamma(s,\pi_1\ot\chi\mu_2\mu_3,\addchar)\prod_{i=2,3}\gamma(s,\chi\mu_1\mu_i\nu_{5-i},\addchar). 
\end{align*}
\end{prop}


\begin{proof}
Since $\ome_i$ coincides with $\nu_i$ on $\frko^\times$, we apply Proposition \ref{prop:12} and get the formula for $\calw_B(f_{\cD,s})$ by replacing {$\pi_i$, $\ome_i$, $\mu_i$, $\nu_i$, $\chi$} by $\pi_i^\vee$, $\ome_i^{-1}$, $\mu_i^{-1}$, $\nu_i^{-1}$, $\chi\hat\ome$, respectively. 
Proposition \ref{prop:11} applied to $W_{\nu_i^{-1}}=\breve W_i$ and $I_3(\hat\om^{-1},\chi^{}\hat\om^{})$ gives
\begin{multline*}
Z(\rho(t_n)\breve W_1,\rho(t_n)\breve W_2,\rho(t_n)\breve W_3,f_{\cD,s})=(1+q^{-1})^{-3}\prod_{i=1}^3\left(\frac{\nu_i(\vpi)^{-1}}{q\mu_i(\vpi)^{-1}}\right)^n\\
\times\chi(-1)\gam\left(s+\frac{1}{2},\pi_1^\vee\otimes(\chi\hat\ome)\nu_2^{-1}\nu_3^{-1},\addchar\right)^{-1}\prod_{i=2,3}\gam\left(s+\frac{1}{2},(\chi\hat\ome)\nu_1^{-1}\nu_i^{-1}\mu_{5-i}^{-1},\addchar\right)^{-1}, 
\end{multline*}
from which the formula for $Z_\frkp^*({f_{\cD,s}})$ readily follows. 
\end{proof}

{For later use we rewrite Corollary \ref{cor:general} in the following way: 
\begin{cor}\label{cor:general2}
For arbitrarily chosen $\phi_i,\vph_i\in\cals(F)$, if $n$ is sufficiently large, then 
\begin{align*}
Z_\frkp^*(f_\Phi(\chi))=&\hat\ome(-1)E_\frkp\left(\frac{1}{2},\pi_1\times\pi_2\times\pi_3\otimes\chi\right)\\
&\times\widehat{\phi_1}(0)\widehat{\phi_2}(0)\widehat{\phi_3}(0)Z(\widehat{\vph_1},(\chi\nu_1\mu_2\mu_3)^{-1}\Abs_F^{1/2})\prod_{i=2,3}Z(\widehat{\vph_{5-i}},(\chi\mu_1\mu_i\nu_{5-i})^{-1}\Abs_F^{1/2}). 
\end{align*}
\end{cor}}

We will use the following lemma to achieve the functional equation of the $p$-adic $L$-function in \S \ref{ssec:functeq}. 

\begin{lm}\label{lem:modfactor}
Put $\breve\chi=\chi^{-1}\hat\ome^{-1}$. 
Then 
\[E_\frkp(1-s,\pi_1\times\pi_2\times\pi_3\otimes\breve\chi)=\hat\ome(-1)E_\frkp(s,\pi_1\times\pi_2\times\pi_3\otimes\chi)\vep(s,\pi_1\times\pi_2\times\pi_3\otimes\chi,\addchar). \]
\end{lm}

\begin{proof}
Since $\pi_i\otimes\ome_i^{-1}\simeq\pi_i^\vee$, we get 
\begin{align*}
E_\frkp(s,\pi_1\times\pi_2\times\pi_3\otimes\breve\chi)^{-1}
=&L(s,\pi_1^\vee\times\pi_2^\vee\times\pi_3^\vee\otimes\overline{\chi})\gamma(s,\pi_1^\vee\ot\overline{\chi\nu_2\nu_3},\addchar)\prod_{i=2,3}\gamma(s,\overline{\chi\nu_1\nu_i\mu_{5-i}},\addchar), 
\end{align*}
where $\pi_i\simeq I(\mu_i,\nu_i)$. 
By definition we arrive at 
\begin{align*}
&\vep(s,\pi_1\times\pi_2\times\pi_3\otimes\chi,\addchar)L(s,\pi_1\times\pi_2\times\pi_3\otimes\chi)^{-1}E_\frkp\left(1-s,\pi_1\times\pi_2\times\pi_3\otimes\breve\chi\right)^{-1}\\
=&\gam(s,\pi_1\times\pi_2\times\pi_3\otimes\chi,\addchar)\gamma(1-s,\pi_1^\vee\ot\overline{\chi\mu_2\mu_3},\addchar)\prod_{i=2,3}\gamma(1-s,\overline{\chi\mu_1\mu_i\nu_{5-i}},\addchar).  
\end{align*}
The statement can now be deduced from multiplicativity and the functional equation of gamma factors. 
\end{proof}


\section{Computation of the local zeta integral: the ramified case}\label{sec:3}

Recall that $\St$ denotes the Steinberg representation of $\GL_2(F)$. {We deal with two types of representations $\pi_i$ of $\GL_2(F)$: either (i) an irreducible unramified principal series representation or (ii) the Steinberg representation twisted by an unramified character.} 
Since 
\beq
Z(W_1\otimes\chi_1,W_2\otimes\chi_2,W_3\otimes\chi_3,f)=Z(W_1,W_2,W_3,f\otimes(\chi_1\chi_2\chi_3)\circ\nu_3) \label{tag:21}
\eeq
for characters $\chi_1,\chi_2,\chi_3$ of $F^\times$, where
\begin{align*}
(W_i\otimes\chi_i)(g_i)&=W_i(g_i)\chi_i(\det g_i), & 
(f\otimes\chi\circ\nu_3)(g)&=f(g)\chi(\nu_3(g)) 
\end{align*} 
{for $g_1,g_2,g_3\in\GL_2(F)$ and $g\in\GSp_6(F)$, }
there is no harm in assuming that $\pi_i\simeq I(\Abs_F^{-t_i},\Abs_F^{t_i})$ with $t_i\in\C$ {in Case (i)} or $\pi_i\simeq\St$ {in Case (ii)}.  
When $\pi_i\simeq I(\Abs_F^{-t_i},\Abs_F^{t_i})$, we denote the unique Whittaker function which takes the value $1$ on $\GL_2(\frko)$ by $W_i^0\in\scrw(\pi_i)$ and let $W^\pm_i\in\scrw(\pi_i)$ be the unique Whittaker function characterized by 
\[W^\pm_i(\bft(a))=|a|^{(\pm 2t_i+1)/2}\II_\frko(a) \]
for $a\in F^\times$. 
When {$\pi_i\simeq\St$, we define $W_i^+\in\scrw(\St)$ by 
\[W^+_i(\bft(a))=|a|\II_\frko(a). \]
We set $t_i=\frac{1}{2}$ in Case (ii). 
Then $\pi_i$ is a quotient of $I(\Abs_F^{-t_i},\Abs_F^{t_i})$ in both cases. 
We define $f_i^\dagger\in I(\Abs_F^{-t_i},\Abs_F^{t_i})$ as in the proof of Proposition \ref{prop:11}. }
Recall that $W_i^+=W(f_i^\dagger)$. 
Put 
\begin{align*}
\eta_1&=\begin{pmatrix} 0 & -1 \\ \vpi & 0 \end{pmatrix}, & 
\calw^\pm_i&=\pi_i(\eta_1)W_i^\pm, &
\calw^0_i&=\pi_i(\eta_1)W_i^0. 
\end{align*} 

\begin{lm}\label{lem:20}
If $\pi_i$ is an irreducible unramified principal series, then 
\[W_i^0=q^{1/2}\frac{\calw_i^+-\calw_i^-}{q^{-t_i}-q^{t_i}}. \]
\end{lm}

\begin{proof}
{Since the ordinary vector $W_i^\pm$ is obtained as the stabilization of the spherical vector $W_i^0$ with respect to $q^{(\mp 2t_i+1)/2}$, we get the relation $W_i^\pm=W^0_i-q^{-(\mp 2t_i+1)/2}\calw^0_i$ (see \S 3.2 of \cite{CH17Crelle}). 
The stated identity now follows in view of $\pi_i(\eta_1)\calw^0_i=W_i^0$. }
\end{proof}

Fix an unramified character $\chi=\Abs_F^s$ of $F^\times$. 
We will abbreviate $I_3(\chi)=I_3(1,\chi)$. 
{We consider the section
\begin{align*}
f_{\Phi^0}(\chi)&\in I_3(\chi), & 
\Phi^0&=\II_{\Sym_3(\frko)}. 
\end{align*}
By the Iwahori decomposition of $K_0^{(3)}(\frkp)$} 
\[\calp_3J_3\bfn(\Sym_3(\frko))=\calp_3J_3K_0^{(3)}(\frkp)=\calp_3K_0^{(3)}(\frkp)J_3K_0^{(3)}(\frkp). \]
The restriction of the section $f_{\Phi^0}(\chi)$ to $\GSp_6(\frko)$ is the characteristic function of $K_0^{(3)}(\frkp)J_3K_0^{(3)}(\frkp)$. 
In particular, 
\[\rho_3(k)f_{\Phi^0}(\chi)=f_{\Phi^0}(\chi)\] 
for $k\in K_0^{(3)}(\frkp)$ (\cf Lemma \ref{lem:13}). 

\begin{lm}\label{lem:21}
Assume that $\pi_1\simeq\St$. 
Then 
\[Z(\calw_1^+,\calw_2^+,\calw_3^+, f_{\Phi^0}(\chi)) 
=-\frac{q^{s-2}}{(1+q^{-1})^3}\zet(s+1+t_2+t_3)\prod_{i=2,3}\zet(s+1+t_i-t_{5-i}). \]
\end{lm}

\begin{Remark}\label{rem:21}
{
When $\pi_1,\pi_2,\pi_3$ are discrete series of $\PGL_2(F)$ of level $\frkp$, Gross and Kudla \cite{GK92} constructed a $K_0^{(3)}(\frkp)$-fixed section $\Phi^\natural(s)\in I_3(1,\chi)$ with nice properties and showed that $Z(W_1^+,W_2^+,W_3^+,\Phi^\natural(s))$ equals $L\bigl(s+\frac{1}{2},\pi_1\times\pi_2\times\pi_3)$ times a normalizing factor. 
The section $\Phi^\natural(s)$, which depends on the sign $\vep\bigl(\frac{1}{2},\pi_1\times\pi_2\times\pi_3,\addchar\bigl)$ (\cf Remark \ref{rem:localfactor} below), is different from our choice $f_{\Phi^0}(\chi)$. 
However, }
Lemma \ref{lem:21} is compatible with their computation. 
Let $h^0(\chi)\in I_3(\chi)$ be the function whose restriction to $\GSp_6(\frko)$ is the characteristic function of $K_0^{(3)}(\frkp)$. 
Put $\eta_3=\iota(\eta_1,\eta_1,\eta_1)$. 
Then 
\begin{align*}
\eta_3^{} K_0^{(3)}(\frkp)\eta_3^{-1}&=K_0^{(3)}(\frkp), & 
f_{\Phi^0}(\chi)&=q^{3+3s}\rho_3(\eta_3)h^0(\chi) 
\end{align*}
by Lemma 3.1 of \cite{GK92}. 
We obtain 
\begin{align*}
Z(\calw^+_1,\calw^+_2,\calw^+_3,f_{\Phi^0}(\chi))
=&q^{3+3s}Z(W_1^+,W_2^+,W_3^+,h^0{(\chi)}). 
\end{align*}
When $\pi_1\simeq\pi_2\simeq\pi_3\simeq\St$, Proposition 4.2 of \cite{GK92} gives 
\[Z(W_1^+,W_2^+,W_3^+,h^0{(\chi)})=-(q+1)^{-3}q^{-2s-2}L\left(s+\frac{1}{2},\St\times\St\times\St\right). \]
\end{Remark}

\begin{proof}
{Applying (\ref{tag:15}) with $\chi=\Abs_F^s$, $\mu_i=\Abs_F^{-t_i}$, $\nu_i=\Abs_F^{t_i}$, $f_i=\rho(\eta_1)f_i^\dagger$ and $\vPh=\Phi^0$, }we have 
\begin{align*}
Z(\calw_1^+,\calw_2^+,\calw_3^+,f_{\Phi^0}(\chi))
=\int_{F^\times}\int_{T'\bsl\SL_2(F)}\calw_1^+(\bft(a)g)|a|^{s+t_2+t_3}\calf(g)\,\d g'\d^\times a. 
\end{align*}
{Let $f_i'=\rho(J_1)f_i=\rho(\bft(\vpi))f_i^\dagger$.} 
Since 
\[f_i'(J_1\bfn(x))=f_i^\dagger(J_1\bft(\vpi)\bfn(x/\vpi))=q^{(1-2t_i)/2}\II_\frkp(x), \]
we get 
\begin{multline*}
\calf(g)
=(1+q^{-1})^{-2}\int_{F^2}\d x_2\d x_3\, f'_2(J_1\bfn(x_2))f'_3(J_1\bfn(x_3))\int_{F^{\times 2}}\prod_{i=2,3}|a_i|^{2t_i}\frac{\d^\times a_i}{|a_i|}\\
\times\int_{F^2} {f_{\Phi^0}}(\iota_0(g,\bfm(a_2)\bfn^-(u_2)\bfn(x_2),\bfm(a_3)\bfn^-(u_3)\bfn(x_3)),\chi)\,\d u_2\d u_3\\
=\int_{F^{\times 2}\oplus F^2}{f_{\Phi^0}}(\iota_0(g,\bfm(a_2)\bfn^-(u_2),\bfm(a_3)\bfn^-(u_3)),\chi)\frac{\prod_{i=2,3}|a_i|^{2t_i-1}\d^\times a_i\d u_i}{q^{1+t_2+t_3}(1+q^{-1})^2}. 
\end{multline*}
In view of (\ref{tag:13}) 
\begin{align*}
\calf(\bfn^-(u)\bfn(x))
&=\int_{F^{\times 2}\oplus F^2}
\Phi^0\left(\begin{pmatrix}
x & -a_2 & -a_3 \\
-a_2 & -u_2 & -a_2a_3u \\
-a_3 & -a_2a_3u &  -u_3
\end{pmatrix}\right)\frac{\prod_{i=2,3}|a_i|^{1+2s+2t_i}\d^\times a_i\d u_i}{q^{1+t_2+t_3}(1+q^{-1})^2}\\
&=q^{-1-t_2-t_3}(1+q^{-1})^{-2}\II_\frko(x)\int_{\frko^2}\II_\frko(a_2a_3u)\prod_{i=2,3}|a_i|^{1+2s+2t_i}\d^\times a_i. 
\end{align*}
Owing to (\ref{tag:16}) we arrive at 
\begin{align*}
&\gam\left(s+\frac{1}{2},\pi_1\otimes\Abs_F^{t_2+t_3},\addchar\right)Z(\calw_1^+,\calw_2^+,\calw_3^+, f_{\Phi^0}(\chi)) \\
=&\int_{F^\times}\int_F\calw_1^+(\bft(a)J_1^{-1}\bfn(x))|a|^{-s-t_2-t_3}\calf_{\addchar}(a,x)\,\d x\d^\times a, 
\end{align*}
where 
\begin{align*}
\calf_{\addchar}(a,x)&=(1+q^{-1})^{-1}\int_F\calf(\bfn^-(u)\bfn(x))\addchar(-au)\,\d u\\
&=q^{-1-t_2-t_3}(1+q^{-1})^{-3}\II_\frko(x)\int_{\frko^2}\II_{\frko}\left(\frac{a}{a_2a_3}\right)\prod_{i=2,3}|a_i|^{2s+2t_i}\d^\times a_i. 
\end{align*}
We conclude that 
\begin{align*}
&q^{1+t_2+t_3}(1+q^{-1})^3\gam\left(s+\frac{1}{2},\pi_1\otimes\Abs_F^{t_2+t_3},\addchar\right)Z(\calw_1^+,\calw_2^+,\calw_3^+, f_{\Phi^0}(\chi)) \\
=&\int_{F^\times}\int_F\d x\d^\times a\, \frac{\calw_1^+(\bft(a)J_1^{-1}\bfn(x))}{|a|^{s+t_2+t_3}}\II_\frko(x)\int_{\frko^2}\II_{\frko}\left(\frac{a}{a_2a_3}\right)\prod_{i=2,3}|a_i|^{2s+2t_i}\d^\times a_i\\
=&\int_{F^\times}\d^\times a\, \frac{\calw_1^+(\bft(a_2a_3a)J_1^{-1})}{|a_2a_3a|^{s+t_2+t_3}}\int_{\frko^2}\II_{\frko}(a)\prod_{i=2,3}|a_i|^{2s+2t_i}\d^\times a_i \\
=&\int_{\frko}\d^\times a\, \frac{W_1^+(\bft(a_2a_3a\vpi))}{|a|^{s+t_2+t_3}}\prod_{i=2,3}\int_\frko|a_i|^{s+t_i-t_{5-i}}\d^\times a_i\\
=&\frac{\zet\left(\frac{1}{2}-s+t_1-t_2-t_3\right)}{q^{(2t_1+1)/2}}\prod_{i=2,3}\zet\left(s+\frac{1}{2}+t_1+t_i-t_{5-i}\right).   
\end{align*}

Assume that $\pi_1\simeq\St$. 
Then $t_1=\frac{1}{2}$ and 
\[\gam\left(s+\frac{1}{2},\pi_1\otimes\Abs_F^{t_2+t_3},\addchar\right)=-q^{-s-t_2-t_3}\frac{\zet(1-s-t_2-t_3)}{\zet(s+1+t_2+t_3)}, \]
from which we complete our proof. 
\end{proof}

\begin{prop}\label{prop:21}
Let $\pi_i$ be either an unramified principal series representation or the Steinberg representation twisted by an unramified character. 
Set $W_i^{}=W_i^0$ in the former case and $W_i^{}=W_i^+$ in the latter case. 
Put $\breve W_i^{}=W_i^{}\otimes\ome_i^{-1}$. 
If not all $\pi_i$ are principal series, then for an unramified character $\chi$ of $F^\times$
\begin{align*}
Z(W_1,W_2,W_3, f_{\Phi^0}(\chi))&=Z(\breve W_1,\breve W_2,\breve W_3, f_{\Phi^0}(\chi\hat\ome))\\
&=-(\hat\ome^2\chi^4)(\vpi)q(1+q)^{-3}\frac{L\left(\frac{1}{2},\pi_1\times\pi_2\times\pi_3\otimes\chi\right)}{\vep\left(\frac{1}{2},\pi_1\times\pi_2\times\pi_3\otimes\chi,\addchar\right)}. 
\end{align*}
\end{prop}

\begin{Remark}\label{rem:localfactor}
If $\pi_1$ and $\pi_2$ are irreducible unramified principal series representations, then 
\begin{align*}
L(s,\pi_1\times\pi_2\times\St)&=L\left(s+\frac{1}{2},\pi_1\times\pi_2\right), & 
\vep(s,\pi_1\times\pi_2\times\St,\addchar)&=q^{-4s+2}\ome_1(\vpi)^2\ome_2(\vpi)^2,\\
L(s,\pi_1\times\St\times\St)&=L(s,\pi_1)L(s+1,\pi_1), &
\vep(s,\pi_1\times\St\times\St,\addchar)&=q^{-4s+2}\ome_1(\vpi)^2, \\
L(s,\St\times\St\times\St)&=\zet\left(s+\frac{3}{2}\right)\zet\left(s+\frac{1}{2}\right)^2, & 
\vep(s,\St\times\St\times\St,\addchar)&=-q^{-(10s-5)/2}.  
\end{align*}
\end{Remark}

\begin{proof}
In view of \cite[Lemma 3.1]{Ikeda89} and (\ref{tag:21}) we may assume that $\pi_1\simeq\St$ and $\pi_i$ is a quotient of $I(\Abs_F^{-t_i},\Abs_F^{t_i})$ for $i=2,3$. 
If all $\pi_i$ are discrete series representations, then since $\calw_1^+=-W_1$, the result follows from Lemma \ref{lem:21}. 
Let $\chi=\Abs_F^s$ and $\pi_3\simeq I(\Abs_F^{-t_3},\Abs_F^{t_3})$. 
Lemma \ref{lem:21} gives 
\[Z(W_1^+,\calw_2^+,\calw_3^\pm, f_{\Phi^0}(\chi)) 
=\frac{q^{s-2}}{(1+q^{-1})^3}L(s+1+t_2,\pi_3)\zet(s+1-t_2\pm t_3). \]
Thanks to Lemma \ref{lem:20} we obtain 
\begin{align*}
\frac{Z(W_1^+,\calw_2^+,W_3^0, f_{\Phi^0}(\chi))}{q^{s-2}L(s+1+t_2,\pi_3)}
&=q^{1/2}\frac{\zet(s+1-t_2+t_3)-\zet(s+1-t_2-t_3)}{(1+q^{-1})^3(q^{-t_3}-q^{t_3})}\\
&=(1+q^{-1})^{-3}q^{1/2}q^{-s-1+t_2}L(s+1-t_2,\pi_3). 
\end{align*}
If $\pi_2\simeq I(\Abs_F^{-t_2},\Abs_F^{t_2})$, then 
\[Z(W_1^+,\calw_2^+,W_3^0, f_{\Phi^0}(\chi))=(1+q^{-1})^{-3}q^{(2t_2-5)/2}L(s+1,\pi_2\times\pi_3), \]
and so again by Lemma \ref{lem:20},
\begin{align*}Z(W_1^+,W_2^0,W_3^0, f_{\Phi^0}(\chi))
&=(1+q^{-1})^{-3}L(s+1,\pi_2\times\pi_3)\frac{q^{t_2-2}-q^{-t_2-2}}{q^{-t_2}-q^{t_2}}\\
&=-(1+q^{-1})^{-3}q^{-2}L(s+1,\pi_2\times\pi_3). \end{align*}
If $\pi_2\simeq\St$, we obtain the claimed result by letting $t_2=\frac{1}{2}$.  
\end{proof}


\section{Computation of the local zeta integral: the archimedean case}\label{sec:arch}


\subsection{The archimedean sections}\label{ssec:realsection}

We define the sign character $\sgn:\R^\x\to\{\pm 1\}$ by $\sgn(x)=\frac{x}{|x|}$. 
Let $\Sym_n^+(\R)$ denote the set of positive definite symmetric matrices of rank $n$. 
The Siegel upper half-space $\frkH_n$ of degree $n$ consists of complex symmetric matrices of size $n$ with positive definite imaginary part. 
The Lie group $\GSp_{2n}^+(\R)=\{g\in\GSp_{2n}(\R)\;|\;\nu_n(g)>0\}$ acts on the space $\frkH_n$ by $g\cdot Z=(AZ+B)(CZ+D)^{-1}$, where $Z\in\frkH_n$ and $g=\begin{pmatrix} A & B \\ C & D \end{pmatrix}$ with matrices $A,B,C,D$ of size $n$. 
Let $C^\infty(\frakH_n)$ be the space of $\C$-valued smooth functions on the upper half complex plane $\frakH_n$. 
For an integer $k$ and $f\in C^\infty(\frkH_n)$ we define 
\begin{align}
f|_kg(Z)&=f(g\cdot Z)J(g,Z)^{-k}, & J(g,Z)&=\nu_n(g)^{-n/2}\det(CZ+D). \label{tag:modular}
\end{align}
Put $\bfi=\sqrt{-1}\bfone_n$. We shall identify the compact unitary group $\U(n)=\{u\in\GL_{n}(\C)\;|\;\bar u^{\rm t}u=\ono_n\}$ with the {stabilizer} $\{g\in \Sp_{2n}(\R)\;|\;g\cdot \bfi=\bfi\}$ of $\bfi$ via the map $g\mapsto\overline{J(g,\bfi)}$. 

We recall special functions $H_{ij}$ on $U(3)$ introduced in \cite[\S 1]{Ikeda98Crelle} for $1\leq i,j\leq 3$. For $u\in \U(3)$, we define $H_{ij}(u)$ to be the $(i,j)$-entry of the matrix $u^{\rm t} u$. 
By definition, $H_{ij}$ is a function on $\O(3)\bksl\U(3)$, and hence we can extend it to a unique function on $\GSp_6(\R)$ such that
\begin{align*}
H_{ij}(\bfn(z)\bfm(A,\nu)u)&=H_{ij}(u) & 
(z&\in\Sym_3(\R),\;A\in\GL_3(\R),\;\nu\in\R^\times,\;u\in\U(3)).
\end{align*}
\emph{A parity type} is a triplet $\gap=(\gap_1,\gap_2,\gap_3)$ of integers which belongs to one of the following triplets 
\[\gap\in\stt{(0,0,0),\; (0,1,1),\; (1,0,1),\; (1,1,2)}.\]
Note that $\lam_3=\lam_1+\lam_2$. 
Put
\[\calh_\gap:=\begin{cases}
1& \text{ if }\gap=(0,0,0),\\
\ol{H_{23}}& \text{ if }\gap=(0,1,1),\\
H_{12}&\text{ if }\gap=(1,0,1),\\
H_{12}\ol{H_{23}}&\text{ if }\gap=(1,1,2).
\end{cases}\]

Let $\chi_\infty$ be a character of $\R^\x$. 
For each integer $k$ and parity type $\lam$, we define the archimedean section $f_{\infty,s}^{[k,\gap]}\in I_3(\sgn^{k-\lam_1},\chi_\infty\sgn^{k-\lam_1}\Abs_\R^s)$ by \begin{align*}
&f^{[k,\gap]}_{s,\infty}(g):=\calh_\gap(g)\chi_\infty(\nu_3(g))\cdot J(g,\bfi)^{-k+\lam_1}\abs{J(g,\bfi)}^{k-\lam_1-2s-2}. \end{align*}
One verifies that
\begin{align*}
H_{ij}(g\iota(\kap_{\tht_1},\kap_{\tht_2},\kap_{\tht_3}))&=e^{\sqrt{-1}(\tht_i+\tht_j)}H_{ij}(g), & 
\kap_\tht&=\pMX{\cos\tht}{\sin\tht}{-\sin\tht}{\cos\tht}, 
\end{align*}
so we have 
\beq\label{tag:weight}
f^{[k,\gap]}_{s,\infty}(g\iota(\kap_{\tht_1},\kap_{\tht_2},\kap_{\tht_3}))
=f^{[k,\gap]}_{s,\infty}(g)e^{\sqrt{-1}\{k\tht_1+(k-\lam_2)\tht_2+(k-\lam_3)\tht_3\}}. 
\eeq
Namely the $\SO_2(\R)^3$-type of $f^{[k,\gap]}_{s,\infty}$ is $(k,k-\lam_2,k-\lam_3)$.

\begin{Remark}
Ikeda introduced a spherical function $H_\mu$ on $\O(3)\bsl\U(3)$ with highest weight $\mu$ in \cite{Ikeda98Crelle} and computed the zeta integral of the holomorphic section associated to a spherical function $H_{\mu|\kap}$ of weight $\kap=(k,l,m)$ which he constructed by applying differential operators on $\U(3)$ to $H_\mu$. 
However, it seems difficult to compute the Fourier coefficients of the Eisenstein series associated to this section. 
On the other hand, we start with the holomorphic section $f^{[k,\gap]}_{s,\infty}$ above. 
We will compute the associated degenerate Whittaker function in \S\S \ref{ssec:42}--\ref{ssec:44}, which shows that the pull-back of the associated Eisenstein series is nearly holomorphic. 
Lemma \ref{L:arch.E} below computes the zeta integral (\ref{tag:Aintegral}). 
Inspired by \cite{Mizumoto90}, we construct a triple product of nearly homomorphic modular forms of weight $\kap$ by applying weight lowering operators to this pull-back. 
Proposition \ref{P:1.E} below computes the Fourier expansion of the ordinary projection of its holomorphic projection. 
\end{Remark}


\subsection{Archimedean degenerate Whittaker functions}\label{ssec:42}

Recall that in \eqref{E:DegW.2}, the degenerate Whittaker function associated with the section $f^{[k,\gap]}_{s,\infty}$ and $B\in \Sym_3(\R)$ is given by 
\[\calw_B(g,f^{[k,\gap]}_{s,\infty})=\int_{\Sym_3(\R)}f^{[k,\gap]}_{s,\infty}(J_3\bfn(x)g)e^{-2\pi\sqrt{-1}\tr(B x)}\,\rmd x \quad (g\in \GSp_6(\R)). \]
{The Whittaker function $\calw_B(g,f^{[k,\gap]}_{s,\infty})$ will appear later as the archimedean contribution in the Fourier expansion of Siegel Eisenstein series (\cf\propref{P:pbkEC.E}). In this subsection, we give an expression of the value $\calw_B(\bfm(A),f^{[k,\gap]}_{s,\infty})$ for $A\in\GL_3(\R)$ in terms of the derivatives of the confluent hypergeometric functions.} For a positive integer $m$ we put 
\[\Gamma_m(s)=\pi^{m(m-1)/4}\prod_{j=0}^{m-1}\Gamma\left(s-\frac{j}{2}\right).\]
Following \cite[(3.6)]{Shimura82MA}, for $z\in\Sym_3^+(\R)$ and $(\al,\beta)\in\C^2$ we put  
\[\om(z;\al,\beta):=\frac{\det(z)^{\beta}}{\Gamma_3(\beta)}\int_{\Sym_3^+(\R)} 
e^{-\tr(zu)}\det(u+\ono_3)^{\alpha-2}(\det u)^{\beta-2}\,\rmd u.\]
Thanks to \cite[Theorem 3.1]{Shimura82MA}, this integral
is absolutely convergent for $\Re\beta>2$ and can be continued to a holomorphic function on $(\al,\beta)\in\C^2$. 
We shall consider the confluent hypergeometric function $\om^\star(z;\al,\beta)$ defined by \beq\label{E:zetaF.E}
\begin{aligned}
\om^\star(z;\alpha,\beta):&=\det(4\pi z)^{\al-2}\cdot \om(4\pi z;\al,\beta)\\
&=\frac{1}{\Gamma_3(\beta)}\int_{\Sym_3^+(\R)} e^{-\tr(u)}\det(u+4\pi z)^{\alpha-2}(\det u)^{\beta-2}\,\rmd u.
\end{aligned}
\eeq
{
\begin{lm}\label{L:7.E} 
Let $P(\hh)$ be a polynomial in $\hh=(\hh_{ij})\in\Sym_3(\R)$. We have 
\[\int_{\Sym_3^+(\R)} e^{-\tr(u)} P(u)\frac{(\det u)^{s-2}}{\Gamma_3(s)}\rmd u=P(-\partial_{ij})(\det \hh)^{-s}|_{\hh=\bfone_3}.\]
\end{lm}
\begin{proof}
If $T$ is positive definite and $\Re s>2$, then 
 \[\int_{\Sym_3^+(\R)} e^{-\tr(\hh u)}\frac{(\det u)^{s-2}}{\Gamma_3(s)}\, \rmd u=(\det\hh)^{-s}\]
 by \cite[(1.14)
 ]{Shimura81Duke}. 
 The declared formula follows immediately from the fact that
 \[P(-\partial_{ij})(e^{-\tr(\hh u)})=P(u)e^{-\tr(\hh u)}.\qedhere\]
\end{proof}
\begin{Remark}\label{R:1.E}
In view of \lmref{L:7.E}, we find that if $(\alpha,\beta)\in\Z^2$ with $\alp\geq 2$ and $\beta\leq 0$, then $\om^\star(z;\alpha,\beta)$ is a polynomial function in $z$ of degree at most $\alpha-2$. In particular, the definition of $\om^\star(z;\alpha,\beta)$ can be extended to arbitrary symmetric matrices $z\in \Sym_3(\R)$ by continuity if $(\al,\beta)\in\Z_{\geq 2}\times \Z_{\leq 0}$. We also remark that by the functional equation 
\[\om(z;2-\beta,2-\alpha)=\om(z;\al,\beta)\]
in \cite[(3.7)]{Shimura82MA}, 
the polynomial \[\om^\star(z;\alpha,\beta)=(\det 4\pi z)^{\alpha+\beta-2}\om^\star(z;2-\beta,2-\alpha)\] is divisible by $\det( z)^{\max\stt{\alpha+\beta-2,0}}$ in $\C[z]$. 
\end{Remark}}
\begin{lm}\label{L:H1.E}
For $x\in\Sym_3(\R)$ we have 
\begin{align*}
H_{23}(J_3\bfn(x))&=2\sqrt{-1}(x_{11}x_{23}-x_{12}x_{13}+\sqrt{-1}x_{23})/\det(x+\bfi),\\
H_{12}(J_3\bfn(x))&=2\sqrt{-1}(x_{12}x_{33}-x_{23}x_{13}+\sqrt{-1}x_{12})/\det(x+\bfi). 
\end{align*}
\end{lm}

\begin{proof} 
The Iwasawa decomposition of $J_3\bfn(x)$ can be written as
\begin{align*}
J_3\bfn(x)&=\pMX{z^{\rm t}}{*}{0}{z^{-1}}\pMX{zx}{-z}{z}{zx}, & 
z&\in \GL_3(\R) 
\end{align*}
with $z^{\rm t}z=(\ono_3+x^2)^{-1}$. 
Let $u=z(x-\bfi)\in\U(3)$. 
Then $u^{\rm t}u=(x-\bfi)(x+\bfi)^{-1}$. 
We denote the adjugate of a matrix $A\in\Mat_3(\R)$ by $\adj(A)$. 
Since $A\cdot\adj(A)=(\det A)\ono_3$, we have 
\[u^{\rm t}u=\det(x+\bfi)^{-1}(x-\bfi)\adj(x+\bfi)=-2\sqrt{-1}\det(x+\bfi)^{-1}\adj(x+\bfi)+\bfone_3. \]
By definition we find that
\begin{align*}
H_{23}(J_3\bfn(x))&=H_{23}(u)=\det(x+\bfi)^{-1}\cdot 2\sqrt{-1}\det\pMX{x_{11}+\sqrt{-1}}{x_{12}}{x_{13}}{x_{23}}\\
&=\det(x+\bfi)^{-1}\cdot 2\sqrt{-1}(x_{11}x_{23}-x_{12}x_{13}+\sqrt{-1}x_{23}). 
\end{align*}
One can compute $H_{12}(J_3\bfn(x))$ in the same way.
\end{proof}

\begin{defn}
We associate to a parity type $\gap$ the differential operator $\sD_\gap$ on $\hh=(\hh_{ij})\in\Sym_3(\R)$ by 
\begin{align*}
\sD_{(0,1,1)}&:=\frac{1}{2\pi^2\sqrt{-1}}\{\partial_{13}\partial_{12}-\partial_{23}(\partial_{11}-4\pi)\}, & 
\sD_{(0,0,0)}&=\mathrm{id},\\
\sD_{(1,0,1)}&:=\frac{1}{2\pi^2\sqrt{-1}}\{\partial_{12}\partial_{33}-\partial_{23}\partial_{13}\}, & 
\sD_{(1,1,2)}&:=\sD_{(0,1,1)}\sD_{(1,0,1)}.
\end{align*}
Here 
\[\partial_{ij}:=\frac{\partial}{\partial \hh_{ij}}\cdot \begin{cases} 1&\text{ if }i=j, \\
\frac{1}{2} &\text{ if }i\neq j.\end{cases} \]
\end{defn}

\begin{defn}\label{D:zeta.E} 
For each parity type $\gap$, writing $\boldsymbol\om^\star_s(T):=\om^\star(\hh,k-r,s)$ for $T\in\Sym_3(\R)$, we define 
\begin{align*}
\sD_\gap\om^\star(\hh,k-r,s)&:=\sD_\gap(\boldsymbol\om^\star_s)(T)\\
&=\frac{1}{\Gamma_3(s)}\int_{\Sym_3^+(\R)}e^{-\tr(u)}\sD_\lam\left(\det(4\pi \hh+u)^M\right)(\det u)^{s-2}\,\rmd u.
\end{align*}
\end{defn}
\begin{prop}\label{P:FC.E}Let $r$ be an integer with $\gap_2\leq r\leq k-2$. Let $A\in \GL_3(\R)^+$ and $B\in\Sym_3(\Q)$ with $\det B\neq 0$. 
If $B$ is positive definite, then  
\[\lim_{s\to \frac{k-\gap_1}{2}-r-1}\cW_B(\bfm(A),f^{[k,\gap]}_{s,\infty})=C^{[k,r,\gap]}_1e^{-2\pi\tr(A^{\rm t}BA)}\frac{\sD_\gap\om^\star\left(A^{\rm t}BA;k-r,\gap_2-r\right)}{(\det A)^{k-\lam_1-2r-4}}, \]
where 
\[C^{[k,r,\gap]}_1=(\sqrt{-1})^{k-\gap_2}\frac{2^{3(3+2r-k-\gap_2)}\pi^6}{\Gamma_3(k-r)}. \]
If $B$ is not positive definite, then 
\[\lim_{s\to\frac{k-\gap_1}{2}-r-1}\cW_B(\bfm(A),f^{[k,\gap]}_{s,\infty})=0. \]
\end{prop}
\begin{proof}
In view of the equation
\[\cW_B(\bfm(A),f^{[k,\gap]}_{s,\infty})=(\det A)^{-2s-2}\cW_{A^{\rm t}BA}(\ono_6,f^{[k,\gap]}_{s,\infty})=(\det A)^{-2s-2}\cW_{A^{\rm t}BA}(f^{[k,\gap]}_{s,\infty}), \]
we may assume $A=\ono_3$. By definition, 
\beq\label{E:P45.1}\cW_B(f^{[k,\gap]}_{s,\infty})=\int_{\Sym_3(\R)}\det(x+\bfi)^{-\al_0}\det (x-\bfi)^{-\beta_0} \calh_\gap(J_3\bfn(x))e^{-2\pi\sqrt{-1}\tr(B x)}\,\rmd x.\eeq
To proceed, we introduce a new set of differential operators $\cD_\gap$ on the space of smooth functions on $\Sym_3(\R)$ for each parity type $\gap$ defined by 
\begin{align*}
\cD_{(0,1,1)}&:=\frac{1}{2\pi^2\sqrt{-1}}\{\partial_{13}\partial_{12}-\partial_{23}(\partial_{11}-2\pi)\}, & 
\cD_{(0,0,0)}&:=\mathrm{id},\\
\cD_{(1,0,1)}&:=\frac{1}{2\pi^2\sqrt{-1}}\{\partial_{12}(\partial_{33}+2\pi)-\partial_{23}\partial_{13}\}, & 
\cD_{(1,1,2)}&:=\cD_{(0,1,1)}\cD_{(1,0,1)}.
\end{align*}
It is easy to verify that the actions of the two sets of differential operators $\sD_\lam$ and $\cD_\lam$ on polynomials $P$ on $\Sym_3(\R)$ are related by the following equation 
\beq\label{E:P45.2}\cald_\lam(e^{-2 \pi\tr(T)}P(T))=e^{-2 \pi\tr(T)}\sD_\lam P(T).\eeq
A direct computation combined with \lmref{L:H1.E} shows that
\[\cD_\gap(e^{-2\pi\sqrt{-1}\tr(Tx)})=\det(x+\bfi)^{\gap_1}\det(x-\bfi)^{\gap_2}\calh_\gap(J_3\bfn(x))e^{-2\pi\sqrt{-1}\tr(Tx)}. \]
Following \cite[(1.25)]{Shimura82MA}, for $(g,T)\in \Sym_3^+(\R)\times \Sym_3(\R)$ and $(\al,\beta)\in\C$, we put
\[\xi(g,T,\al,\beta)=\int_{\Sym_3(\R)}\det(x+\bfi g)^{-\al}\det (x-\bfi g)^{-\beta}e^{-2\pi\sqrt{-1}\tr(T x)}\,\rmd x.\]
Then \eqref{E:P45.1} can be rephrased as 
\begin{align*}
\cW_B(f^{[k,\gap]}_{s,\infty})=& \cD_\gap(\xi(\ono_3,\hh;\al_0+\gap_1,\beta_0+\gap_2))|_{\hh=B}
\end{align*}
with $\al_0=s+1+\frac{k-\gap_1}{2}$ and $\beta_0=s+1-\frac{k-\gap_1}{2}$. 
On the other hand, by \cite[(1.29)]{Shimura82MA} we have
\begin{align*}
\xi(\ono_3,T;\al,\beta)=(\sqrt{-1})^{3(\beta-\al)}\frac{(2\pi)^6e^{-2 \pi\tr(T)}}{2^3 \Gamma_3(\al)\Gamma_3(\beta)}\int_{u>0,\,u>-2\pi T}e^{-2\tr(u)}\det(u+2\pi T)^{\al-2}(\det u)^{\beta-2}\,\rmd u.\end{align*}
If $T\in\Sym_3^+(\R)$, then the last integral equals $2^{3(2-\al-\beta)}\om^\star(T;\al,\beta)\cdot\Gamma_3(\beta)$, and hence by \eqref{E:P45.2} we obtain 
\begin{align*}\lim_{s\to \frac{k-\gap_1}{2}-r-1}\cW_B(f^{[k,\gap]}_{s,\infty})&=\lim_{s\to \gap_2-r}C^{[k,r,\gap]}_1\cD_\lam\left(e^{-2\tr(T)}\om^\star(T;k-r,s)\right)|_{T=B}\\
&=C^{[k,r,\gap]}_1\cdot e^{-2\pi\tr(B)}\sD_\gap\om^\star\left(B;k-r,\gap_2-r\right). \end{align*}
This proves the proposition if $B$ is positive definite. If the signature of $T$ is $(p,q)$ with $q>0$ and $p+q=3$, then  \cite[Theorem 4.2]{Shimura82MA} shows that there exists a holomorphic function $\breve\om(\al,\beta)$ such that 
\[\xi(\ono_3,T;\al,\beta)=\frac{\Gamma_p\bigl(\beta-\frac{q}{2}\bigl)\Gamma_q\bigl(\alpha-\frac{p}{2}\bigl)}{\Gamma_3(\alpha)\Gamma_3(\beta)}\cdot \breve\om(\al,\beta).\]
This in particular implies that $\xi(\ono_3,T;k-r,\lam_2-r)=0$ for $\lam_2\leq r\leq k-2$, and hence \[\lim_{s\to \frac{k-\gap_1}{2}-r-1}\cW_B(f^{[k,\gap]}_{s,\infty})= \cD_\gap(\xi(\ono_3,\hh;k-r,\lam_2-r))|_{\hh=B}
=0\] if $B$ is not positive definite. 
\end{proof}


\subsection{The polynomial expansion of the Whittaker functions}\label{ssec:43}

Let \[y=\diag{y_1,y_2,y_3}\in\R_+^3;\quad A=\diag{\sqrt{y_1},\sqrt{y_2},\sqrt{y_3}}\in\GL_3(\R).\]   
For $B=(b_{ij})\in\Sym_3(\R)$, we define the Whittaker function $\bfW^{[k,r,\gap]}_B(y)$ by \[
\bfW^{[k,r,\gap]}_B(y):= (\det A)^{2r-2k+4+\lam_1}\sqrt{y_2}^{\gap_2}\sqrt{y_3}^{\gap_1+\gap_2}\cdot \scrd_\lam\om^\star(A^{\rm t}BA,k-r,\gap_2-r).\]
If $B$ is positive definite, then \propref{P:FC.E} yields
\beq\label{E:Winfty}
\lim_{s\to \frac{k-\lam_1}{2}-r-1}\cW_B(\bfm(A),f_{s,\infty}^{[k,\lam]})=C_1^{[k,r,\lam]}\cdot \frac{(\det A)^k}{\sqrt{y_2}^{\lam_2}\sqrt{y_3}^{\lam_1+\lam_2}}\bfW^{[k,r,\gap]}_B(y)\cdot e^{-2\pi (b_{11}y_1+b_{22}y_2+b_{33}y_3)}.
\eeq
We proceed to give some details about the polynomial expansion of $\bfW^{[k,r,\gap]}_B(y)$. Put \[M=k-r-2.\] Note that $M\geq 0$ and $\lam_2-r\leq 0$. In \remref{R:1.E}, we have seen  that $\om^\star(\hh;M+2,\lam_2-r)$ is a polynomial function in $\hh$ of degree at most $M$. We write
\begin{align*}
\om^\star(\hh;M+2,\lam_2-r)&=\sum_{0\leq j_1, j_2, j_3\leq M}c_{j_1j_2j_3}\hh_{12}^{j_3}\hh_{23}^{j_1}\hh_{13}^{j_2}, & 
c_{j_1j_2j_3}&\in \C[\hh_{11},\hh_{22},\hh_{33}], 
\end{align*}  
where $\hh=(\hh_{ij})\in\Sym_3(\R)$. 
Since 
\begin{align*}
\om^\star(\vep^{\rm t} T\vep;\al,\beta)&=\ome^\star(T;\al,\beta), & 
\vep&=\diag{-1,1,1}
\end{align*} 
in view of the expression (\ref{E:zetaF.E}), we get $c_{j_1j_2j_3}=(-1)^{j_2+j_3}c_{j_1j_2j_3}$. 
Thus $c_{j_1j_2j_3}=0$ unless $j_2\equiv j_3\pmod 2$. 
By symmetry we conclude that $c_{j_1j_2j_3}=0$ unless $j_1\equiv j_2\equiv j_3\pmod 2$. 
On the other hand, it is also explained in \remref{R:1.E} that the polynomial $\om^\star(\hh;M+2,\lam_2-r)$ is divisible by $\det(T)^{\max\stt{M-r+\lam_2,0}}$. We thus obtain
\begin{align*}&\left(T_{23}^{\lam_1}T_{12}^{\lam_2}T_{13}^{\lam_1+\lam_2}\sD_\gap\om^\star(\hh,M+2,\lam_2-r)\right)|_{T=A^{\rm t} BA}\\
=&\sum_{\max\stt{M-r+\lam_2,0}\leq j_1,j_2,j_3\leq M} a_{j_1j_2j_3}y_1^{j_1}y_2^{j_2}y_3^{j_3}\in\C[y_1,y_2,y_3].\end{align*}
This shows that \[\bfW^{[k,r,\gap]}_B(y)=(y_1y_2y_3)^{-M}y_1^{-\lam_2}\cdot \left(T_{23}^{\lam_1}T_{12}^{\lam_2}T_{13}^{\lam_1+\lam_2}\sD_\gap\om^\star(\hh,M+2\lam_2-r)\right)|_{T=A^{\rm t} BA}\] is a polynomial in $\C[y_1^{-1},y_2^{-1},y_3^{-1}]$ of the form
\beq\label{E:FC.E}
\bfW^{[k,r,\gap]}_B(y)=\sum_{0\leq a,b,c\leq \min\stt{r-\lam_2,k-r-2}}Q_{a,b,c}^{[k,\gap]}(B,r)y_1^{-a-\lam_2}y_2^{-b}y_3^{-c}, 
\eeq
and each coefficient $Q_{a,b,c}^{[k,\gap]}(B,r)$ is a polynomial in $B$. 


\subsection{The coefficients of $\bfW^{[k,r,\gap]}_B(y)$}\label{ssec:44}
Put
\[\Sym_3^0(\R):=\left\{B\in \Sym_3(\R)\;\Biggl|\; B=\begin{pmatrix}0&b_3&b_2\\
b_3&0&b_1\\
b_2&b_1&0
\end{pmatrix},\,\det B\neq 0\right\} \]
We shall determine explicitly the coefficient $Q_{0,b,c}^{[k,\gap]}(B,r)$ of $\bfW^{[k,r,\gap]}_B(y)$ in the expression \eqref{E:FC.E} for $B\in\Sym_3^0(\R)$. Note that $B\in\Sym_3^0(\R)$ is never to be positive definite, so $\cW_B(\bfm(A),f_{s,\infty}^{[k,\lam]})|_{s=\frac{k-\lam_1}{2}-r-1}=0$ by \propref{P:FC.E}. However, we will see later in \propref{P:1.E} that these coefficients of $\bfW^{[k,r,\gap]}_B(y)$ with $B\in\Sym_3^0(\R)$ surprisingly contribute to the Fourier expansion of the ordinary and holomorphic part of the pull-back of our Eisenstein series due to the miraculous effect of the ordinary projector. 

We begin with a lemma. Let $\lam$ be a parity type. For every non-negative integer $M$ and $(T,u)\in \Sym_3(\R)\times \Sym_3(\R)$, we put 
\[\bfK_{\gap}^M(\hh;u):=\sD_\gap(\det(4\pi \hh+u)^M)\in \C[T].\]
Then \defref{D:zeta.E} can be rephrased as 
\beq\label{E:P45.3}\sD_\gap\om^\star(T,M+2,s)=\frac{1}{\Gamma_3(s)}\int_{\Sym_3^+(\R)} e^{-\tr(u)}\bfK_\lam^{M}(T;u)(\det u)^{s-2}\,\rmd u.\eeq
Let $\cY$ be the matrix with variables $Y_1,Y_2,Y_3$ given by  
\[\cY=\begin{pmatrix}0&\sqrt{Y_1Y_2}&\sqrt{Y_1Y_3}\\
\sqrt{Y_1Y_2}&0&\sqrt{Y_2Y_3}\\
\sqrt{Y_1Y_3}&\sqrt{Y_2Y_3}&0
\end{pmatrix}. \]
For two functions $f,g:\R_+\to\C$ and $c\in \R$ we say that $f(y)=g(y)+\bigO(y^c)$ if $\lim_{y\to\infty}\frac{f(y)-g(y)}{y^c}=0$.
\begin{lm}\label{L:53.E}
The polynomial $\bfK_{\gap}^M\bigl((4\pi)^{-1}\cY;u\bigl)\in\C[\sqrt{Y_1},\sqrt{Y_2},\sqrt{Y_3},u]$ has the expression
\begin{align*}
\bfK_{\gap}^M((4\pi)^{-1}\cY;u)&=C^{[k,r,\gap]}_2\bfc_\gap(Y_2,Y_3;u)\cdot Y_1^{M-\frac{\gap_1}{2}}+\bigO(Y_1^{M-\frac{\gap_1}{2}})
\end{align*}
with $C^{[k,r,\gap]}_2\in\C$ and $\bfc_\gap(Y_2,Y_3;u)\in \C[\sqrt{Y_2},\sqrt{Y_3},u]$ give by  
\begin{align*}
C^{[k,r,\gap]}_2&=\frac{(2M+\gap_1)!}{(2M)!}\cdot\frac{2^{3(\gap_1+\gap_2)-\gap_1}M!}{(\sqrt{-1})^{\gap_2-\gap_1}(M-\gap_1-\gap_2)!}, \\
\bfc_\gap(Y_2,Y_3;u)&=\left(-u_{22}Y_3-u_{33}Y_2+2Y_2Y_3+2u_{23}\sqrt{Y_2Y_3}\right)^{M-\gap_1-\gap_2}\cdot \sqrt{Y_2}^{\gap_1+\gap_2}\sqrt{Y_3}^{\gap_2}.
\end{align*}
\end{lm}

\begin{proof}
This is proved by a direct computation. 
Note that 
\begin{align*}
\partial_{11}\det(\hh+u)=&(\hh_{22}+u_{22})(\hh_{33}+u_{33})-(\hh_{23}+u_{23})^2,\\
\partial_{12}\det(\hh+u)=&-(\hh_{12}+u_{12})(\hh_{33}+u_{33})+(\hh_{23}+u_{23})(\hh_{13}+u_{13}),\\
\partial_{13}\det(\hh+u)=&-(\hh_{13}+u_{13})(\hh_{22}+u_{22})+(\hh_{12}+u_{12})(\hh_{23}+u_{23}),\\
\partial_{23}\det(\hh+u)=&-(\hh_{23}+u_{23})(\hh_{11}+u_{11})+(\hh_{12}+u_{12})(\hh_{13}+u_{13}),\\
\partial_{33}\det(\hh+u)=&(\hh_{11}+u_{11})(\hh_{22}+u_{22})-(\hh_{12}+u_{12})^2.
\end{align*}

Put 
\begin{align*}
\Delta&=\det(\hh+u), & 
R&=(-u_{22}Y_3-u_{33}Y_2+2Y_2Y_3+2u_{23}\sqrt{Y_2Y_3})Y_1. 
\end{align*} 
Since $\Delta|_{\hh=\cY}=R+\bigO(Y_1)$, we have 
\begin{align*}
&\bfK_{(0,1,1)}^M((4\pi)^{-1}\cY;u)\\
=&(2\pi^2\sqrt{-1})^{-1}\cdot (4\pi)^2\{\partial_{13}\partial_{12}-\partial_{23}(\partial_{11}-1)\}\Delta^M|_{\hh=\cY}\\
\equiv &-8\sqrt{-1}[M(M-1)R^{M-2}(\partial_{13}\Delta\partial_{12}\Delta-\partial_{23}\Delta\partial_{11}\Delta)+MR^{M-1}\{\partial_{13}\partial_{12}-\partial_{23}(\partial_{11}-1)\}\Delta]|_{\hh=\cY}\\
\equiv &-8\sqrt{-1}MR^{M-1}\partial_{23}\Delta|_{\hh=\cY}\\
\equiv & -8\sqrt{-1}MR^{M-1}\sqrt{Y_2Y_3}Y_1\pmod{\bigO(Y_1^M)}, 
\end{align*}
which verifies the case $\gap=(0,1,1)$. 
When $\gap=(1,0,1)$, we have 
\begin{align*}
&\bfK_{(1,0,1)}^M((4\pi)^{-1}\cY;u)\\
\equiv &-8\sqrt{-1}\{M(M-1)R^{M-2}(\partial_{12}\Delta\partial_{33}\Delta-\partial_{13}\Delta\partial_{23}\Delta)+MR^{M-1}(\partial_{12}\partial_{33}\Delta-\partial_{13}\partial_{23}\Delta)\}|_{\hh=\cY}\\
\con &-8\sqrt{-1}\biggl\{M(M-1)R^{M-2}(-R\sqrt{Y_1Y_2})+MR^{M-1}\left(-\frac{3}{2}\sqrt{Y_1Y_2}\right)\biggl\}\\
\con &4\sqrt{-1}M(2M+1)R^{M-1}\sqrt{Y_1Y_2} \pmod{\bigO(Y_1^{M-\frac{1}{2}})}
\end{align*}
as claimed. 
Since 
\[\sD_{(0,1,1)}\Delta^M|_{T=\caly}=-8\sqrt{-1}M\Delta^{M-1}T_{12}T_{13}|_{T=\caly}+\bigO(Y_1^M), \]
we have  
\begin{align*}
\bfK_{(1,1,2)}^M((4\pi)^{-1}\cY;u)
\equiv & 32M(M-1)(2M-1)R^{M-2}\sqrt{Y_1Y_2}Y_1\sqrt{Y_2Y_3}\\
&-64 M(M-1)\Delta^{M-2}(T_{13}\partial_{33}\Delta-T_{12}\partial_{23}\Delta)|_{\hh=\caly}\pmod{\bigO(Y_1^{M-\frac{1}{2}})}, 
\end{align*}
which proves the case $\lam=(1,1,2)$. 
\end{proof}

Let $(k,l,m)$ be a triplet of integers such that $k\geq l\geq m\geq 2$. We say that the triplet $(k,l,m)$ has the parity type $\gap$ if 
\beq\label{E:parity}\begin{aligned}
\gap_1,\gap_2&\in\{0,1\}, & 
\gap_1&\con l-m\pmod{2}, & 
\gap_2&\con k-l\pmod{2}, & 
\gap_3&=\gap_1+\gap_2. 
\end{aligned}\eeq
The following explicit formula of the coefficient $Q_{0,b,c}^{[k,\gap]}(B,r)$ is a key ingredient  in the proof of \propref{P:1.E}.
\begin{lm}\label{L:coeff.E}
Let $\gap$ be the parity type of $(k,l,m)$. {Suppose that \[k\leq l+m-1.\]} Let $r$ be an integer such that \[\lam_2\leq k-\frac{l+m+\lam_1}{2}\leq r\leq \frac{l+m}{2}-2\leq k-2.\]  
Put 
\begin{align*}
M&=k-r-2, & 
b&=\frac{1}{2}(k-l-\gap_2), & 
c&=\frac{1}{2}(k-m-\gap_3), & 
n&=M+\frac{1}{2}(l+m-\lam_1). 
\end{align*} 
If 
\[B=\begin{pmatrix}0&b_3&b_2\\
b_3&0&b_1\\
b_2&b_1&0
\end{pmatrix}\in\Sym_3^0(\R)\] has zero diagonal entries, then we have 
\[Q_{0,b,c}^{[k,\gap]}(B,r)=w_{0,b,c}\cdot (b_1b_2b_3)^nb_1^{-k}b_2^{-l}b_3^{-m},\]
where   
\[w_{0,b,c}=(4\pi)^{3M-b-c-2\gap_1-\gap_2}2^{M+\lam_1+2\lam_2-b-c}\frac{(\sqrt{-1})^{\gap_1-\gap_2}(2M+\gap_1)!M!}{(2M)!(M-\lam_1-\lam_2-b-c)!}\frac{(r-\gap_2)!}{b!c!(r-\gap_2-b-c)!}. \]
\end{lm}

\begin{proof}
Substitute $Y_i=4\pi\frac{b_1b_2b_3}{b_i^2}y_i$ into the matrix $\cY$. 
Then $\cY=4\pi A^{\rm t}BA$ and 
\[\bfW^{[k,r,\gap]}_B(y)=\sum_{a,b,c}Q_{a,b,c}^{[k,\gap]}(B,r)(4\pi)^{a+b+c}Y_1^{-a}Y_2^{-b}Y_3^{-c} \cdot b_1^{b+c-a}b_2^{a+c-b}b_3^{a+b-c}. \]
 On the other hand, by definition we have 
\[\bfW^{[k,r,\gap]}_B(y)=\left(\frac{(4\pi)^3b_1b_2b_3}{Y_1Y_2Y_3}\right)^M\frac{\sqrt{Y_1}^{\gap_1}\sqrt{Y_2}^{\gap_1+\gap_2}\sqrt{Y_3}^{2\gap_1+\gap_2}}{(4\pi)^{2\gap_1+\gap_2}b_1^{\gap_1+\gap_2}b_2^{\gap_1}}{\sD_\gap}\om^\star((4\pi)^{-1}\cY,M+2,\gap_2-r)\]
with $M=k-r-2$. From these equations, we find that 
\[Q_{0,b,c}^{[k,\gap]}(B,r)=w_{0,b,c}\cdot (b_1b_2b_3)^nb_1^{-k}b_2^{-l}b_3^{-m}\text{ for some }w_{0,b,c}\in\C. \]
Our task is to determine $w_{0,b,c}$. In view of \eqref{E:P45.3} and \lmref{L:53.E}, we see that $w_{0,b,c}$ is the coefficient of $Y_2^{M-\gap_1-\gap_2-b}Y_3^{M-\gap_1-\gap_2-c}$ in the polynomial
  \begin{align*}
  &\frac{(4\pi)^{3M-b-c-2\gap_1-\gap_2}}{\sqrt{Y_2}^{\gap_1+\gap_2}\sqrt{Y_3}^{\gap_2}}\int_{\Sym_3^+(\R)}e^{-\tr(u)}C^{[k,r,\gap]}_2\bfc_\gap(Y_2,Y_3;u)\frac{(\det u)^{s-2}}{\Gamma_3(s)}\,\rmd u\biggl|_{s=\gap_2-r}\\
 =& (4\pi)^{3M-b-c-2\gap_1-\gap_2}C^{[k,r,\gap]}_2\\
 &\quad\times\int_{\Sym_3^+(\R)}e^{-\tr(u)}\left(-u_{22}Y_3-u_{33}Y_2+2Y_2Y_3+2u_{23}\sqrt{Y_2Y_3}\right)^{M-\gap_1-\gap_2}\frac{(\det u)^{s-2}}{\Gamma_3(s)}\,\rmd u|_{s=\gap_2-r} 
  \end{align*}
  Here we have used \lmref{L:7.E} in the above equality. Put 
\begin{align*}
L&=M-\gap_1-\gap_2, & 
r_1&=r-\gap_2. 
\end{align*}
Notice that $b\leq c$ by assumption. 
The coefficient of $Y_2^{L-b}Y_3^{L-c}$ in the last integral is given by 
\begin{align*}
&\sum_{i=0}^b\frac{2^{L-b-c}(-1)^{b+c}\cdot L!}{(b-i)!(c-i)!(L-b-c)!(2i)!}\int_{\Sym_3^+(\R)}e^{-\tr(u)}u_{33}^{b-i}u_{22}^{c-i}(2u_{23})^{2i}\frac{(\det u)^{s-2}}{\Gamma_3(s)}\,\rmd u|_{s=-r_1}
\\
=&\sum_{i=0}^b\frac{2^{L-b-c}\cdot L!2^{2i}}{(b-i)!(c-i)!(L-b-c)!(2i)!} \partial_{33}^{b-i}\partial_{22}^{c-i}\partial_{23}^{2i}(\hh_{22}\hh_{33}-\hh_{23}^2)^{r_1}|_{\hh_{22}=\hh_{33}=1,\hh_{23}=0}\\
=&\sum_{i=0}^b\frac{2^{L-b-c}\cdot L!2^{2i}}{(b-i)!(c-i)!(L-b-c)!(2i)!}\sum_{j=0}^{r_1} {r_1\choose j}\partial_{33}^{b-i}\partial_{22}^{c-i}\partial_{23}^{2i}\hh_{22}^{r_1-j}\hh_{33}^{r_1-j}(-\hh_{23}^2)^j|_{\hh_{22}=\hh_{33}=1,\hh_{23}=0}\\
=&\sum_{i=0}^b\frac{2^{L-b-c}\cdot L!}{(b-i)!(c-i)!(L-b-c)!} {r_1\choose i}(-1)^i\partial_{33}^{b-i}\partial_{22}^{c-i}(\hh_{22}^{r_1-i}\hh_{33}^{r_1-i})|_{\hh_{22}=\hh_{33}=1}\\
=&\frac{2^{L-b-c}\cdot L!}{(L-b-c)!}\sum_{i=0}^b {r_1\choose i}{r_1-i\choose b-i}{r_1-i\choose c-i}(-1)^i
\end{align*}
in view of \lmref{L:7.E}. 
The last summation equals
\[\frac{r_1!}{(r_1-b)!b!}\sum_{i=0}^b{b\choose i}{r_1-i\choose r_1-c}(-1)^i
=\frac{r_1!}{(r_1-b)!b!}\cdot {r_1-b\choose r_1-b-c}
=\frac{r_1!}{b!c!(r_1-b-c)!}, \]
where we can deduce this equality by equating the terms of degree $r_1-c$ of the identity 
\[\sum_{i=0}^b{b\choose i}(1+X)^{r_1-i}(-1)^i=(1+X)^{r_1}\left(1-\frac{1}{1+X}\right)^b=(1+X)^{r_1-b}X^b,\]
Finally, we see that $w_{0,b,c}$ equals
\[(4\pi)^{3M-b-c-2\gap_1-\gap_2}C^{[k,r,\gap]}_2\frac{2^{M-\lam_1-\lam_2-b-c}\cdot (M-\lam_1-\lam_2)!}{(M-\lam_1-\lam_2-b-c)!}\frac{(r-\gap_2)!}{b!c!(r-\gap_2-b-c)!} \]
by putting together the above computations, which completes our proof. 
\end{proof}


\subsection{The archimedean zeta integral}\label{ssec:arhizeta}

Let $V_\pm$ be the weight raising/lowering operator given by 
\[V_\pm:=\frac{1}{(-8\pi)}\left(\pDII{1}{-1}\ot 1\pm\pMX{0}{1}{1}{0}\ot\sqrt{-1}\right)\in\Lie(\GL_2(\R))\ot_\R\C.\]
For each integer $k\geq 2$ we denote by $\sigma_k$ the discrete series of  $\GL_2(\R)$ of the minimal weight $k$ and by $W_k$ the Whittaker function of $\sigma_k$ characterized by 
\[W_k(\diag{y,1})=y^{k/2}e^{-2\pi y}\bbI_{\R_+}(y). \]
Set $W_k^{[t]}=V_+^{t}W_k$. 
It follows from (\ref{tag:MaassShimura}) below that
\beq\label{E:41.E}
W_k^{[t]}(\diag{y,1})=\sum_{j=0}^t (-4\pi)^{j-t}{t\choose j}\frac{\Gamma(t+k)}{\Gamma(j+k)}\cdot y^{\frac{k}{2}+j}e^{-2\pi y}\bbI_{\R_+}(y). 
\eeq
Let $(k,l,m)$ be the triplet of integers as in the previous subsection. 
{In the literature, the various archimedean zeta integral of Garrett, Piateski-Shapiro and Rallis have been evaluated for different purposes  (\cf \cite{Garrett}, \cite{Orl87}, \cite{GH93AJM}, \cite{Ikeda98Crelle,Ikeda99} and {\cite{BP2}}). When the parity type $\lam=(0,0,0)$, \ie $k\con l\con m\pmod{2}$, the integral $Z_\infty(s)$ is essentially computed in \cite{Orl87}. We shall extend Orloff's computation to general parity types.}
Put $\cJ_\infty=\pDII{-1}{1}$. 
We will evaluate the archimedean zeta integral in \eqref{E:GPSR} with the above choices of sections and Whittaker functions. We put 
{\beq
\textcolor{black}{Z_\infty(s):=Z\biggl(\rho(\cJ_\infty)W_k,\rho(\cJ_\infty)W^{\bigl[\frac{k-l-\lam_2}{2}\bigl]}_l,\rho(\cJ_\infty)W^{\bigl[\frac{k-m-\lam_3}{2}\bigl]}_m,f^{[k,\gap]}_{s,\infty}\biggl),} \label{tag:Aintegral}
\eeq}
where $\gap=(\lam_1,\lam_2,\lam_3)$ is the parity type of $(k,l,m)$. 
Put 
\[\gamma^\star_{(k,m,l)}(s)=(\sqrt{-1})^{k+2\gap_2+\gap_1}\frac{\Gamma\bigl(s+\frac{k-m-l}{2}+1\bigl)}{\Gamma\bigl(s-\frac{k-\gap_1}{2}+\gap_2+1\bigl)}\cdot\frac{\Gamma\bigl(s+\frac{k+\gap_1}{2}\bigl)}{\Gamma\bigl(s+\frac{k+\gap_1}{2}+1\bigl)} \cdot\frac{\pi^{3s+1}(4\pi)^{l+m-\frac{k-\gap_1}{2}+\gap_2}}{4\Gamma\bigl(s+\frac{m+l-k}{2}\bigl)\Gamma(2s+k)}. \]

{We remark that the special value $\gamma^\star_{(k,m,l)}(\frac{k-\lam_1}{2}-r-1)$ is very similar to $w_{0,b,c}$ if $k<l+m-1$. This point is crucial in the normalization of Eisenstein series and will be explicated in (\ref{tag:equality}). Recall that $L(s,\sigma_k\times\sigma_l\times\sigma_m)$ equals
\[\Gamma_\C\left(s+\frac{k+l+m-3}{2}\right)\Gamma_\C\left(s+\frac{k-l+m-1}{2}\right)\Gamma_\C\left(s+\frac{k+l-m-1}{2}\right)\Gamma_\C\left(s+\frac{m+l-k-1}{2}\right) \]
if $k<l+m$, and 
\[\Gamma_\C\left(s+\frac{k+l+m-3}{2}\right)\Gamma_\C\left(s+\frac{k-l+m-1}{2}\right)\Gamma_\C\left(s+\frac{k+l-m-1}{2}\right)\Gamma_\C\left(s+\frac{k-l-m+1}{2}\right) \]
if $k\geq l+m$. 
When $k\geq l+m$, the (unbalanced) critical points are $s=n-\frac{k+l+m-3}{2}$ with $n\in\Z$, $l+m-1\leq n\leq k-1$. 
It is important to note that $\gam^\star_{(k,m,l)}\bigl(s-\frac{1}{2}\bigl)$ has a zero at these points, which is one of reasons for our restriction to the balanced critical points.  
}
\begin{lm}\label{L:arch.E}
Let $\gap$ be the parity type of $(k,l,m)$.
Then 
\[Z_\infty(s)=\chi_\infty(-1)
\cdot\frac{\gamma^\star_{(k,m,l)}(s)}{2^{5+(k+m+l)}}L\left(s+\frac{1}{2},\sigma_k\times\sigma_l\times\sigma_m\right)\times
\begin{cases} 
1 &\text{if $k<l+m$, } \\
\frac{\Gam_\C\left(s+\frac{l+m-k}{2}\right)}{\Gam_\C\left(s+\frac{k-l-m}{2}+1\right)} &\text{if $k\geq l+m$. }
\end{cases}\]
\end{lm}

\begin{proof} 
For $a=(a_1,a_2,a_3)\in \R_+^3$ and $x\in \R$, we set
\begin{align*}
t(a)&=\mathrm{diag}\left(\pDII{a_1}{a_1^{-1}},\pDII{a_2}{a_2^{-1}},\pDII{a_3}{a_3^{-1}}\right), \\
u(x)&=\mathrm{diag}\left(\pMX{1}{\frac{x}{3}}{0}{1},\pMX{1}{\frac{x}{3}}{0}{1},\pMX{1}{\frac{x}{3}}{0}{1}\right). 
\end{align*}
When $x\neq 0$, the Iwasawa decomposition of $\Eta \iota(u(x)t(a))$ can described as follows:
Put 
\[P=\begin{pmatrix}
a_1^2&a_1a_2&a_1a_3\\
a_1a_2&a_2^2&a_2a_3\\
a_1a_3&a_2a_3&a_3^2
\end{pmatrix}. \]
We write $\Eta \iota(u(x)t(a))=\bfn({b})\bfm(A)\bfu$ with ${b}\in \Sym_3(\R)$, $A\in\GL_3(\R)$ and $\bfu=\begin{pmatrix} D & -C \\ C & D\end{pmatrix}\in\U(3)$. 
Since $D^{-1}C=x^{-1}P$, we can choose $U\in\GL_3(\R)$ so that 
\begin{align*}
U^{\rm t}U&=(x^2\bfone_3+P^2)^{-1}, & 
\bfu&=\pMX{Ux}{-UP}{UP}{Ux}\in\U(3). 
\end{align*} 
Put $u=Ux-\sqrt{-1}UP$. 
Then $u^{\rm t}u=(x\ono_3-P\sqrt{-1})(x\ono_3+P\sqrt{-1})^{-1}$. 
By direct computations we get 
\begin{align*}
\det A&=a_1a_2a_3\abs{z}^{-1}, & 
\det u&=\frac{\bar z}{|z|}, & 
H_{23}(u)&=-2\sqrt{-1} \frac{a_2a_3}{z}, & 
H_{12}(u)&=-2\sqrt{-1}\frac{a_1a_2}{z},
\end{align*}
where {$z=x+\sqrt{-1}(a_1^2+a_2^2+a_3^2)$} (see \cite[(6.7), (6.8)]{GK92}). 
Put 
\begin{align*}
b&=\frac{1}{2}(k-l-\gap_2), & 
c&=\frac{1}{2}(k-m-\gap_3), & 
(\bfs,\bfk,\bfl,\bfm)&=\left(s+\frac{\gap_3}{2},k-\gap_2,l,m-\gap_1\right). 
\end{align*}
{Recall that $\hat\ome=\sgn^{k-\lam_1}$. }
It follows that 
\begin{align*}
f^{[k,\gap]}_{s,\infty}(\Eta\iota(u(x)t(a)\bfd(-1)))
&=\chi_\infty(-1)(-1)^{k-\lam_1}\left(2\sqrt{-1}\frac{a_1a_2}{\bar z}\right)^{\gap_1}\left(-2\sqrt{-1}\frac{ a_2a_3}{z}\right)^{\gap_2}\left(\frac{a_1a_2a_3}{|z|}\right)^{2s+2}\left(\frac{z}{|z|}\right)^{k-\gap_1}\\
&=\chi_\infty(-1)2^{\gap_1+\gap_2}\sqrt{-1}^{\gap_2-\gap_1}(a_1a_2a_3)^{2\bfs+2}a_1^{-\gap_2}a_3^{-\gap_1}|z|^{-2\bfs-2-\bfk}(-z)^\bfk. 
\end{align*}
From \eqref{tag:weight} and \eqref{E:41.E}, we see that $2Z_\infty(s)$ equals 
\begin{align*}
&\hat\om(-1)\int_\R\int_{\R_+^3}W_k(\bfn(x/3)\bfm(a_1))W^{[b]}_l(\bfn(x/3)\bfm(a_2))W^{[c]}_m(\bfn(x/3)\bfm(a_3))f^{[k,\gap]}_{s,\infty}(\Eta\iota( u(x)t(a)\bfd(-1)))\,\rmd x\prod_{j=1}^3\frac{\rmd ^\x a_j}{\abs{a_j}^2}\\
=&(\chi_\infty\hat\om)(-1)2^{\gap_1+\gap_2}\sqrt{-1}^{\gap_2-\gap_1}(-4\pi)^{-b-c}\sum_{A=0}^\infty\sum_{B=0}^\infty(-4\pi)^{A+B}\binom{b}{A}\binom{c}{B}\frac{\Gamma(l+b)}{\Gamma(l+A)}\frac{\Gamma(m+c)}{\Gamma(m+B)}\\
&\times\int_\R\int_{\R_+^3}a_1^{2\bfs+\bfk}a_2^{2\bfs+\bfl+2A}a_3^{2\bfs+\bfm+2B}\abs{z}^{-2\bfs-2-\bfk}(-z)^{\bfk}e^{2\pi \sqrt{-1}x}e^{-2\pi(a_1^2+a_2^2+a_3^2)}\,\rmd x\prod_{j=1}^3\rmd^\x a_j. 
\end{align*}

Put $\al=\bfs+1+\frac{\bfk}{2}$ and $\beta=\bfs+1-\frac{\bfk}{2}$. 
The last integral equals
\[\frac{(-2\pi\sqrt{-1})^\al(2\pi \sqrt{-1})^\beta}{\Gamma(\al)\Gamma(\beta)}\int_{\R_+^4} a_1^{\bfk+2\bfs}a_2^{\bfl+2A+2\bfs}a_3^{\bfm+2B+2\bfs}\frac{(1+t)^{\al-1}t^{\beta-1}}{e^{4\pi(a_1^2+a_2^2+a_3^2)(1+t)}}\rmd t\prod_{j=1}^3\rmd^\x a_j. \]
We here use the identity 
\[\int_\R \frac{e^{-2\pi\sqrt{-1}x}\,\rmd x}{(x+\sqrt{-1}y)^\alp(x-\sqrt{-1}y)^\bet}=\frac{(-2\pi\sqrt{-1})^\alp(2\pi\sqrt{-1})^\bet}{\Gam(\alp)\Gam(\bet)}\int_{\R_+}\frac{(t+1)^{\alp-1}t^{\bet-1}}{e^{2\pi y(1+2t)}}\,\rmd t \]
(see \cite[(6.11)]{GK92}). 
The quadruple integral above equals  
\begin{align*}
&\frac{1}{(4\pi)^{\frac{\bfk+\bfl+\bfm}{2}+3\bfs+A+B}}
\int_{\R_+^4}\frac{a_1^{\bfk+2\bfs}a_2^{\bfl+2A+2\bfs}a_3^{\bfm+2B+2\bfs}(1+t)^{\al-1}t^{\beta-1}}{e^{a_1^2+a_2^2+a_3^2}(1+t)^{\frac{\bfk+\bfl+\bfm}{2}+3\bfs+A+B}}\rmd t\prod_{j=1}^3\rmd^\x a_j\\
=&\frac{\Gamma\bigl(\frac{\bfk}{2}+\bfs\bigl)\Gamma\bigl(\frac{\bfl}{2}+\bfs+A\bigl)\Gamma\bigl(\frac{\bfm}{2}+\bfs+B\bigl)}{2^3(4\pi)^{\frac{\bfk+\bfl+\bfm}{2}+3\bfs+A+B}}\int_{\R_+} 
\frac{(1+t)^{\al-1}t^{\beta-1}}{(1+t)^{\frac{\bfk+\bfl+\bfm}{2}+3\bfs+A+B}}\rmd t. 
\end{align*}
Recall that 
\begin{align*}
\int_0^\infty\frac{(1+t)^{\al-1}t^{\beta-1}}{(1+t)^{\frac{\bfk+\bfl+\bfm}{2}+3\bfs+A+B}}\rmd t
=&B\left(\bet,1-\alp-\bet+\frac{\bfk+\bfl+\bfm}{2}+3\bfs+A+B\right)\\
=&\frac{\Gam(\bet)\Gam\left(1-\alp-\bet+\frac{\bfk+\bfl+\bfm}{2}+3\bfs+A+B\right)}{\Gam\left(1-\alp+\frac{\bfk+\bfl+\bfm}{2}+3\bfs+A+B\right)}. 
\end{align*}

We finally get 
\begin{align*}
Z_\infty(s)=&\chi_\infty(-1)2^{\gap_1+2\gap_2-2-4s-3k}\pi^{2-s+\frac{\gap_2+\gap_3-3k}{2}}(\sqrt{-1})^{k+\gap_1}(-1)^{\gap_2+b+c}\\
&\times\frac{\Gamma\bigl(\bfs+\frac{\bfk}{2}\bigl)}{\Gamma\bigl(\bfs+\frac{\bfk}{2}+1\bigl)}\Gamma(l+b)\Gamma(m+c)\sum_{A,B}(-1)^{A+B}\binom{b}{A}\binom{c}{B}\frac{\Gamma_\infty(s;A,B)}{\Gamma(l+A)\Gamma(m+B)}, 
\end{align*}
where
\[\Gamma_\infty(\bfs;A,B)=\frac{\Gamma\bigl(\bfs+\frac{\bfl}{2}+A\bigl)\Gamma\bigl(\bfs+\frac{\bfm}{2}+B\bigl)\Gamma(\bfs+\frac{\bfk+\bfl+\bfm}{2}-1+A+B)}{\Gamma(2\bfs+\frac{\bfl+\bfm}{2}+A+B)}. \]
Lemma 3 of \cite{Orl87} with $\al=l=\bfl$, $t=\bfs+\frac{l}{2}$, $\beta=B+\frac{\bfm-l}{2}$ and $N=b$ gives
\begin{align*}
&\Gam(l+b)\sum_{A=0}^b(-1)^A\binom{b}{A}\frac{\Gamma\bigl(\bfs+\frac{l}{2}+A\bigl)\Gamma(\bfs+\frac{\bfk+l+\bfm}{2}-1+A+B)}{\Gamma(l+A)\Gamma(2\bfs+\frac{l+\bfm}{2}+A+B)}\\
=&(-1)^b\frac{\Gam\bigl(\bfs+\frac{l}{2}\bigl)\Gam\bigl(\bfs+B+\frac{\bfm}{2}+l+b-1\bigl)\Gam\bigl(\bfs+B+\frac{\bfm}{2}+b\bigl)\Gam\bigl(\bfs-\frac{l}{2}+1\bigl)}{\Gam\bigl(2\bfs+B+\frac{\bfm+l}{2}+b\bigl)\Gam\bigl(\bfs+B+\frac{\bfm}{2}\bigl)\Gam\bigl(\bfs-\frac{l}{2}-b+1\bigl)}. 
\end{align*}
It follows that 
\begin{align*}
&\Gamma(l+b)\Gamma(m+c)\sum_{A,B}(-1)^{A+B}\binom{b}{A}\binom{c}{B}\frac{\Gamma_\infty(s;A,B)}{\Gamma(l+A)\Gamma(m+B)}\\
=&(-1)^b\frac{\Gam\bigl(\bfs+\frac{l}{2}\bigl)\Gam\bigl(\bfs-\frac{l}{2}+1\bigl)}{\Gam\bigl(\bfs-\frac{l}{2}-b+1\bigl)}\Gamma(m+c)\sum_B(-1)^B\binom{c}{B}\frac{\Gam\bigl(\bfs+B+\frac{\bfm}{2}+l+b-1\bigl)\Gam\bigl(\bfs+B+\frac{\bfm}{2}+b\bigl)}{\Gam\bigl(2\bfs+B+\frac{\bfm+l}{2}+b\bigl)\Gam(m+B)}. 
\end{align*}
Again we apply Lemma 3 of \cite{Orl87} with $\al=m$, $t=\bfs+\frac{\bfm}{2}+b$, $\beta=\frac{l-\bfm}{2}-b$ and $N=c$ to obtain 
\begin{align*}
&\Gamma(m+c)\sum_B(-1)^B\binom{c}{B}\frac{\Gam\bigl(\bfs+B+\frac{\bfm}{2}+l+b-1\bigl)\Gam\bigl(\bfs+B+\frac{\bfm}{2}+b\bigl)}{\Gam\bigl(2\bfs+B+\frac{\bfm+l}{2}+b\bigl)\Gam(m+B)}\\
=&(-1)^c\frac{\Gam\bigl(\bfs+\frac{\bfm}{2}+b\bigl)\Gam\bigl(\bfs+\frac{l}{2}+m+c-1\bigl)\Gam\bigl(\bfs+\frac{l}{2}+c\bigl)\Gam\bigl(\bfs+\frac{\bfm}{2}+b-m+1\bigl)}{\Gam\bigl(2\bfs+\frac{\bfm+l}{2}+b+c\bigl)\Gam\bigl(\bfs+\frac{l}{2}\bigl)\Gam\bigl(\bfs+\frac{\bfm}{2}-m+b-c+1\bigl)}. 
\end{align*}
Then we can see that the double summation equals  
\begin{align*}
&(-1)^{b+c}\frac{\Gam\bigl(\bfs+\frac{\bfm}{2}+b\bigl)\Gam\bigl(\bfs+\frac{l}{2}+m+c-1\bigl)\Gam\bigl(\bfs+\frac{l}{2}+c\bigl)\Gam\bigl(\bfs+c-\frac{l}{2}+1\bigl)}{\Gam\bigl(\bfs-\frac{l}{2}-b+1\bigl)\Gam\bigl(2\bfs+\frac{\bfm+l}{2}+b+c\bigl)}\\
=&(-1)^{b+c}\frac{\Gamma\bigl(s+\frac{k-l+m}{2}\bigl)\Gamma\bigl(s+\frac{k+l+m}{2}-1\bigl)\Gamma\bigl(s+\frac{k-m+l}{2}\bigl)\Gamma\bigl(s+\frac{k-l-m}{2}+1\bigl)}{\Gamma(2s+k)}\cdot\frac{1}{\Gamma\bigl(\bfs-\frac{\bfk}{2}+1\bigl)}. 
\end{align*}
The last equality uses $b=\frac{\bfk-\bfl}{2}$, $m+c=\frac{\bfk+\bfm}{2}$, $\bfs+c=s+\frac{k-m}{2}$ and $2\bfs+\bfk+\bfm=2s+k+m$.
\end{proof}


\section{Classical and $p$-adic modular forms}


\subsection{Notation and conventions}\label{SS:5.1}

Besides the standard symbols $\Z$, $\Q$, $\R$, $\C$, $\Z_\ell$, $\Q_\ell$ we denote by $\R_+$ the group of strictly positive real numbers. 
Fix algebraic closures of $\Q$ and $\Q_p$, denoting them by $\overline{\Q}$ and $\overline{\Q}_p$.  
Let $\A$ be the ring of ad\`{e}les of $\Q$ and $\mu_n$ the group of $n$-th roots of unity in $\overline{\Q}$. 
Put $\wh\Z=\prod_\ell\Z_\ell$. 
For each place $v$ of $\Q$, we write $\Q_v$ for the completion of $\Q$ with respect to $v$. 
We shall regard $\Q_v$ and $\Q_v^\x$ as subgroups of $\A$ and $\A^\x$ in a natural way. For $a\in \A^\x$, let $a_v\in\Q_v^\x$ denote the $v$-component of $a$.

We denote by the formal symbol $\infty$ the real place of $\Q$. The notation $\ell$ is often referred to a rational prime. 
Let $\addchar_\Q:\A/\Q\to\C^\x$ be the additive character with the archimedean component $\addchar_\infty(x)=e^{2\pi\sqrt{-1}x}$ and $\addchar_\ell:\Q_\ell\to\C^\x$ the local component of $\addchar_\Q$ at $\ell$. Let $\Abs_\A:\Q^\times\bsl\A^\times\to\R_+$ be the ad\`{e}lic absolute value given by $\Abs_\A(a)=\abs{a}_\A=\prod_v\abs{a_v^{}}_v$.
For each rational prime $\ell$, let $\val_\ell:\Q_\ell^\x\to\Z$ denote the valuation normalized so that $\val_\ell(\ell)=1$. 
To avoid possible confusion, denote by $\uf_\ell=(\uf_{\ell,v})\in\A^\x$ the id\`{e}le defined by $\uf_{\ell,\ell}=\ell$ and $\uf_{\ell,v}=1$ if $v\not =\ell$. 

If $\om:\Q^\x\bksl \A^\x\to\C^\x$ is a Hecke character, then we denote by $\om_v:\Q_v^\x\to\C^\x$ the restriction of $\om$ to $\Q_v^\x$.  If $\chi:(\Z/N\Z)^\x\to\Qbar^\x$ is a Dirichlet character modulo $N$, let $\chi_\A:\Q^\x\R_+\bksl \A^\x/(1+N\wh\Z)^\x\to\Qbar^\x$ be the unique Hecke character of $\Q$ such that $\chi_\A(\uf_\pmq)=\chi(\pmq)^{-1}$ for any prime number $\pmq\ndivides N$. We shall call $\chi_\A$ the \emph{ad\`{e}lic lift} of $\chi$. If no confusion can arise, we write $\chi_v=\chi_{\A,v}$ for the restriction of $\chi_\A$ to $\Q_v^\times$. {Then by definition $\chi(b)=\prod_{\ell\divides N}\chi_\ell(b)$ if $b$ is an integer prime to $N$, and }\beq\label{E:convention1}\chi_\pmq(\pmq)=\chi(\pmq)^{-1} \text{ for  }\pmq\ndivides N.\eeq  Let $\zeta_\Q(s)$ be the \emph{complete} Riemann zeta function given by 
\[\zeta_\Q(s)=\pi^{-\frac{s}{2}}\Gamma\biggl(\frac{s}{2}\biggl)\prod_{\ell<\infty}\zeta_\ell(s),\quad \zeta_\ell(s):=(1-\ell^{-s})^{-1}.\]
{It is well-known that $\zeta_\Q(2)=\frac{\pi}{6}$. }
 
\begin{defn}[Teichm\"{u}ller and cyclotomic characters]
Let $p$ be a prime. We let $\bfp=4$ if $p=2$ and $\bfp=p$ otherwise.
The action of $G_\Q$ on $\mu_{p^\infty}:=\displaystyle\lim_{\longrightarrow n}\mu_{p^n}$ gives rise to a continuous homomorphism $\cyc:G_\Q\to\Z_p^\times$, called the $p$-adic cyclotomic character, defined by $\sig(\zet)=\zet^{\cyc(\sig)}$ for every $\zet\in\mu_{p^\infty}$. 
The character $\cyc$ splits into the $p$-adic Teichm\"{u}ller character $\Om:G_\Q\twoheadrightarrow\Gal(\Q(\mu_{\bfp})/\bfQ)\to\Z_p^\times$ and $\Dmd{\cdot}:G_\Q\twoheadrightarrow\Gal(\Q_\infty/\Q)\stackrel{\sim}{\to} 1+\bfp\Z_p$. 
The character $\Om$ sends $\sig$ to the unique solution in $\Z_p^\times$ of $\Om(\sig)^{l+1}=\Om(\sig)\equiv \cyc(\sig)\pmod{p}$, where $l=2\bigl\lceil\frac{p}{2}\bigl\rceil$. 
We often regard $\Om$ and $\Dmd{\cdot}^s$ with $s\in\Z_p$ as characters of $\Z_p^\times$. 
We sometimes identify $\Om$ with the Dirichlet character $\iota_p\circ\Om:(\Z/\bfp\Z)^\times\to\C^\times$.  
\end{defn}
Now we fix an arbitrary rational prime $p$ and an isomorphism $\iota_p:\overline{\Q}_p\iso\C$ for the remaining part of this paper. Let $\chi$ be a character of $\Z_p^\x$ of finite order, which can be regard as either a complex character or a $p$-adic character via composition with $\iota_p$. Let $c(\chi)$ be the exponent of the conductor of $\chi$. We view $\chi$ as a character of $G_\bfQ$ via composition with the cyclotomic character $\cyc$.  
Let $\Q^\mathrm{ab}=\bigcup_{N=1}^\infty\Q(\mu_N)$ be the maximal abelian extension of $\bfQ$ and  
\[\mathrm{rec}_\bfQ:\bfQ^\times\R_+\bsl\bfA^\times\stackrel{\sim}{\longrightarrow}\Gal(\Q^\mathrm{ab}/\Q) \]
the geometrically normalized reciprocity law map, i.e., $\mathrm{rec}_\Q(\vpi_\ell)|_{\Q^\mathrm{ab}}=\Frob_\ell$. 
Since $\chi$ factors through the quotient $\Z_p^\times\twoheadrightarrow (\Z/p^{c(\chi)}\Z)^\times$, we can identify $\chi$ with a Dirichlet character of $p$-power conductor. 
Then since $\chi_\A(\vpi_\ell)=\chi(\ell)^{-1}=\chi(\cyc(\Frob_\ell))$ for $\ell\neq p$,
\begin{align*}
\chi_\A&=\chi\circ\cyc\circ\mathrm{rec}_\Q, & \chi_p|_{\Z_p^\times}=\chi. 
\end{align*}



\subsection{Differential operators and nearly holomorphic modular forms}\label{ssec:nearlyholom}

Let $\GL_2^+(\bfR)$ be the subgroup of $\GL_2(\bfR)$ consisting of matrices with positive determinant and $\frkH_1$ the upper half plane on which $\GL_2^+(\bfR)$ acts via fractional transformation. 
Define a subgroup of $\SL_2(\bfZ)$ of finite index 
\[\Gamma_0(N)=\biggl\{\begin{pmatrix} a & b \\ c & d \end{pmatrix}\in{\SL_2(\bfZ)}\;\biggl|\;c\text{ is divisible by }N\biggl\}. \]
The Lie group $\GL_2^+(\bfR)$ acts on the complex vector space of complex valued functions $f$ on $\frkH_1$ as in (\ref{tag:modular}).  

The Maass-Shimura differential operators $\delta_k$ and $\LDiff_z$ on $C^\infty(\frakH_1)$ are given by 
\begin{align*}
\delta_k&=\frac{1}{2\pi\sqrt{-1}}\left(\frac{\partial }{\partial z}+\frac{k}{2\sqrt{-1}y}\right), & 
\LDiff_z&=-\frac{1}{2\pi\sqrt{-1}}y^2\frac{\partial}{\partial \ol{z}}
\end{align*}
with $y=\Im z\in\R_+$. 
Let $\chi:(\Z/N\Z)^\x\to\C^\x$ be a Dirichlet character, which we extend to a character $\chi^\downarrow:\Gamma_0(N)\to\C^\x$ by $\chi^\downarrow\biggl(\pMX{a}{b}{c}{d}\biggl)=\chi(d)$. 
For a non-negative integer $m$ the space $\cN^{[t]}_k(N,\chi)$ of nearly holomorphic modular forms of weight $k$, level $N$ and character $\chi$ consists of slowly increasing functions $f\in C^\infty(\frakH_1)$ such that $\LDiff_z^{t+1} f=0$ and $f|_k\gam=\chi^\downarrow(\gam)f$ for $\gam\in\Gamma_0(N)$ (\cf\cite[page 314]{Hida93Blue}). By definition $\cN^{[0]}_k(N,\chi)=\cM_k(N,\chi)$ is the space of elliptic modular forms of weight $k$, level $N$ and character $\chi$. Put $\cN_k(N,\chi)=\bigcup_{m=0}^{\lfloor\frac{k}{2}\rfloor} \cN^{[t]}_k(N,\chi)$.

Denote the space of elliptic cusp forms in $\cM_k(N,\chi)$ by $\sS_k(N,\chi)$. 
Put $\delta_k^m=\delta_{k+2m-2}\cdots\delta_{k+2}\delta_k$. 
If $f\in \cN_k(N,\chi)$, then $\delta_k^mf\in\cN_{k+2m}(N,\chi)$ (see \cite[page 312]{Hida93Blue}). 

Define an open compact subgroup of $\GL_2(\widehat\Z)$ by 
\[U_0(N)=\left\{g\in\GL_2(\widehat\Z)\;\biggl|\;g\equiv \begin{pmatrix} * & * \\ 0 & * \end{pmatrix} \pmod{N\widehat\Z}\right\}. \]
We extend $\chi_\A$ to a character $\chi_\A^\downarrow$ of $U_0(N)$ by $\chi_\A^\downarrow(g)=\prod_{\ell|N}\chi_\ell^\downarrow(g_\ell)$ (see (\ref{tag:updown}) for the definition of $\chi_\ell^\downarrow$). 
Let $\cA_k(N,\chi_\A^{-1})$ be the space of functions $\vPh:\GL_2(\A)\to\C$ such that $V_-^m\varPhi=0$ for some $m$ and such that 
\begin{align*}
\vPh(z\gam g\kap_\tht u)&=\chi_\bfA^{}(z)^{-1}\vPh(g)e^{\sqrt{-1}k\tht}\chi_\A^\downarrow(u)^{-1} & 
(z&\in\bfA^\times,\;\gam\in\GL_2(\Q),\;\tht\in\R,\;u\in U_0(N)). 
\end{align*} 

\begin{defn}[The ad\`{e}lic lift]
With each nearly holomorphic modular form $f\in\cN_k(N,\chi)$ we can associate a unique automorphic form $\itPhi(f)\in\cA_k(N,\chi_\A^{-1})$ defined by the equation
\[\itPhi(f)(\gam g_\infty u):= (f|_k g_\infty)(\sqrt{-1})\cdot \chi_\A^\downarrow(u)^{-1}\]
for $\gam \in\GL_2(\Q)$, $g_\infty\in \GL^+_2(\R)$ and $u\in \opcpt_0(N)$ (\cf \cite[\S 3]{Casselman73MA}). We call $\itPhi(f)$ the \emph{ad\`{e}lic lift} of $f$. Conversely, we can recover $f$ from $\itPhi(f)$ by 
\[f(x+\sqrt{-1} y)=y^{-k/2}\itPhi(f)\left(\pMX{y}{x}{0}{1}\right). \]
\end{defn}

Recall that $V_\pm$ are the operators as defined in \S \ref{ssec:arhizeta}. 
By definition we have 
\begin{align*}
\varPhi(\delta_k f)&=V_+\varPhi(f), & 
\varPhi(\LDiff_z f)&=V_-\varPhi(f). 
\end{align*}
We define the Whittaker coefficient and the constant term of $\vPh\in\cA_k(N,\chi_\A^{-1})$ by 
\begin{align*}
W(g,\vPh)&=\int_{\bfQ\bsl\bfA}\vPh(\bfn(x)g)\addchar_\Q(-x)\,\d x, & 
\bfa_0(g,\vPh)&=\int_{\bfQ\bsl\bfA}\vPh(\bfn(x)g)\,\d x. 
\end{align*}


\subsection{Ordinary $\bfI$-adic modular forms}\label{S:nc.unb}
 
For any subring $A\subset \C$ the space $\sS_k(N,\chi;A)$ consists of elliptic cusp forms $f=\sum_{n=1}^\infty \bfa(n,f)q^n\in\sS_k(N,\chi)$ such that $\bfa(n,f)\in A$ for all $n$. 
For every subring $A\subset\overline{\Q}_p$ containing $\Z[\chi]$ we define the space of cusp forms over $A$ by 
\[\sS_k(N,\chi;A)=\sS_k(N,\chi;\Z[\chi])\otimes_{\Z[\chi]}A. \]  
Here we have viewed $\chi$ as a $p$-adic Dirichlet character via $\iota_p^{-1}$. 

\begin{defn}[$p$-stabilized newforms]
We say that a normalized Hecke eigenform $f\in \sS_k(Np,\chi)$ is an (ordinary) $p$-stabilized newform (with respect to $\iota_p:\C\simeq\overline{\Q}_p$) if $f$ is new outside $p$ and the eigenvalue of $\U_p$, i.e. the $p$-th Fourier coefficient $\iota_p(\bfa(p, f))$, is a $p$-adic unit. 
The prime-to-$p$ part $N'$ of the conductor of $f$ is called the tame conductor of $f$. 
There is a unique decomposition $\chi=\chi'\Om^a\eps$ with $a\in\Z/l\Z$, where $l=2\bigl\lceil\frac{p}{2}\bigl\rceil$,  $\chi'$ is a Dirichlet character modulo $N'$ and $\eps$ is a character of $1+\bfp\Z_p$. 
We call $\chi'\Om^a$ the tame nebentypus of $f$. 
\end{defn}

Let $f^\circ=\sum_{n=1}^\infty \bfa(n,f^\circ)q^n\in\sS_k(Np,\chi)$ be a primitive Hecke eigenform of conductor $N_{f^\circ}$. 
We call $f^\circ$ ordinary if $\iota_p^{-1}(\bfa(p,f^\circ))$ is a $p$-adic unit.
If this is the case, then precisely one of the roots of the polynomial $X^2-\bfa(p,f^\circ)X+\chi(p)p^{k-1}$ (call it $\alp_p(f)$) satisfies $|\iota_p(\alp_p(f))|_p=1$. 
We associate to an ordinary primitive form $f^\circ$ the $p$-stabilized newform by 
\beq
f(\tau)=f^\circ(\tau)-\frac{\chi(p)p^{k-1}}{\alp_p(f)}f^\circ(p\tau)\in\sS_k(N_{f^\circ}p,\chi), \label{tag:51}
\eeq
if $N_{f^\circ}$ and $p$ are coprime, and $f=f^\circ$ if $p$ divides $N_{f^\circ}$.

Let $\cO$ be the ring of integers of a finite extension of $\Qp$ and $\bfI$ be a normal domain finite flat over $\Lam=\cO\powerseries{1+p\Zp}$. 
A point $Q\in\Spec \bfI(\Qbar_p)$, a ring homomorphism $Q:\bfI\to\Qbar_p$, is said to be locally algebraic if the restriction of $Q$ to $1+p\Zp$ is of the form $Q(z)=z^{k_Q}\ep_Q(z)$ with $k_Q$ an integer and $\ep_Q(z)\in\mu_{p^\infty}$. 
We shall call $k_Q$ the \emph{weight} of $Q$ and $\ep_Q$ the \emph{finite part} of $Q$. 
Let $\frakX_\bfI$ be the set of locally algebraic points $Q\in\Spec\bfI(\Qbarp)$ of weight $k_Q\geq 1$. 
A point $Q\in\frakX_\bfI$ is said to be \emph{arithmetic} if $k_Q\geq 2$. 
Let $\frakX_\bfI^\ari$ be the set of arithmetic points, $\wp_Q=\Ker Q$ the prime ideal of $\bfI$ corresponding to $Q$ and $\cO(Q)$ the image of $\bfI$ under $Q$. 

Let $N$ be a positive integer prime to $p$ and $\chi:(\Z/N\bfp\Z)^\x \to\cO^\x$ a Dirichlet character modulo $N\bfp$. 
An $\bfI$-adic cusp form is a formal power series $\bdsf(q)=\sum_{n=1}^\infty\bfa(n,\bdsf)q^n\in \bfI\powerseries{q}$ with the following property: there exists an integer $a_\bdsf$ such that for arithmetic points $Q\in\frakX^+_\bfI$ with $k_Q\geq a_\bdsf$, the specialization $\bdsf_Q(q)=\sum_{n=1}^\infty Q(\bfa(n,\bdsf))q^n$ is the Fourier expansion of a cusp form $\bdsf_Q\in \sS_{k_Q}(Np^e,\chi\Om^{-k_Q}\ep_Q;\calo(Q))$. 
Denote by $\bfS(N,\chi,\bfI)$ the space of $\bfI$-adic cusp forms of tame level $N$ and (even) branch character $\chi$. 
This space $\bfS(N,\chi,\bfI)$ is equipped with the action of the Hecke operators $T_\ell$ for $\ell\ndivides Np$ as in \cite[page 537]{Wiles88} and the operators $\bfU_\ell$ for $\ell\divides pN$ given by $\bfU_\ell(\sum_n \bfa(n,\bdsf)q^n)=\sum_n\bda(n\ell,\bdsf)q^n$. 

Hida's ordinary projector $\eord$ is defined by 
\[\eord:=\lim_{n\to\infty}\bfU_p^{n!}. \]
It has a well-defined action on the space of classical modular forms over a $p$-adically complete ring, which preserves the cuspidal part as well as on the space $\bfS(N,\chi,\bfI)$ (\cf\cite[page 537 and Proposition 1.2.1]{Wiles88}). 
The space $\bfS^\Ord(N,\chi,\bfI):=\eord\bfS(N,\chi,\bfI)$ is called the space of ordinary $\bfI$-adic forms with respect to $\chi$. 
For any $p$-adically complete $\Z[\chi]$-algebra $A$, we put 
\begin{align*}
\cM_k^\Ord(N,\chi;A)&=\eord\cM_k(Np^e,\chi;A); &
\sS_k^\Ord(N,\chi;A)&=\eord\sS_k(Np^e,\chi;A),
\end{align*} 
where $e$ is any integer that is greater than the exponent of the $p$-primary part of the conductor of $\chi$. 
A key result in Hida's theory for ordinary $\bfI$-adic cusp forms is that if $\bdsf\in \bfS^\Ord(N,\chi,\bfI)$, then for \emph{every} arithmetic point $Q\in\frakX_\bfI^+$, we have $\bdsf_Q\in\sS_{k_Q}^\Ord(N,\chi\Om^{-k_Q}\ep_Q;\calo(Q))$ {(\cf\cite[Theorem 3, p.215]{Hida93Blue} or \cite{Ghate13} for $p=2$)}. 
We call $\bdsf\in\bfS^\Ord(N,\chi,\bfI)$ a \emph{primitive Hida family} if $\bdsf_Q$ is a cuspidal $p$-stabilized newform of tame level $N$ for every arithmetic point $Q\in\frakX_\bfI^+$. 

\subsection{Galois representations}
If  $\bdsf\in\bfS^\Ord(N,\chi,\bfI)$ is a primitive Hida family of {cusp} forms, we denote by $V_\bdsf$ the associated $p$-adic Galois representation. {Recall that $V_\bdsf$ is a lattice in $(\Frac\bfI)^2$ with a continuous Galois action such that $V_\bdsf\ot_{\bfI,Q}{\Qbar_p}$ is the Deligne's $p$-adic Galois representation associated with $\bdsf_Q$ for every arithmetic point $Q$ of $\bfI$.} For $Q\in\frakX_{\bfI}^+$, denote by $\WD_\ell(V_{\bdsf_Q})$ the representation of the Weil-Deligne group $W_{\bfQ_\ell}$ attached to $V_{\bdsf_Q}$ for each prime $\ell$. We remark that if $\ell\parallel N$ but $\ell$ does not divide the conductor of $\chi$, then 
$\WD_\ell(V_{\bdsf_{Q}})$ is the Steinberg representation twisted by an unramified character. 
Moreover, there is an unramified finite order character $\xi_{\bdsf,\ell}:G_{\Q_\ell}\to\overline{\Q}^\times$ such that $\xi_{\bdsf,\ell}^2=\chi_{\ell}^{-1}$ and 
\beq\label{E:Galois1}V_\bdsf|_{G_{\Q_\ell}}\simeq\begin{pmatrix} \xi_{\bdsf,\ell}\cyc\Dmd{\cyc}_\bfI^{-1/2}  & * \\ 0 & \xi_{\bdsf,\ell}\Dmd{\cyc}_\bfI^{-1/2} \end{pmatrix}. \eeq

\section{A $p$-adic family of pull-backs of Siegel Eisenstein series}


\subsection{Siegel Eisenstein series}\label{ssec:61}

We work in ad\`{e}lic form, which allows us to assemble Eisenstein series out of local data.
Put $K_n=\U(n)\GSp_{2n}(\widehat\Z)$. 
{Let $p$ be a fixed rational prime as in the previous section. }
Fix characters $\chi,\hat\ome$ of $\bfZ_p^\times$ of finite order and extend them to Hecke characters $\chi_\bfA,\hat\ome_\bfA:\bfQ^\times\bsl\bfA^\times\to\bfC^\times$ by composition with the quotient map $\bfQ^\times\R_+\bsl\bfA^\times\simeq\widehat{\bfZ}^\times\twoheadrightarrow\bfZ_p^\times$. 
We regard $\chi$ as either a $p$-adic character or a complex character via composition with $\iota_p$. 
Let 
\[I_3(\hat\ome_\bfA^{-1},\chi_\bfA^{}\hat\ome_\bfA^{}\Abs_\A^s)=\Ind_{\cP_3(\A)}^{\GSp_6(\A)}(\chi_\A^2\hat\ome_\bfA^{}\boxtimes\chi_\A^{-3}\hat\ome_\bfA^{-1}\Abs_\A^s)\simeq\otimes'_vI_3(\hat\om_v^{-1},\chi_v^{}\hat\om_v^{}\Abs_{\Q_v}^s)\]
be the global degenerate principal series representation of $\GSp_6(\A)$ on the space of right $K_3$-finite functions $f:\GSp_6(\A)\to\C$ satisfying the transformation laws
\[f(\bfn(z)\bfm(A,\nu)g)=\hat\ome_\bfA(\nu^{-1}\det A)\chi_\A(\nu^{-3}(\det A)^2)\abs{\nu^{-3}(\det A)^2}^{1+s}_\bfA f(g) \]
for $A\in\GL_3(\A)$, $\nu\in\A^\x$, $z\in\Sym_3(\A)$ and $g\in \GSp_6(\A)$.  
We define global holomorphic sections of $I_3(\hat\ome_\bfA^{-1},\chi_\bfA^{}\hat\ome_\bfA^{}\Abs_\A^s)$ similarly. 
The Eisenstein series associated to a holomorphic section $f_s$ of $I_3(\hat\ome_\bfA^{-1},\chi_\bfA^{}\hat\ome_\bfA^{}\Abs_\A^s)$ is defined by 
\[E_\A(g,f_s)=\sum_{\gamma\in \cP_3(\Q)\bksl \GSp_6(\Q)}f_s(\gamma g). \] 
Such series is absolutely convergent for $\Re s>1$ and can be continued to a meromorphic function in $s$ on the whole plane.  

Fix a triplet $(k,l,m)$ of positive integers with $k\geq l\geq m\geq 2$ and let $\lam=(\lam_1,\lam_2,\lam_3)$ be the parity type of $(k,l,m)$ introduced in \eqref{E:parity}. Fix a square-free integer $N$ which is not divisible by $p$.
We write $\hat\ome=\ome_1\ome_2\ome_3$ as a product of three characters $\om_1,\om_2,\om_3$ of $\bfZ_p^\times$. 
Set 
\[\cD=(\chi,\om_1,\om_2,\om_3). \]
Assume that $\hat\om_\infty=\sgn^{k-\lam_1}$. 
Now we define a distinguished section of $I_3(\hat\om^{-1}_v,\chi_v\hat\om_v\Abs_{\Q_v}^{s})$ for each $v\nmid N$:  
\begin{itemize}
\item In the archimedean case we consider the section $f^{[k,\gap]}_{s,\infty}$ defined in \S \ref{ssec:realsection}; 
\item In the $p$-adic case we consider $f_{\cD,s,p}$, where the section $f_{\cD,s,p}$ of $I_3(\hat\om_p^{-1},\chi^{}_p\hat\om^{}_p\Abs_{\Q_p}^{s})$ is attached to the quadruplet $\cD$ in Definition \ref{def:psection}; 
\item If $\ell$ and $Np$ are coprime, then $f_{s,\ell}^0$ is the section with $f_{s,\ell}^0(\GSp_6(\Z_\ell))=1$. 
\end{itemize}
{Let $f_{s,N}=\ot_{\ell|N}f_{s,\ell}$ be an arbitrary holomorphic section of $\bigotimes_{\ell|N}I_3(\hat\om^{-1}_\ell,\chi_\ell\hat\om_\ell\Abs_{\Q_\ell}^{s})$ such that $f_{s,\ell}$ is invariant by $K_0^{(3)}(N\Z_\ell)$ for all $\ell|N$}. Define the normalized Siegel Eisenstein series  
\beq\label{E:ES1}E^\star_\bfA(g,f_{\cD,s,N}^{[k,\gap]})=L^{(\infty pN)}(2s+2,\chi^2_\A\hat\om_\A^{})L^{(\infty pN)}(4s+2,\chi^4_\A\hat\om^2_\A)\gamma_{(k,l,m)}^\star(s)^{-1}\cdot E_\A(g,f_{\cD,s,N}^{[k,\gap]}), \eeq
where $\gamma_{(k,l,m)}^\star(s)$ is defined in \S \ref{ssec:arhizeta} and $f_{\cD,s,N}^{[k,\gap]}$ is a global holomorphic section of $I_3(\hat\ome_\bfA^{-1},\chi_\bfA^{}\hat\ome_\bfA^{}\Abs_\A^{s})$ defined by 
\[f_{\cD,s,N}^{[k,\gap]}(g)=f^{[k,\gap]}_{s,\infty}(g_\infty)f_{\cD,s,p}(g_p)f_{s,N}((g_\ell)_{\ell|N})\prod_{\ell\nmid Np}f^0_{s,\ell}(g_\ell). \]
 
Since $f_{\cD,s,p}$ is supported in the big cell $\cP_3(\Q_p) J_3\cP_3(\Q_p)$, {according to \cite[(3.2.2.2)]{HLS06Doc}} for $g\in \GSp_6(\A)$ with $g_p\in\cP_3(\Q_p)$, we have the Fourier expansion
\begin{align}\label{E:EFC.E}
E_\A(g,f_{\cD,s,N}^{[k,\gap]})&=\sum_{B\in\Sym_3(\Q)} \cW_B(g,f_{\cD,s,N}^{[k,\gap]}),
\end{align}
 where
\begin{align*}
\cW_B(g,f_{\cD,s,N}^{[k,\gap]})&=\int_{\Sym_3(\A)}f_{\cD,s,N}^{[k,\gap]}(J_3\bfn(z)g)\addchar_\Q(-\tr(Bz))\,\rmd z.
\end{align*}


\subsection{The Fourier expansion of the pull-back of Eisenstein series}
Let  $r$ be an integer with $\lam_2\leq r\leq k-2$. Put \[s_0=\frac{k-\gap_1}{2}-r-1. \]
Let $\bfE^{[k,r,\gap]}_{\cald,N}(f_{s_0,N})\colon\frakH_1^3\to\C$ be the modular form of weight $(k,k-\gap_2,k-\gap_3)$ 
defined by 
\[\bfE^{[k,r,\gap]}_{\cD,N}(f_{s_0,N})(x+y\sqrt{-1})
:=\lim_{s\to s_0}\frac{E^\star_\A\bigl(\iota(\bfn(x_1)\bfm(\sqrt{y_1}),\bfn(x_2)\bfm(\sqrt{y_2}),\bfn(x_3)\bfm(\sqrt{y_3})),f_{\cD,s,N}^{[k,\gap]}\bigl)}{\sqrt{y_1}^k\sqrt{y_2}^{k-\gap_2}\sqrt{y_3}^{k-\gap_3}} \]
for $y=(y_1,y_2,y_3)\in \R_+^3$ and $x=(x_1,x_2,x_3)\in \R^3$. 
 {We give the Fourier expansion of $\bfE^{[k,r,\gap]}_{\cald,N}(f_{s_0,N})$ after preparing some notation. 
For any reduced ring $R$, we put \[\Sym_3^*(R)=\stt{A\in \Sym_3(\Frac(R))\mid \Tr(AB)\in R\text{ for all } B\in\Sym_3(R)}.\] }
Let $\Sym_3^+$ denote the set of positive definite rational symmetric matrices of rank $3$. Let $T_3^+=\Sym^+\cap \Sym_3^*(\Z)$ denote the set of positive definite symmetric half-integral matrices of rank $3$. Recall that 
\[\Xi_p=\{(b_{ij})\in{\Sym_3^*(\Zp)}\;|\;b_{11},b_{22},b_{33}\in p\Zp \text{ and } b_{12},b_{23},b_{31}\in{2^{-1}}\Zp^\times\}. \]
For $\ell\nmid N$, $F_{B,\ell}$ denotes the polynomial $F_B$ with the base field $F=\Q_\ell$ in \eqref{E:Siegel}. Set $\Q_N=\prod_{\ell|N}\Q_\ell$. 

\begin{prop}\label{P:pbkEC.E}
Put $n=\max\{1,c(\chi),c(\ome_i)\}$ and $t=\min\stt{k-r-2+\lam_2,r}<\frac{k}{2}$. 
{The pull-back \[\bfE^{[k,r,\gap]}_{\cD,N}(f_{s_0,N})\in \cN_k^{[t]}(N,\om_1^{-1})\ot_\C\cN_k^{[t]}(N,\om_2^{-1})\ot_\C \cN_k^{[t]}(N,\om_3^{-1})\] }is a nearly holomorphic modular form on $\frakH_1^3$ of level $\Gamma_0(Np^{2n})^3$ and nebentypus $(\ome_1^{-1},\ome_2^{-1},\ome_3^{-1})$ with the Fourier expansion given by 
\begin{align*}
\bfE^{[k,r,\gap]}_{\cD,N}(f_{s_0,N})
&=\frac{C^{[k,r,\gap]}_1}{\gamma_{(k,l,m)}^\star\bigl(s_0)}\sum_{B\in T^+_3\cap\Xi_p} \bfW_B^{[k,r,\gap]}(y)\cdot \cQ_B(\cD)a_B(\chi^2\hat\om,k-2r-\gap_1)b_{B,N}^{[k,r,\gap]}q_1^{b_{11}}q_2^{b_{22}}q_3^{b_{33}}, 
\end{align*}
{where $\cQ_B(\cD)=\cQ_B(\chi,\om_1,\om_2,\om_3)$ is defined as in \eqref{E:QB},} $q_i=e^{2\pi\sqrt{-1}(x_i+\sqrt{-1}y_i)}$, 
\begin{align*}
&a_B(\chi^2\hat\om,k-2r-\gap_1):=\prod_{\ell\ndivides Np}F_{B,\ell}(\chi_\ell(\ell)^2\hat\om_\ell(\ell)\ell^{2r+\gap_1-k});\\
&b_{B,N}^{[k,r,\gap]}:=\lim_{s\to s_0}\int_{\Sym_3(\Q_N)}f_{s,N}(J_3\bfn(z))\addchar_\Q(-\tr(Bz))\,\rmd z.   
\end{align*}
\end{prop}

\begin{proof}Note that $\det B\in\Zp^\x$ for $B\in\Xi_p$, so $\cW_B(g,f_{\cald,s,N}^{[k,\gap]})=0$ unless $\det B\neq 0$. 
The level and nebentypus are determined by Proposition \ref{prop:13}.  {In particular, our section $f_{\cD,s,N}^{[k,\gap]}$ is right invariant by $\bfn(\Sym_3(\wh\Z))$. Combined with \propref{P:FC.E}, this shows that $\cW_B(g,f_{\cald,s,N}^{[k,\gap]})=0$ unless $B\in T_3^+$.} 
We can derive the Fourier expansion formula from \eqref{E:EFC.E} combined with the factorization
\begin{align*}\cW_B(\bfn(z)\bfm(A),f_{\cald,s,N}^{[k,\gap]})&=e^{2\pi\sqrt{-1}(b_{11}x_1+b_{22}x_2+b_{33}x_3)}\cW_B(\bfm(A),f^{[k,\gap]}_{s,\infty}) \cW_B(f_{\cD,s,p})\prod_{\ell|N}\cW_B(f_{s,\ell})\prod_{\ell\nmid Np}\cW_B(f_{s,\ell}^0)\\
&(A=\diag{\sqrt{y_1},\sqrt{y_2},\sqrt{y_3}}\in\GL_3(\R),\,z=\diag{x_1,x_2,x_3}\in\Sym_3(\R)).\end{align*}
and the computations of local Whittaker functions 
\[
\lim_{s\to s_0}\cW_B(\bfm(A),f^{[k,\gap]}_{s,\infty}), \quad 
\cW_B(f^0_{s,\ell}); \quad
\cW_B(f_{\cD,s,p})\]
in \eqref{E:Winfty},  \eqref{tag:unramWhittaker}, and  \propref{prop:13} respectively. With this Fourier expansion, we can deduce that $\bfE^{[k,r,\gap]}_{\cD,N}(f_{s_0,N})$ is nearly holomorphic with
\[\LDiff_{z_1}^{t+1}\LDiff_{z_2}^{t+1}\LDiff_{z_3}^{t+1}\bfE^{[k,r,\gap]}_{\cD,N}(f_{s_0,N})=0\] from the fact that $\bfW_B^{[k,r,\gap]}(y)$ is a polynomial in $\C[y_1^{-1},y_2^{-1},y_3^{-1}]$ of degree less than or equal to \[t=\min\stt{k-r-2,r-\lam_2}+\lam_2\leq \frac{k-2+\lam_2}{2}<\frac{k}{2}\]
in view of \eqref{E:FC.E}.  \end{proof}


\begin{defn}\label{D:nEis.E}
For $\ell|N$, we define $f_{s,\ell}^*:=f_{\Phi^0_\ell}(\chi_\ell\hat\om_\ell\Abs_{\Q_\ell}^{s})\in I_3(\hat\om^{-1}_\ell,\chi_\ell\hat\om_\ell\Abs_{\Q_\ell}^{s})$ associated with $\Phi^0_\ell=\bbI_{\Sym_3(\Z_\ell)}$ the holomorphic section supported in the open cell introduced in \defref{def:cell}. Put 
\begin{align*}
\bfE^{[k,r,\gap]}_{\cald,N}&=\bfE^{[k,r,\gap]}_{\cald,N}(f_{s_0,N}^*); & 
f_{s,N}^*&:=\bigotimes_{\ell|N}f_{s,\ell}^*. 
\end{align*}
\end{defn}
If $B\in T_3^+$, by \eqref{tag:Fourier} we have
\[b_{B,N}^{[k,r,\gap]}=\wh\bbI_{\Sym_3(\Z_N)}(-B)=\bbI_{\Sym_3^*(\Z_N)}(B)=1.\] It follows immediately from \propref{P:pbkEC.E} that
\beq\label{E:FE1}
\bfE^{[k,r,\gap]}_{\cD,N}
=\frac{C^{[k,r,\gap]}_1}{\gamma_{(k,l,m)}^\star(s_0)}\sum_{B\in T^+_3\cap\Xi_p} \bfW_B^{[k,r,\gap]}(y)\cdot \cQ_B(\cD)a_B(\chi^2\hat\om,k-2r-\gap_1)q_1^{b_{11}}q_2^{b_{22}}q_3^{b_{33}}. 
\eeq

\subsection{Holomorphic and ordinary projections of $\bfE^{[k,r,\gap]}_{\cD,N}$}
Recall that $\LDiff_{z}$ is the weight-lowering operator defined in \S \ref{ssec:nearlyholom}. 
For $k\geq 2$ and $t<k/2$, we write $\Hol:\cN_k^{[t]}(N,\chi)\to \cM_k(N,\chi)$ for the holomorphic projection on the space of nearly holomorphic modular forms {(\cf\cite[(8a), page 314]{Hida93Blue}). Recall that if $f\in \cN_{k}^{[t]}(N,\chi)$, then $\Hol(f)\in \cM_{k}(N,\chi)$ is the unique holomorphic form such that $f=\Hol(f)+\sum_{j=1}^t\delta_{k-2j}^j h_j$ with $h_j\in \cM_{k-2j}(N,\chi)$.} Motivated by \cite[(3.6)]{Mizumoto90}, we consider the modular forms obtained by applying the weight-lowering operators to the pull-back of Siegel Eisenstein series.  
This is different from the classical Rankin-Selberg setting (\cf \cite{CM20}),  where the weight-raising operators are used instead. 
\begin{prop}\label{P:1.E} Suppose that $k< l+m-1$. Let $r$ be an integer which satisfies 
\[k-\frac{l+m+\lam_1}{2}\leq r\leq \frac{l+m}{2}-2. \]
Put $n=k-r-2+\frac{l+m-\gap_1}{2}$. 
Then $e_\Ord\Hol\Big(\LDiff_{z_2}^\frac{k-l-\gap_2}{2}\LDiff_{z_3}^\frac{k-m-\gap_3}{2}\bfE^{[k,r,\gap]}_{\cD,N}\Big)$ has the $q$-expansion 
\[(-1)^{k+\frac{m+l+\gap_1}{2}+\gap_2}\sum_{B=(b_{ij})\in T_3^+\cap\Xi_p} \cQ_B(\chi\cyc^n,\om_1\cyc^{-k},\om_2\cyc^{-l},\om_3\cyc^{-m})a_B(\chi^2\hat\om,k-2r-\gap_1)q_1^{b_{11}}q_2^{b_{22}}q_3^{b_{33}}. \]
\end{prop}

\begin{proof}
Put $b=\frac{k-l-\gap_2}{2}$ and $c=\frac{k-m-\gap_3}{2}$. 
If $f$ is a holomorphic function on $\frakH_1$, then 
\[\LDiff_z^n(y^{-a} f)=
\begin{cases}
(4\pi)^{-n}n!{a\choose n}\cdot y^{n-a} f &\text{if $n\leq a$, }\\
0 &\text{if $n>a$. }
\end{cases}\] 
By the Fourier expansion \eqref{E:FE1} and the polynomial expansion \eqref{E:FC.E} of Whittaker functions, the difference 
\[\LDiff_{z_2}^b\LDiff_{z_3}^c\bfE^{[k,r,\gap]}_{\cD,N}-\frac{C^{[k,r,\gap]}_1}{{\gamma_{(k,l,m)}^\star(s_0)}}\sum_B \frac{b!c!}{(4\pi)^{b+c}}Q_{0,b,c}^{[k,\gap]}(B,r)\cQ_B(\cD)a_B(\chi^2\hat\om,k-2r-\gap_1) q_1^{b_{11}}q_2^{b_{22}}q_3^{b_{33}}\]
belongs to $(y_1^{-1},y_2^{-1},y_3^{-1})\C[y_1^{-1},y_2^{-1},y_3^{-1}]\powerseries{q_1,q_2,q_3}$. On the other hand, we can write 
\[\LDiff_{z_2}^b\LDiff_{z_3}^c\bfE^{[k,r,\gap]}_{\cD,N}=\Hol(\LDiff_{z_2}^b\LDiff_{z_3}^c\bfE^{[k,r,\gap]}_{\cD,N})+\sum_{i+j+t\geq 1}\del^i_{k-i}f_i(q_1)\del^j_{l-j}g_j(q_2)\del^t_{m-t}h_t(q_3), 
\]
where $f_i$, $g_j$ and $h_t$ are holomorphic modular forms. 
Equating the constant terms of this identity as a polynomial in $y_1^{-1},y_2^{-1},y_3^{-1}$ and employing the relation 
\beq
\del^t_k=\sum_{a=0}^t{t\choose a}\frac{\Gam(t+k)}{\Gam(a+k)}(-4\pi y)^{a-t}\left(\frac{1}{2\pi\sqrt{-1}}\frac{\partial}{\partial z}\right)^a \label{tag:MaassShimura}
\eeq
(see \cite[(3), page 311]{Hida93Blue}), we see that the holomorphic projection $\Hol(\LDiff_{z_2}^b\LDiff_{z_3}^c\bfE^{[k,r,\gap]}_{\cD,N})(q)$ equals
\[\frac{C^{[k,r,\gap]}_1}{{\gamma_{(k,l,m)}^\star(s_0)}}\sum_B\frac{b!\,c!}{(4\pi)^{b+c}} Q_{0,b,c}^{[k,\gap]}(B,r)\cQ_B(\cD)a_B(\chi^2\hat\om,k-2r-\gap_1)  q_1^{b_{11}}q_2^{b_{22}}q_3^{b_{33}}-\sum_{i+j+t\geq 1}\theta^if_i(q_1)\theta^j g_j(q_2)\theta^th_t(q_3). \]
Here $\theta$ stands for the Serre's operator $\theta(\sum_i a_iq^i)=\sum_i ia_iq^i$. 
Since $e_\Ord \theta=0$, the $q$-expansion of the ordinary projection $\eord\Hol(\LDiff_{z_2}^b\LDiff_{z_3}^c\bfE^{[k,r,\gap]}_{\cD,N})(q)$ equals
\[\frac{C^{[k,r,\gap]}_1}{{\gamma_{(k,l,m)}^\star(s_0)}}\sum_B \frac{b!\,c!}{(4\pi)^{b+c}}\bfc_B\cdot q_1^{b_{11}}q_2^{b_{22}}q_3^{b_{33}},\]
where
\begin{align*}
\bfc_B&=\lim_{j\to\infty}Q_{0,b,c}^{[k,\gap]}(B_j,r)\cQ_{B_j}(\cD)a_{B_j}(\chi^2\hat\om,k-2r-\gap_1), & 
B_j&:=\begin{pmatrix}
p^{j!}b_{11}&b_{12}&b_{13}\\
b_{12}&p^{j!}b_{22}&b_{23}\\
b_{13}&b_{23}&p^{j!}b_{33}
\end{pmatrix}. 
\end{align*}
{We recall that in the above expression of $\bfc_B$,\begin{itemize}\itemsep 2mm\item $Q_{0,b,c}^{[k,\gap]}(B_j,r)$ is the archimedean contribution from the coefficients in the polynomial expansion of $\bfW_B^{[k,r,\lam]}(y)$ in \eqref{E:FC.E}, \item $\cQ_B(\cD)$ given in {\eqref{E:QB} of} \defref{def:psection} is the contribution from the $p$-adic Whittaker functions; \item $a_{B_j}(\chi^2\hat\om,k-2r-\gap_1)$ defined in \propref{P:pbkEC.E} is obtained from the spherical Whittaker functions.\end{itemize}
Since $p^{j!}\to 1$ in $\Z_\ell$ as $j\to\infty$ for any rational prime $\ell\neq p$, we find by definition that  }
\begin{align*} 
\cQ_{B_j}(\cD)&=\cQ_B(\cD), & 
\lim_{j\to\infty}a_{B_j}(\chi^2\hat\om,k-2r-\gap_1)&=a_B(\chi^2\hat\om,k-2r-\gap_1). 
\end{align*}
Note that $Q_{0,b,c}^{[k,\gap]}(B,r)$ is a polynomial in $B$, so we have
\begin{align*}
\bfc_B&=Q_{0,b,c}^{[k,\gap]}(B_\infty,r)\cQ_{B}(\cD)a_B(\chi^2\hat\om,k-2r-\gap_1), & 
B_\infty&=\begin{pmatrix}
0&b_{12}&b_{13}\\
b_{12}&0&b_{23}\\
b_{13}&b_{23}&0
\end{pmatrix}. 
\end{align*}
By the formula of $Q_{0,b,c}^{[k,\gap]}(B_\infty,r)$ in \lmref{L:coeff.E} and \defref{def:psection} of $\cQ_B$, { we obtain
\[Q_{0,b,c}^{[k,\gap]}(B_\infty,r)=w_{0,b,c}2^{-3n+k+l+m}\cQ_B(\cyc^n,\cyc^{-k},\cyc^{-l},\cyc^{-m}),\]}
and it follows that
\begin{align*}
\bfc_B&=w_{0,b,c}\cdot 2^{-3n+k+l+m}\cQ_B(\chi\cyc^n,\om_1\cyc^{-k},\om_2\cyc^{-l},\om_3\cyc^{-m})a_B(\chi^2\hat\om,k-2r-\gap_1).
\end{align*}
We thus obtain the lemma by noting the equality
\beq
(-1)^{k+\frac{m+l+\gap_1}{2}+\gap_2}\gamma_{(k,l,m)}^\star\left(\frac{k-\gap_1}{2}-r-1\right)=\frac{C^{[k,r,\gap]}_1b!\,c!}{(4\pi)^{b+c}}2^{-3n+k+l+m}\cdot w_{0,b,c}. \label{tag:equality}
\eeq
The constant $C^{[k,r,\gap]}_1$ is defined in \propref{P:FC.E}. 
The equality can be checked by the following items:
\begin{itemize}
\item The power of $2$:
\begin{align*}
&3(3+2r-k-\gap_2)+\{2(k-r)-3\}-2b-2c+(k+l+m-3n)\\
&+(7M-3b-3c-3\gap_1)=-2-k+2(l+m)+\gap_1+2\gap_2.
\end{align*}
\item The power of $\pi$: $(6-2)-b-c+(3M-b-c-2\gap_1-\gap_2)=-3r+k+l+m+\gap_2-\gap_1-2$.
\end{itemize}
\end{proof}


\subsection{The modular forms $G_{k_1,k_2,k_3}^{[n]}(\scrd)$}

\begin{defn}\label{D:balanced.E} 
Let $(k_1,k_2,k_3)$ be a triplet of positive integers. 
Put $k^*=\max\stt{k_1,k_2,k_3}$. 
We say that $(k_1,k_2,k_3)$ is \emph{balanced} if $2k^*<k_1+k_2+k_3$. 
An integer $n$ is said to be \emph{critical} for $(k_1,k_2,k_3)$ if 
\[k^*\leq n\leq k_1+k_2+k_3-k^*-2.\]
Note that critical integers exist {if and only if} $2k^*< k_1+k_2+k_3-1$.
\end{defn}

\begin{defn}\label{D:Gmod.E}
Fix a balanced triplet $(k_1,k_2,k_3)$ of positive integers. 
Take a permutation $\sig$ of $\{1,2,3\}$ so that $k^*=k_{\sig(1)}\geq k_{\sig(2)}\geq k_{\sig(3)}$. 
Denote the parity type of $(k_{\sig(1)}, k_{\sig(2)}, k_{\sig(3)})$ by $\del=(\del_1,\del_2,\del_3)$. 
For each critical integer $n$ for $(k_1,k_2,k_3)$ and quadruplet $\scrd=(\ep_0,\ep_1,\ep_2,\ep_3)$ of finite-order $p$-adic characters of $\bfZ_p^\times$ we define the modular form $G_{k_1,k_2,k_3}^{[n]}(\scrd)$ by
\[G_{k_1,k_2,k_3}^{[n]}(\scrd):=(-1)^{k+\frac{m+l+\gap_1}{2}+\gap_2}\eord\Hol\biggl(\LDiff_{z_\sig(2)}^{\frac{k^*-k_{\sig(2)}-\delta_2}{2}}\LDiff_{z_{\sig(3)}}^{\frac{k^*-k_{\sig(3)}-\delta_3}{2}}\bfE^{[k^*,r,\del]}_\cD\biggl), \]
where $r=\Big\lceil\frac{k^*+k_1+k_2+k_3}{2}\Big\rceil-n-2$ and $\cald=(\iota_p\circ\ep_0,\iota_p\circ\ep_1,\iota_p\circ\ep_2,\iota_p\circ\ep_3)$.
\end{defn}

\begin{cor}\label{C:Gmod.E}  
With notation in \defref{D:Gmod.E}, $G_{k_1,k_2,k_3}^{[n]}(\scrd)$ is an ordinary cusp form of weight $(k_1,k_2,k_3)$, level $\Gam_0(Np^\infty)^3$ and nebentypus $(\eps_1^{-1},\eps_2^{-1},\eps_3^{-1})$ whose $q$-expansion at the infinity cusp is given by 
\[\sum_{B=(b_{ij})\in T_3^+\cap\Xi_p}\cQ_B(\ep_0\cyc^n,\ep_1\cyc^{-k_1},\ep_2\cyc^{-k_2},\ep_3\cyc^{-k_3})
a_B(\ep_0^2\ep_1\ep_2\ep_3,2n-(k_1+k_2+k_3)+4)\cdot q_1^{b_{11}}q_2^{b_{22}}q_3^{b_{33}}. \]
\end{cor}

\begin{proof}
The assertion for the Fourier expansion is a direct consequence of \propref{P:1.E} by symmetry.  \lmref{L:constant} below implies the cuspidality of $G_{k_1,k_2,k_3}^{[n]}(\scrd)$.
\end{proof}
\begin{lm}\label{L:constant}
Let $f\in\cM_k(N,\chi;A)$ and $\vPh(f)$ be the ad\`{e}lic lift of $f$. Assume that $\bfa_0(g,\vPh(f))=0$ whenever $g_p\in B_2(\bfQ_p)$. 
Then $e_\Ord f\in\sS_k^\Ord(N,\chi;A)$.  
\end{lm}
\begin{proof} We write $\vPh=\vPh(f)$ for brevity. Out task is to prove that $\bfa_0(g,\vPh)=0$ for all $g\in \GL_2(\bfA)$. 
Since 
\[\bfa_0(\gam\bfn(x)\diag{a,d} g\kap_\tht,\vPh)=(ad^{-1})^{k/2}e^{\sqrt{-1}k\tht}\bfa_0(g_\bff,\vPh)\]
for $\gam\in B_2(\bfQ)$, $x\in\bfA$, $a,d\in\R_+$ and $\tht\in\bfR$, it suffices to show that $\bfa_0(g,\vPh)=0$ for all $g\in \GL_2(\widehat{\bfZ})$. 
Since 
\[\GL_2(\bfZ_p)=\bfn^-(p\bfZ_p)B_2(\bfZ_p)\sqcup \bfn(\bfZ_p)J_1B_2(\bfZ_p), \]
where $J_1=\begin{pmatrix} 0 & -1 \\ 1 & 0 \end{pmatrix}$, it suffices to show that $\bfa_0(h\bfn^-(y),\vPh)=\bfa_0(hJ_1,\vPh)=0$ for any $h\in \GL_2(\widehat{\bfZ}^{(p)})$ and $y\in p\bfZ_p$. 
Recall that the operator $\bfU_p$ is defined by 
\[\bfU_p\vPh(g)=p^{(k-2)/2}\sum_{x\in\bfZ_p/p\bfZ_p}\vPh\biggl(g\begin{pmatrix} \vpi_p & x \\ 0 & 1 \end{pmatrix}\biggl). \]
Recall that $\vpi_p\in\widehat{\bfQ}^\times$ is defined by $\vpi_{p,p}=p$ and $\vpi_{p,\ell}=1$ for $\ell\neq p$.
Since 
\[\begin{pmatrix} 1 & 0 \\ y & 1 \end{pmatrix}\begin{pmatrix} \vpi_p^m & x \\ 0 & 1 \end{pmatrix}=\begin{pmatrix} \frac{\vpi_p^m}{1+xy} & \frac{x}{1+xy} \\ 0 & 1 \end{pmatrix}\begin{pmatrix} 1 & 0 \\ \vpi_p^my & 1+xy \end{pmatrix}\in B_2(\bfQ_p)U_0(N)\]
for $y\in p\bfZ_p$, $x\in\bfZ_p$ and sufficiently large $m$, we get 
\[\bfa_0(h\bfn^-(y),\bfU_p^m\vPh)=p^{(k-2)m/2}\sum_{x\in\bfZ_p/p\bfZ_p}\vPh\left(h\begin{pmatrix} \frac{\vpi_p^m}{1+xy} & \frac{x}{1+xy} \\ 0 & 1 \end{pmatrix}\right)=0\]
by assumption. 
It follows that $\bfa_0(h\bfn^-(y),\vPh)=\displaystyle\lim_{n\to\infty}\bfa_0(h\bfn^-(y),\bfU_p^{n!}f)=0$. 
If $x\in p^n\bfZ_p^\times$ with $n<m$, then 
\[\begin{pmatrix} 0 & -1 \\ 1 & 0 \end{pmatrix}\begin{pmatrix} \vpi_p^m & x \\ 0 & 1 \end{pmatrix}=\begin{pmatrix} \vpi^{m-n}_p & -\vpi_p^nx^{-1} \\ 0 & \vpi_p^n\end{pmatrix}\begin{pmatrix} \vpi_p^nx^{-1} & 0 \\ \vpi^{m-n}_p & \vpi_p^{-n}x \end{pmatrix}\in B_2(\bfQ_p)\bfn^-(p\bfZ_p). \]
One can therefore see that 
\[\bfa_0(hJ_1,\bfU_p^m\vPh)=p^{(k-2)m/2}\bfa_0(hJ_1\diag{\vpi_p^m,1},f)=p^{(k-1)m/2}\bfa_0(\diag{1,\vpi_{(p)}^{-m}}hJ_1\vPh), \]
from which we conclude that 
\[\bfa_0(hJ_1,\vPh)=\displaystyle\lim_{n\to\infty}p^{(k-1)n!/2}\bfa_0(\diag{1,\vpi_{(p)}^{-n!}}hJ_1,\vPh)=0. \]
Here $\vpi_{(p)}\in\widehat{\bfZ}^\times$ is defined by $\vpi_{(p),p}=1$ and $\vpi_{(p),\ell}=p$ for $\ell\neq p$. 
\end{proof}


\subsection{The $p$-adic interpolation of $G_{k_1,k_2,k_3}^{[n]}(\scrd)$ }\label{SS:6.5}

Let $\bfu=1+\bfp\in 1+\bfp\Zp$ be a topological generator.  
We identify $\cO\powerseries{\Gal(\Q_\infty/\Q)}$ with $\cO\powerseries{X}$ where $X=[\bfu]-1$ with the group-like element $[\bfu]$ in $\Lam$. 
Put 
\begin{align*}
\Lam&=\cO\powerseries{\Gal(\Q_\infty/\Q)}, & 
\Lambda_3&=\cO\powerseries{X_1,X_2,X_3}, & 
\Lambda_4&=\Lam_3\powerseries{T}. 
\end{align*} 

For each $\ell$ and $B\in T_3^+$, let $F_{B,\ell}(X)\in\Z[X]$ be as defined in (\ref{tag:unramWhittaker}). 
Let $\al_X:\Zp^\x\to\Zp\powerseries{X}^\x$ be the character $\al_X(z)=\Dmd{z}_X=(1+X)^{\log_p z/\log_p\bfu}$. 
Let $\ul{\chi}=(\chi_1,\chi_2,\chi_3)$ be a triplet of $\calo$-valued finite-order characters of $\bfZ_p^\times$. 
For each $a\in\Z/(p-1)\Z$ we define the formal power series $\cG_{\ul{\chi}}^{(a)}\in \Lambda_4\powerseries{q_1,q_2,q_3}$ by 
\begin{align*}
\cG_{\ul{\chi}}^{(a)}(X_1,X_2,X_3,T)&=\sum_{B=(b_{ij})\in T_3^+\cap\Xi_p}\cQ_B^{(a)}(X_1,X_2,X_3,T)\cdot \cF_B^{(a)}(X_1,X_2,X_3,T)\cdot q_1^{b_{11}}q_2^{b_{22}}q^{b_{33}}_3,
\end{align*}
where $\cQ_B^{(a)}$ and $\cF_B^{(a)}\in\Lam_3\powerseries{T}$ are power series given by 
\begin{align*}
\cQ_B^{(a)}(X_1,X_2,X_3,T)&=\Om(8b_{23}b_{31}b_{12})^a\Dmd{8b_{23}b_{31}b_{12}}_T\chi_1(2b_{23})^{-1}\Dmd{2b_{23}}_{X_1}^{-1}\chi_2(2b_{31})^{-1}\Dmd{2b_{31}}_{X_2}^{-1}\chi_3(2b_{12})^{-1}\Dmd{2b_{12}}_{X_3}^{-1}, \\
\cF_B^{(a)}(X_1,X_2,X_3,T)&=\prod_{\ell\ndivides pN}F_{B,\ell}(\Dmd{\ell}_{X_1,X_2,X_3,T}^{(a)}\ell^{-2}), 
\end{align*}
where  
\beq\label{E:LX.7}\Dmd{\ell}^{(a)}_{X_1,X_2,X_3,T}:=(\Om^{-2a}\chi_1\chi_2\chi_3)(\ell)\ell^{-2}\cdot\Dmd{\ell}_{X_1}\Dmd{\ell}_{X_2}\Dmd{\ell}_{X_3}\Dmd{\ell}_T^{-2}\in \Lam_4^\x. \eeq
Here the set $\frakX^\bal_{\Lam_4}$ consists of $(\ulQ,P)=(Q_1,Q_2,Q_3,P)\in(\frakX_\Lam^+)^3\times\frakX_\Lam^{}\subset \Spec\Lam_4(\Qbarp)$ such that $(k_\Qx,k_\Qy,k_\Qz)$ is balanced and $k_P$ is critical for $(k_\Qx,k_\Qy,k_\Qz)$. 

\begin{prop}\label{P:GLam.E} For every $(\ulQ,P)\in\frakX^\bal_{\Lam_4}$, we have
\[\cG_{\ul{\chi}}^{(a)}(\ulQ,P)
=G_{k_\Qx,k_\Qy,k_\Qz}^{[k_P]}(\ep_P\Om^{a-k_P},\chi_1^{-1}\ep_\Qx^{-1}\Om^{k_\Qx},\chi_2^{-1}\ep_\Qy^{-1}\Om^{k_\Qy},\chi_3^{-1}\ep_\Qz^{-1}\Om^{k_\Qz}). \]
In particular, this implies that
\[\cG_{\ul{\chi}}^{(a)}\in \bfS^\Ord(N,\chi_1,\cO\powerseries{X_1})\wh\ot_{\cO} \bfS^\Ord(N,\chi_2,\cO\powerseries{X_2})\wh\ot_{\cO}\bfS^\Ord(N,\chi_3,\cO\powerseries{X_3})\wh\ot_{\cO}\cO\powerseries{T}. \]
\end{prop}
\begin{proof}
Set $\chi:=\ep_P\Om^{a-k_P}$, $\om_i=\chi_i^{-1}\ep_{Q_i}^{-1}\Om^{k_{Q_i}}$ and $\hat\om=\om_1\om_2\om_3$. 
One can check that
\begin{align*}
\cQ_B^{(a)}(\ulQ,P)&=\frac{(\eps_P\Om^a)(8b_{12}b_{23}b_{13})\Dmd{8b_{12}b_{23}b_{13}}^{k_P}}{(\chi_1\eps_{Q_1})(2b_{23})(\chi_2\eps_{Q_2})(2b_{13})(\chi_3\eps_{Q_3})(2b_{12})\Dmd{2b_{23}}^{k_{Q_1}}\Dmd{2b_{31}}^{k_{Q_2}}\Dmd{2b_{12}}^{k_{Q_3}}}\\
&=\cQ_B(\chi\cyc^{k_P},\om_1\cyc^{-k_\Qx},\om_2\cyc^{-k_\Qy},\om_3\cyc^{-k_\Qz}), \\
\Dmd{\ell}^{(a)}_{X_1,X_2,X_3,T}(\ulQ,P)&=(\Om^{-2a}\chi_1\chi_2\chi_3)(\ell)\ell^{-2}\cdot(\eps_{Q_1}\eps_{Q_2}\eps_{Q_3}\eps_P^{-2}\Om^{2k_P-k_{Q_1}-k_{Q_2}-k_{Q_3}})(\ell)^{-1}\ell^{k_{Q_1}+k_{Q_2}+k_{Q_3}-2k_P}\\
&=\chi_\ell^2(\ell)\abs{\ell}^{2k_P+2}\hat\om_\ell(\ell)\abs{\ell}^{-(k_\Qx+k_\Qy+k_\Qz)},\\
\cF_B^{(a)}(\ulQ,P)&=a_B(\chi^2\hat\om,2k_P-(k_\Qx+k_\Qy+k_\Qz)+4) 
\end{align*}
(see Definition \ref{def:psection} of $\cQ_B$). 
Recall our convention in \eqref{E:convention1} that $\chi_\ell(\ell)=\iota_p(\chi(\ell))^{-1}$ and $\hat\ome_\ell(\ell)=\iota_p(\hat\ome(\ell))^{-1}$. From \corref{C:Gmod.E}, we deduce the interpolation formula and that  
\beq\label{E:P67.1}\cG_{\ul{\chi}}^{(a)}(\ulQ,P)\in \sS^\Ord_{k_{Q_1}}(N,\ome_1^{-1};\cO(Q_1))\wh\ot_{\cO} \sS^\Ord_{k_{Q_2}}(N,\ome_2^{-1};\cO(Q_2))\wh\ot_{\cO}\sS^\Ord_{k_{Q_3}}(N,\ome_3^{-1};\cO(Q_3))\wh\ot_{\cO}\cO(P).\eeq
By the control theorem for ordinary $\Lam$-adic forms \cite[Theorem 3, p.215]{Hida93Blue}, for any arithmetic point $Q$, the specialization map $X\mapsto \bfu^{k_Q}\eps_Q(\bfu)-1$ yields an isomorphism 
\[\bfS^\Ord(N,\chi,\cO\powerseries{X})/(1+X-\bfu^{k_Q}\eps_Q(\bfu))\iso\sS^\Ord_{k_{Q}}(N,\chi\Om^{-k_Q}\eps_Q;\cO(Q)) .\]
 Hence, from \eqref{E:P67.1} we find that for all $P$ with $k_P=2$ \[\cG_{\ul{\chi}}^{(a)}(X_1,X_2,X_3,P)\in \bfS^\Ord(N,\chi_1,\cO\powerseries{X_1})\wh\ot_{\cO} \bfS^\Ord(N,\chi_2,\cO\powerseries{X_2})\wh\ot_{\cO}\bfS^\Ord(N,\chi_3,\cO\powerseries{X_3})\ot_{\cO}\cO(P).\]
Now we can deduce the second statement from the above equation combined with the argument in \cite[Lemma 1, page 328]{Hida93Blue}. 
\end{proof}


\section{Four-variable $p$-adic triple product $L$-functions}



\subsection{Measures}

The Tamagawa measures $\d g$ on $\PGL_2(\bfA)$ and $\d g'$ on $\SL_2(\bfA)$ are given by $\d g=\zet_\bfQ(2)^{-1}\prod_v\d g_v$ and $\d g'=\zet_\bfQ(2)^{-1}\prod_v\d g_v'$. 
Since $Z\bsl H\simeq \PGL_2\times \SL_2\times\SL_2$, we can define the Tamagawa measure on $Z\bsl H$ by $\d g_1^{}\d g'_2\d g_3'$, where $\d g_1$ is the Tamagawa measure on $\PGL_2(\A)$ and $\d g_2'=\d g_3'$ are that on $\SL_2(\A)$. 
The Tamagawa numbers of $\PGL_2$, $\SL_2$ and $Z\bsl H$ are $2$, $1$ and $2$, respectively {(\cf \cite[Lemma 6.1.1]{IP21}).} 



\subsection{Garrett's integral representation}

Let $\pi_i$ ($i=1,2,3$) be an irreducible cuspidal automorphic representation of $\GL_2(\A)$ generated by an elliptic cusp form of weight $k_i$ and nebentypus $\ome_i^{-1}$. 
Put $\hat\ome=\ome_1\ome_2\ome_3$ and $\breve\pi_i=\pi_i^{}\otimes\ome_{i,\bfA}^{-1}$ for $i=1,2,3$. 
Fix a character $\chi_\bfA$ of $\bfA^\times/\bfQ^\times\bfR_+$. 
For each triplet of cusp forms $\vph_i\in\breve\pi_i$ and a holomorphic section $f_s$ of $I_3(\hat\ome^{-1}_\bfA,\chi_\bfA^{}\hat\ome_\bfA^{}\Abs_\A^s)$ we consider the global zeta integral defined by 
\[Z(\vph_1,\vph_2,\vph_3,E_\bfA(-,f_s))=\int_{Z(\A)H(\Q)\bsl H(\A)}\vph_1(g_1)\vph_2(g_2)\vph_3(g_3)E_\A(\iota(g_1,g_2,g_3),f_s)\,\d g_1^{}\d g_2'\d g_3'. \]
The integral converges absolutely for all $s$ away from the poles of the Eisenstein series and is hence meromorphic in $s$.
Unfolding the Eisenstein series as in \cite{PSR87}, we get 
\[Z(\vph_1,\vph_2,\vph_3,E_\bfA(-,f_s))=\int_{Z(\A)U^0(\A)\bsl H(\A)}W(g_1,\vph_1)W(g_2,\vph_2)W(g_3,\vph_3)f_s({\Eta}\iota(g_1,g_2,g_3))\,\d g_1^{}\d g_2'\d g_3'. \]
If $W(g,\vph_i)=\prod_vW_{i,v}(g_v)$ and $f_s(g)=\prod_v f_{s,v}(g_v)$ are factorizable, then the integral factors into a product of local integrals and so by \S \ref{ssec:unram} 
\begin{align*}
Z(\vph_1,\vph_2,\vph_3,E_\bfA(-,f_s))
&=\frac{\zet_\bfQ(2)^{-3}L\left(s+\frac{1}{2},\pi_1\times\pi_2\times\pi_3\otimes\chi_\bfA\right)}{L^S(2s+2,\chi^2_\bfA\hat\ome^{}_\bfA)L^S(4s+2,\chi_\bfA^4\hat\ome_\bfA^2)}\prod_{v\in S}\frac{Z(W_{1,v},W_{2,v},W_{3,v},f_{s,v})}{L\left(s+\frac{1}{2},\pi_{1,v}\times\pi_{2,v}\times\pi_{3,v}\otimes\chi_v\right)},  
\end{align*}
where $S$ is a large enough set of places such that $\pi_{i,\ell}$, $W_{i,\ell}$, $\chi_\ell$ and $f_{s,\ell}$ are unramified for all $\ell\notin S$. 
The complete $L$-function $L(s,\pi_1\times\pi_2\times\pi_3\otimes\chi_\bfA)$ admits meromorphic continuation and a functional equation 
\[L(s,\pi_1\times\pi_2\times\pi_3\otimes\chi_\bfA)=\vep(s,\pi_1\times\pi_2\times\pi_3\otimes\chi_\bfA)L(1-s,\pi_1\times\pi_2\times\pi_3\otimes\hat\ome_\bfA^{-1}\chi_\bfA^{-1}). \]
By Theorem 2.7 of \cite{Ikeda2} the $L$-function $L(s,\pi_1\times\pi_2\times\pi_3\otimes\chi_\bfA)$ has a pole if and only if there exists an imaginary quadratic field $E$ and characters of $\chi_i$ of $\A_E^\times/E^\times$ such that $\chi_1\chi_2\chi_3\chi^E=1$ and such that $\pi_i$ is induced automorphically from $\chi_i$, where $\chi^E$ denotes the base change of $\chi$ to $E$. 
Recall that $k^*=\max\{k_1,k_2,k_3\}$. 
In particular, if $k_1+k_2+k_3\geq 2k^*+2$, then $L(s,\pi_1\times\pi_2\times\pi_3\otimes\chi_\bfA)$ is holomorphic everywhere. 
Let us put 
\begin{align*}
\cJ_\infty&=\pDII{-1}{1}\in\GL_2(\R), & 
t_n&=\pMX{0}{p^{-n}}{-p^n}{0}\in\GL_2(\Qp)\hookrightarrow\GL_2(\A). 
\end{align*}
 Let $\lam$ be the parity type of $(k_1,k_2,k_3)$. Suppose that $k^*=k_1$ and $k_1< k_2+k_3$. Let $E^\star_\A\bigl(-,f^{[k_1,\gap]}_{\cD,s,N}\bigl)$ be the Eisenstein series 
defined in \eqref{E:ES1}.

\begin{lm}\label{lem:Garrett} 
Let $f_i\in\sS_{k_i}(N_i,\ome_i^{-1})$ be an ordinary $p$-stabilized newform. 
Put 
\begin{align*}
\varphi_i&=\varPhi(f_i), & 
\breve\varphi_i^{}&=\varphi_i^{}\otimes\ome_{i,\bfA}^{-1}, & 
W(\vph_i)&=\prod_v W_{i,v}, & 
\breve W_{i,v}&=W_{i,v}\ot\om_{i,v}^{-1}, &   
\end{align*}
Let $\chi$ be a character of $\bfZ_p^\times$ of finite order. 
Put {$n=\max\{1,c(\chi),c(\ome_1),c(\ome_2),c(\ome_3)\}$}. 
Let $N$ be the least common multiple of $N_1,N_2,N_3$. 
If $N$ is square-free, then 
\begin{align*}
&Z\Big(\rho(\cJ_\infty t_n)\breve\vph_1,\rho(\cJ_\infty t_n)V_+^\frac{k_1-k_2-\gap_2}{2}\breve\vph_2,\rho(\cJ_\infty t_n)V_+^\frac{k_1-k_3-\gap_3}{2}\breve\vph_3,E^\star_\A\bigl(-,f^{[k_1,\gap]}_{\cD,s,N}\bigl)\Big)\\
=&L^{(N)}\left(s+\frac{1}{2},\pi_1\times\pi_2\times\pi_3\otimes\chi_\bfA\right)E_p\left(s+\frac{1}{2},\pi_{1,p}\times\pi_{2,p}\times\pi_{3,p}\otimes\chi_p\right)\\
&\times\prod_{i=1}^3\frac{\zeta_p(2)}{\zeta_p(1)}\left(\frac{\alp_p(f_i)^2}{p^{k_i}\om_{i,p}(p)}\right)^n\frac{\prod_{\ell|N} Z(\breve W_{1,\ell},\breve W_{2,\ell},\breve W_{3,\ell},f_{s,\ell})}{\zet_\bfQ(2)^32^{5+(k_1+k_2+k_3)}}
\end{align*}
\end{lm}

\begin{proof}
By Garrett's integral representation of triple $L$-functions the left hand side equals
\begin{align*}
&\zet_\bfQ(2)^{-3}L^{(\infty pN)}\left(s+\frac{1}{2},\pi_1\times\pi_2\times\pi_3\otimes\chi_\bfA\right)\\
\times&\gamma_{(k_1,k_2,k_3)}^\star(s)^{-1}Z\Big(\rho(\cJ_\infty)\breve W_{1,\infty},\rho(\cJ_\infty)V_+^\frac{k_1-k_2-\gap_2}{2}\breve W_{2,\infty},\rho(\cJ_\infty)V_+^\frac{k_1-k_3-\gap_3}{2}\breve W_{3,\infty},f^{[k_1,\gap]}_{s,\infty}\Big)\\
\times& Z(\rho(t_n)\breve W_{1,p},\rho(t_n)\breve W_{2,p},\rho(t_n)\breve W_{3,p},f_{\cD,s,p})\prod_{\ell|N} Z(\breve W_{1,\ell},\breve W_{2,\ell},\breve W_{3,\ell},f_{s,\ell})
\end{align*}
in view of Definition \ref{D:nEis.E} of $E^\star_\A\bigl(f^{[k_1,\gap]}_{\cD,s}\bigl)$. 
Since $\breve W_{i,\infty}=W_{i,\infty}$, \lmref{L:arch.E} yields 
\begin{align*}
&Z\Big(\rho(\cJ_\infty)\breve W_{1,\infty},\rho(\cJ_\infty)V_+^\frac{k_1-k_2-\gap_2}{2}\breve W_{2,\infty},\rho(\cJ_\infty)V_+^\frac{k_1-k_3-\gap_3}{2}\breve W_{3,\infty},f^{[k_1,\gap]}_{s,\infty}\Big)\\
=&Z_\infty(s)
=\chi_\infty(-1)\frac{\gamma_{(k_1,k_2,k_3)}^\star(s)}{2^{5+(k_1+k_2+k_3)}}L\left(s+\frac{1}{2},\sig_{k_1}\times\sig_{k_2}\times\sig_{k_3}\otimes\chi_\infty\right). 
\end{align*}
Proposition \ref{prop:13} calculates the $p$-adic part:   
\begin{align*}
&\frac{Z(\rho(t_n)\breve W_{1,p},\rho(t_n)\breve W_{2,p},\rho(t_n)\breve W_{3,p},f_{\cD,s,p})}{L\left(s+\frac{1}{2},\pi_{1,p}\times\pi_{2,p}\times\pi_{3,p}\otimes\chi_p\right)}\prod_{i=1}^3\frac{\zeta_p(1)}{\zeta_p(2)} \left(\frac{p^{k_i}\om_{i,p}(p)}{\alp_p(f_i)^2}\right)^n\\
=&Z^*_p(f_{\cD,s,p})=\chi_p(-1)E_p\left(s+\frac{1}{2},\pi_{1,p}\times\pi_{2,p}\times\pi_{3,p}\otimes\chi_p\right). 
\end{align*}
Since $\chi_\bfA$ is unramified outside $p$, we have $\chi_\infty(-1)=\chi_p(-1)$. 
\end{proof}


\subsection{The congruence number}

Put $\Delta=(\Z/Np\Z)^\x$. 
Let $\widehat\Del$ be the group of Dirichlet characters modulo $Np$. 
Enlarging $\calo$ if necessary, we assume that every $\chi\in\widehat\Del$ takes value in $\calo^\x$. 
Let 
\[\bfS^\Ord(N,\bfI):=\oplus_{\chi\in\widehat\Del}\bfS^\Ord(N,\chi,\bfI)\]
be the space of ordinary $\bfI$-adic cusp forms of tame level $\Gam_1(N)$. 
Let $\sig_d$ denote the usual diamond operator for $d\in\Del$ acting on $\bfS^\Ord(N,\bfI)$ by $\sig_d(\bdsf)_{\chi\in\widehat\Del}=(\chi(d)\bdsf)_{\chi\in\widehat\Del}$. 
The ordinary $\bfI$-adic cuspidal Hecke algebra $\bfT(N,\bfI)$ is defined as the $\bfI$-subalgebra of $\End_\bfI\bfS^\Ord(N,\bfI)$ generated over $\bfI$ by the Hecke operators $T_\ell$ with $\ell\nmid Np$, the operators $\bfU_\ell$ with $\ell|Np$ and the diamond operators $\sig_d$ with $d\in\Del$. 
Let $\rmT^\Ord_k(N,\chi)$ denote the $\calo$-subalgebra of $\End_\cO\eord\sS_k(N,\chi)$ generated over $\calo$ by the operators $T_\ell$ with $\ell\nmid Np$ and $\bfU_\ell$ with $\ell|Np$. For any $p$-stabilized newform $f$ in $\eord\sS_k(N,\chi;\cO)$, let $\lam_f:\rmT^\Ord_k(N,\chi)\to\cO$ be the homomorphism correponding to $f$ such that $\lam_f(T_\ell)=\bfa(\ell,f)$ if $\ell\ndivides Np$, $\lam_f(\bfU_\ell)=\bfa(\ell,f)$ for $\ell\divides Np$. Let $1_f\in\rmT^\Ord_k(N,\chi)\ot\Frac\cO$ be the idempotent with $\lam_f(1_f)=1$.

Let $\bdsf\in\bfS^\Ord(N,\chi,\bfI)$ be a primitive Hida family of tame conductor $N$ and character $\chi$. 
The corresponding homomorphism $\lam_\bdsf: \bfT(N,\bfI)\to\bfI$ is defined by $\lam_\bdsf(T_\ell)=\bfa(\ell,\bdsf)$ for $\ell\ndivides Np$, {$\lam_\bdsf(\bfU_\ell)=\bfa(\ell,\bdsf)$ for $\ell\divides Np$} and $\lam_\bdsf(\sg_d)=\chi(d)$ for  $d\in\Delta$. 
We denote by $\frakm_\bdsf$ the maximal ideal of $\bfT(N,\bfI)$ containing $\Ker\lam_\bdsf$ and by $\bfT_{\frakm_\bdsf}$ the localization of $\bfT(N,\bfI)$ at $\frakm_\bdsf$. It is the local ring of $\bfT(N,\bfI)$ through which $\lam_\bdsf$ factors. It is well-known that $\bfT_{\frakm_\bdsf}$ is a local finite flat $\Lam$-algebra, and there is an algebra direct sum decomposition 
\begin{align}\label{E:HeckeDecom}
\wtd\lam_\bdsf&:\bfT_{\frakm_\bdsf}\ot_{\bfI}\Frac\bfI\iso\Frac\bfI\oplus \sB, & 
t&\mapsto \wtd\lam_\bdsf(t)=(\lam_\bdsf(t),\lam_\sB(t)),
\end{align} 
where $\sB$ is some finite dimensional $(\Frac\bfI)$-algebra (\cite[Corollary 3.7]{Hida88Annals}). Let \[1_{\bdsf}=\wtd\lam_\bdsf^{-1}((1,0))\in \bfT_{\frakm_\bdsf}\ot_{\bfI}\Frac\bfI\] be the idempotent corresponding to the direct summand $\Frac\bfI$ in the above decomposition. Recall that the congruence ideal $C(\bdsf)$ of the morphism $\lam_\bdsf:\bfT_{\frakm_\bdsf}\to \bfI$ is defined by 
\[C(\bdsf):=\lam_\bdsf(\Ann_{\bfT_{\frakm_\bdsf}}(\Ker\lam_\bdsf))\subset\bfI.\]
 Note that by definition $H\cdot 1_\bdsf\in \bfT_{\frakm_\bdsf}$ for any $H\in C(\bdsf)$ and 
\[C(\bdsf)=\lam_\bdsf(\bfT_{\frakm_\bdsf}\cap \wtd\lam_\bdsf^{-1}(\Frac\bfI\oplus\stt{0})).\]
Now we consider the following hypothesis:  

\begin{hypothesis*}[CR]  The residual Galois representation $\ol{\rho}_{\bdsf}$ of $\rho_{\bdsf}$ is absolutely irreducible and $p$-distinguished. 
\end{hypothesis*}
 
Suppose that $p>3$. Under the hypothesis (CR), $\bfT_{\frakm_\bdsf}$ is Gorenstein by \cite[Corollary 2, page 482]{Wiles95}. With this property of $\bfT_{\frakm_\bdsf}$ Hida in \cite{Hida88AJM} proved that the congruence ideal $C(\bdsf)$ is generated by a non-zero element $\eta_\bdsf\in\bfI$, called the congruence number for $\bdsf$. Let $1^*_\bdsf$ be the unique element in $\bfT_{\frakm_\bdsf}\cap \wtd\lam_\bdsf^{-1}(\Frac\bfI\oplus\stt{0})$ such that $\lam_\bdsf(1^*_\bdsf)=\eta_\bdsf$. Then $1_\bdsf^*=\eta_\bdsf\cdot 1_\bdsf$ {by definition}. For $Q\in\frakX_\bfI^+$ and $\chi\in\widehat\Delta$, we write $\wp_{Q,\chi}$ for the ideal of $\bfT(N,\bfI)$ generated by $\wp_Q=\Ker Q$ and $\{\sig_d-\chi(d)\}_{d\in\Del}$. A classical result in Hida theory asserts that 
\[\bfT(N,\bfI)/\wp_{Q,\chi}\simeq\rmT^\Ord_{k_Q}(Np^e,\chi\Om^{-k_Q}\ep_Q)\otimes_\calo\calo(Q) \]
(see Theorem 3.4 of \cite{Hida88Annals}). 
Moreover, for each arithmetic point $Q$, it is also shown by Hida that the specialization $\eta_\bdsf(Q)\in\cO(Q)$ is the congruence number for $\bdsf_Q$ and 
\[1_{\bdsf_Q}=\eta_{\bdsf}^{-1}1^*_{\bdsf}\pmod{\wp_{\chi,Q}}\in \rmT^\Ord_{k_Q}(Np^e,\chi\Om^{-k_Q}\ep_Q)\ot_\cO\Frac\cO(Q)\] 
is the idempotent with $\lam_{\bdsf_Q}(1_{\bdsf_Q})=1$.

\begin{defn}
Let $\bdsf$ be a primitive Hida family satisfying (CR). 
To each choice of the congruence number $\eta_\bdsf$ we associate Hida's canonical period $\Omega_f$ of a $p$-ordinary newform $f$ of weight $k$ obtained by the specialization of $\bdsf$ defined by 
\[\Omega_{f}:=\eta_f^{-1}\cdot (-2\sqrt{-1})^{k+1}\norm{f^\circ}^2_{\Gamma_0(N_{f^\circ})}\cdot \cE_p(f,\Ad),\]
where $\eta_f$ is the specialization of $\eta_\bdsf$, $f^\circ$ the primitive form associated with $f$, $N_{f^\circ}$ its conductor and $\cE_p(f,\Ad)$ the modified $p$-Euler factor attached to the adjoint motive of $f$ (see \S \ref{S:period.1}). 
\end{defn} 


\subsection{Hida's functional}

When $\vph\in\cala_k(N,\ome_\bfA)$ and $\vph'\in\cala_k(N,\ome_\bfA^{-1})$ are cuspidal, we define the pairing by 
\[\pair{\rho(\cJ_\infty)\vph}{\vph'}=\int_{\bfA^\times\GL_2(\bfQ)\bsl\GL_2(\bfA)}\vph(g\cJ_\infty)\vph'(g)\, \d g. \]
Let $f\in\sS_k(N_f,\ome^{-1})$ be an ordinary $p$-stabilized newform of level  $N_f$, i.e., $\bfU_pf=\alp_p(f)f$ with $p$-unit $\iota_p^{-1}(\alp_p(f))$. Write $N_f=N_{\rm t}p^c$ with $N_{\rm t}$ prime to  $p$. For $n\geq c$, 
 we define Hida's functional $L_f$ on $\sS_k(N_{\rm t}p^{2n},\ome;\calo)$ by 
\[L_f(\calf)=\left(\frac{\om_p(p)p^k}{\alp_p(f)^2}\right)^{n-1}\frac{\pair{\rho(\cJ_\infty t_n)(\varphi^{}\ot\ome_\bfA^{-1})}{\varPhi(\calf)}}{\pair{\rho(\cJ_\infty t_1)(\varphi^{}\ot\ome_\bfA^{-1})}{\varphi}}, \]
where $\varphi=\varPhi(f)$ is the ad\`{e}lic lift of $f$. 
Note that for $\calf\in\sS_k(Np^{2n},\ome)$ with $N_{\rm t }\divides  N$, 
\[L_f(\calf)=[\Gam_0(N):\Gam_0(N_{\rm t})]^{-1}L_f(\Tr_{N/N_{\rm t}}\calf). \]

\begin{lm}\label{lem:Hidafunct}
\begin{enumerate}
\item\label{lem:Hidafunct1} $L_f(f)=1$. 
\item\label{lem:Hidafunct2} If $\calf_0\in\caln_{k+2m}(Np^{2n},\ome)$ with $N_{\rm t}\divides N$, then 
\[\frac{L_f(\bfone_f^*\Tr_{N/N_{\rm t}}\eord\Hol(\LDiff_z^m\calf_0))}{\zet_\bfQ(2)[\SL_2(\bfZ):\Gam_0(N)]}=(-1)^{m+1}(2\sqrt{-1})^{k+1}\frac{\pair{\rho(\cJ_\infty t_n)(V_+^m\varphi^{}\ot\ome_\bfA^{-1})}{\varPhi(\calf_0)}}{\Ome_f \left(\frac{\alp_p(f)^2}{p^k\om_{p}(p)}\right)^n\frac{\zeta_p(2)}{\zeta_p(1)}}. \]
\end{enumerate}
\end{lm}

\begin{proof}
The first assertion follows from the following formula stated in \cite[Lemma 3.6]{Hsieh_triple}: 
\begin{align*}
\pair{\rho(\cJ_\infty t_n)(\varphi^{}\ot\ome_\bfA^{-1})}{\varphi}
&=\om_{\infty}(-1)\pair{\rho(\cJ_\infty t_n)\varphi\ot\ome_\bfA^{-1}}{\varphi}\\
&=\frac{(-1)^k\zeta_\Q(2)^{-1}}{[\SL_2(\Z):\Gamma_0(N_{\rm t})]}\cdot\norm{f^\circ}^2_{\Gamma_0(N_{f^\circ})}\cdot\cE_p(f,\Ad)\cdot\frac{\alp_p(f)^{2n}\zeta_p(2)}{p^{kn}\om_p(p)^n\zeta_p(1)}\\
&=-\frac{\zeta_\Q(2)^{-1}}{[\SL_2(\Z):\Gamma_0(N_{\rm t})]}\cdot\frac{\eta_f\Ome_f}{(2\sqrt{-1})^{k+1}}\cdot\left(\frac{\alp_p(f)^2}{p^k\om_p(p)}\right)^n\frac{\zeta_p(2)}{\zeta_p(1)}. 
\end{align*} 
We remark that $\breve\vph=\vph$ and $\ome_{(p)}=\ome_\bfA$ in the notation of \cite{Hsieh_triple}. 

To see the second part, we note that as a consequence of strong multiplicity one theorem for elliptic modular forms, the idempotent $1_f$ is generated by the Hecke operators $T_\ell$ with $\ell\nmid Np$, which implies that $1_f$ is the adjoint operator of $1_{\vph\otimes\ome_\bfA^{-1}}$ with respect to the pairing. 
We are thus led to $L_f(\bfone_f^*\calf)=\eta_fL_f(\calf)$. 
Moreover, $L_f(\bfU_p\calf)=\alp_p(f)L_f(\calf)$
(\cf the proof of Proposition 2.10 of \cite{K13}) and hence 
\[L_f(\eord\calf)=\lim_{j\to\infty}L_f(\bfU_p^{j!}\calf)=\lim_{j\to\infty}\alp_p(f)^{j!}L_f(\calf)=L_f(\calf). \]
One can easily verify that for $\phi\in\varPhi(\sS_k(M,\chi^{-1}))$, $\calf_1\in\caln_k(M,\chi)$ and $\calf_2\in\caln_{k+2}(M,\chi)$
\begin{align*}
\pair{\rho(\cJ_\infty)\phi}{\varPhi(\Hol\calf_1)}&=\pair{\rho(\cJ_\infty)\phi}{\varPhi(\calf_1)}, &
\pair{\rho(\cJ_\infty)\phi}{\varPhi(\LDiff_z\calf_2)}&=-\pair{\rho(\cJ_\infty)V_+\phi}{\varPhi(\calf_2)}. 
\end{align*}
The second part is a consequence of these results. 
\end{proof}


\subsection{The construction of $p$-adic $L$-functions}
Let 
\[\bdsF=(\bdsf,\bdsg,\bdsh)\in \bfS^\Ord(N_1,\chi_1,\bfI)\times \bfS^\Ord(N_2,\chi_2,\bfI) \times \bfS^\Ord(N_3,\chi_3,\bfI)\] 
be a triplet of primitive $\bfI$-adic Hida families of tame square-free level $(N_1,N_2,N_3)$ and tame characters $(\chi_1,\chi_2,\chi_3)$, where $\bfI$ is a finite flat domain over $\Lam=\calo\powerseries{\Gamma}$. 
Assuming that $p>3$ and that all $\bdsf$, $\bdsg$ and $\bdsh$ satisfy Hypothesis (CR), we fix a choice of the congruence numbers $(\eta_\bdsf,\eta_\bdsg,\eta_\bdsh)$. 
Let 
\begin{align*}
\bfone^*_\bdsf&\in \bfT(N_1,\bfI), & 
\bfone^*_\bdsg&\in \bfT(N_2,\bfI), & 
\bfone_\bdsh^*&\in \bfT(N_3,\bfI)
\end{align*} 
be the idempotents multiplied by a fixed choice of congruence numbers $(\eta_\bdsf,\eta_\bdsg,\eta_\bdsh)$ in the Hecke algebras attached to the newforms $(\bdsf,\bdsg,\bdsh)$. 
Put 
\begin{align*}
N&:=\lcm(N_1,N_2,N_3), & 
N^-&:=\gcd(N_1,N_2,N_3), & 
\bfI_3&:=\bfI\wh\ot_{\cO}\bfI\wh\ot_{\cO}\bfI. 
\end{align*}

\begin{defn}
Define the $p$-adic triple product $L$-function $L_{\bdsF,(a)}$ in $\bfI_3\powerseries{T}$ by 
\[L_{\bdsF,(a)}:=\text{ the first Fourier coefficient of }\bfone^*_\bdsf\ot\bfone^*_\bdsg\ot\bfone^*_\bdsh(\Tr_{N/N_1}\ot\Tr_{N/N_2}\ot\Tr_{N/N_3}(\cG_{\ul{\chi}}^{(a)}))\in\bfI_3\powerseries{T}.\]
\end{defn}
We proceed to show the interpolation $L_{\bdsF,(a)}(\ulQ,P)$ at $(\ulQ,P)\in \frakX_{\bfI_4}^{\rm bal}$ is given by critical values of motivic $L$-functions associated with triple product of elliptic modular forms.
\subsection{The interpolation formulae}
 In the notation of the introduction, we let $\cV=V_\bdsf\wh\ot_\calo V_\bdsg\wh\ot_\calo V_\bdsh$ and $\bfV=\cV\wh\ot_\cO\Om^a\Dmd{\cyc}_T$ be the triple tensor product of $\bfI$-adic Galois representations associated with primitive Hida families $\bdsf,\bdsg$ and $\bdsh$ twisted by $\Om^a\Dmd{\cyc}_T$.  Let $N_\bfV=N^-N^4$ be the tame conductor of $\bfV$ and let $t_\bfV$ denote the number of prime factors of $N^{-}$. Define the rank four $G_{\Qp}$-invariant subspace of $\bfV$ by 
\[\Fil^+\bfV=(\Fil^0V_\bdsf\ot \Fil^0 V_\bdsg\ot V_\bdsh+\Fil^0V_\bdsf\ot V_\bdsg\ot \Fil^0V_\bdsh+V_\bdsf\ot \Fil^0 V_\bdsg\ot \Fil^0V_\bdsh)\otimes\Om^a\Dmd{\cyc}_T.\]
For each $\ell|N^-$, let $\sqrt{\Dmd{\ell}_{X_1,X_2,X_3,T}^{(a)}}\in\bfI_4^\times$ be a square root of $\Dmd{\ell}_{X_1,X_2,X_3,T}^{(a)}$ in \eqref{E:LX.7} defined by
\[\sqrt{\Dmd{\ell}^{(a)}_{X_1,X_2,X_3,T}}:=(\Om^{-a}\xi_{\bdsf,\ell}\xi_{\bdsg,\ell}\xi_{\bdsh,\ell})(\ell)\ell^{-1}\cdot\Dmd{\ell}_{X_1}^{1/2}\Dmd{\ell}_{X_2}^{1/2}\Dmd{\ell}_{X_3}^{1/2}\Dmd{\ell}_T^{-1},\]
where $\xi_{\bdsf,\ell}$ (resp. $\xi_{\bdsg,\ell}$ and $\xi_{\bdsh,\ell}$) is the unramified character of $G_{\Q_\ell}$ in \eqref{E:Galois1}. Recall that $w_\ulQ=k_{Q_1}+k_{Q_2}+k_{Q_3}-3$ and $\Gamma_{\bfV_{(\ulQ,P)}}(s)=L_\infty\Big(s+k_P-\frac{w_\ulQ}{2},\pi_{\bdsf_\Qx}\times\pi_{\bdsg_\Qy}\times\pi_{\bdsh_\Qz}\Big)$. 
Let $\cE_p(\Fil^+\bfV_{(\ulQ,P)})$ be the modified $p$-Euler factor defined in \subsecref{S:mod.1}. 

\begin{thm}\label{P:interpolation}
Let $p>3$. 
Assume that $N:=\mathrm{lcm}(N_1,N_2,N_3)$ is square-free and that the conductor of tame nebentypus $\chi_i$ divides $p$. 
Let $t$ denote the number of prime factors of $N$. 
If $\bdsf$, $\bdsg$ and $\bdsh$ satisfy Hypothesis (CR), then for each arithmetic point $(\ulQ,P)=(\Qx,\Qy,\Qz,P)\in\frakX^\bal_{\bfI_4}$ we have 
\begin{align*}
L_{\bdsF,(a)}(\ulQ,P)=&\Gamma_{\bfV_{(\ulQ,P)}}(0)\cdot \frac{L(\bfV_{(\ulQ,P)},0)}{\Omega_{\bdsf_\Qx}\Omega_{\bdsg_\Qy}\Omega_{\bdsh_\Qz}}\cdot (\sqrt{-1})^{k_\Qx+k_\Qy+k_\Qz-3}\cdot \cE_p(\Fil^+\bfV_{(\ulQ,P)}) \cdot \frakf_{\ul{\chi},a,N_1,N_2,N_3}(\ulQ,P),
\end{align*}
where $\frakf_{\ul{\chi},a,N_1,N_2,N_3}\in \bfI_4^\x$ is the fudge  factor given by 
\[\frakf_{\ul{\chi},a,N_1,N_2,N_3}:=\frac{(-1)^{t-t_\bfV}}{N\prod_{\ell\divides N^-}\sqrt{\Dmd{\ell}^{(a)}_{X_1,X_2,X_3,T}}}.\]
\end{thm}

\begin{proof}
For brevity we write $(f_1,f_2,f_3)=(\bdsf_\Qx,\bdsg_\Qy,\bdsh_\Qz)$, $(k,l,m)=(k_\Qx,k_\Qy,k_\Qz)$, $\pi_i=\pi_{f_i}$ and $N_i=N_{f_i}$. 
We may assume that $k\geq l\geq m$. Then we must have $k< l+m-1$. Let $\gap$ be the parity type of $(k,l,m)$. Put 
\begin{align*}
\chi&=\ep_P\Om^{a-k_P}, &
\ome_i&=\Om^{k_{Q_i}}\chi_i^{-1}\ep_{Q_i}^{-1}, & 
\cD&=(\chi,\ome_1^{-1},\ome_2^{-1},\ome_3^{-1}), &
n&=\max\{1,c(\ome_i),c(\chi)\}. 
\end{align*}
We define the functional $L_{f_1,f_2,f_3}$ on 
\[\sS_k(N_1p^{2n},\ome_1^{-1};\calo(Q_1))\otimes_\calo\sS_l(N_2p^{2n},\ome_2^{-1};\calo(Q_2))\otimes_\calo\sS_m(N_3p^{2n},\ome_3^{-1};\calo(Q_3))\]
by 
\[L_{f_1,f_2,f_3}(\calf_1\otimes\calf_2\otimes\calf_3)=L_{f_1}(\calf_1)L_{f_2}(\calf_2)L_{f_3}(\calf_3). \]

Let $1^*_{f_1}$ be the specialization of $\bfone^*_{\bdsf}$ at $\Qx$. 
By definition and the theory of newforms 
\[1^*_{f_1}\ot1^*_{f_2}\ot1^*_{f_3}(\Tr_{N/N_1}\ot\Tr_{N/N_2}\ot\Tr_{N/N_3}(\cG_{\ul{\chi}}^{(a)}(\ulQ,P)))=L_{\bdsF,(a)}(\ulQ,P) \cdot f_1\ot f_2\ot f_3.\] 
We apply the functional $L_{f_1,f_2,f_3}$ to both the sides to get 
\[L_{\bdsF,(a)}(\ulQ,P)=L_{f_1,f_2,f_3}(1^*_{f_1}\ot1^*_{f_2}\ot1^*_{f_3}(\Tr_{N/N_1}\ot\Tr_{N/N_2}\ot\Tr_{N/N_3}(\cG_{\ul{\chi}}^{(a)}(\ulQ,P)))), \]
taking Lemma \ref{lem:Hidafunct}(\ref{lem:Hidafunct1}) into account. 
Let $\varphi_i=\varPhi(f_i)$ and $G_\A(\cD)=\varPhi(\cG_{\ul{\chi}}^{(a)}(\ulQ,P))$ be the ad\`{e}lic lifts. 
Put $\breve\varphi_i^{}=\varphi_i^{}\otimes\ome_{i,\bfA}^{-1}$. 
In the previous section we verified that 
\[G_\A(\cD)=\lim_{s\to-r+\frac{k-\gap_1}{2}-1}(-1)^{k+\frac{l+m+\gap_1}{2}+\gap_2}\eord\Hol\Big(\Big(1\ot V_-^\frac{k-l-\gap_2}{2}\ot V_-^\frac{k-m-\gap_3}{2}\Big)\iota^*E^\star_\A\bigl(-,f_{\cald,s,N}^{[k,\gap]}\bigl)\Big),\]
where $r=k-k_P+\frac{l+m-\gap_1}{2}-2$, where $E^\star_\A\bigl(-,f_{\cald,s,N}^{[k,\gap]})$ is the normalized Eisenstein series in \defref{D:nEis.E} (see \propref{P:GLam.E} and \defref{D:Gmod.E}). Therefore \lmref{lem:Hidafunct} (\ref{lem:Hidafunct2}) yields 
\begin{multline*}
\frac{L_{\bdsF,(a)}(\ulQ,P)}{\zet_\bfQ(2)^3[\SL_2(\bfZ):\Gam_0(N)]^3}
=-(2\sqrt{-1})^{k+l+m+3}\frac{\zeta_p(1)^3}{\zeta_p(2)^3}\prod_{i=1}^3\Ome_{f_i}^{-1} \left(\frac{p^{k_{Q_i}}\om_{i,p}(p)}{\alp_p(f_i)^2}\right)^n\\
\times 4\lim_{s\to k_P-\frac{k+l+m}{2}+1}Z\Big(\rho(\cJ_\infty t_n)\breve\vph_1,\rho(\cJ_\infty t_n)V_+^\frac{k-l-\gap_2}{2}\breve\vph_2,\rho(\cJ_\infty t_n)V_+^\frac{k-m-\gap_3}{2}\breve\vph_3,E^\star_\A\bigl(-,f_{\cald,s,N}^{[k,\gap]}\bigl)\Big).
\end{multline*}

Let $W(\vph_i)=\prod_v W_{i,v}$ be the Whittaker function of $\varphi_{i}$ and put $\breve W_{i,v}:=W_{i,v}\ot\om_{i,v}^{-1}$. 
Let $\pi_i$ be the automorphic representation generated by $\varphi_i$.
We obtain by Lemma \ref{lem:Garrett} that
\begin{align*}
L_{\bdsF,(a)}(\ulQ,P)
=\frac{L\left(k_P-\frac{k+l+m-3}{2},\pi_1\times\pi_2\times\pi_3\otimes\chi_\bfA\right)}{(\sqrt{-1})^{3-(k+l+m)}\Ome_{f_1}\Ome_{f_2}\Ome_{f_3}}\cE_p(\Fil^+\bfV_{(\ulQ,P)})\prod_{\ell|N}Z_\ell^* ,
\end{align*}
where 
\[Z_\ell^*=[\SL_2(\Z):\Gamma_0(\ell)]^3\lim_{s\to k_P-\frac{k+l+m}{2}+1}\frac{Z(\breve W_{1,\ell},\breve W_{2,\ell},\breve W_{3,\ell},f_{s,\ell}^*)}{L\left(s+\frac{1}{2},\pi_{1,\ell}\times\pi_{2,\ell}\times\pi_{3,\ell}\otimes\chi_\ell\right)}. \]
For $\ell\divides N$, \propref{prop:21} gives
\begin{align*}
Z^*_\ell&=-\ell(\hat\ome_\ell^2\chi_\ell^4)(\ell)|\ell|^{4k_P-2(k+l+m)+4}\vep\left(k_P-\frac{k+l+m-3}{2},\pi_{1,\ell}\times\pi_{2,\ell}\times\pi_{3,\ell}\otimes\chi_\ell,\addchar_\ell\right)^{-1}. 
\end{align*} 
By what we have seen in the proof of Proposition \ref{P:GLam.E}  
\[\chi_\ell^2\hat\om_\ell(\ell)\ell^{-2k_P+(k+l+m)-2}=\Dmd{\ell}_{X_1,X_2,X_3,T}(\ulQ,P). \]
The proof of Lemma \ref{lem:rootnumber} completes the proof.
\end{proof}

\begin{defn}
We normalize \padic triple product $L$-function by 
\[L_{\bdsF,(a)}^*:=L_{\bdsF,(a)}\cdot \frakf_{\ul{\chi},a,N_1,N_2,N_3}^{-1}.\]
\end{defn}

\begin{Remark}\label{rem:11}
Provided that $p>3$, $\chi_1\chi_2\chi_3=\Om^{2a}$ for some $a$, a three-variable $p$-adic $L$-function $\call^\bal_{\bdsF}\in\bfI_3$ was constructed by a different approach in \cite[Theorem B]{Hsieh_triple} such that for each balanced central point $\ulQ=(\Qx,\Qy,\Qz)\in\frakX^\bal_{\bfI_3}$ 
\[\left(\call^\bal_{\bdsF}(\ulQ)\right)^2=\Gamma_{\bfV_{\ulQ}}(0)\cdot\frac{L(\bfV^\dagger_{\ulQ},0)}{\Ome_{\bdsf_\Qx}\Ome_{\bdsg_\Qy}\Ome_{\bdsh_\Qz}}\cdot (\sqrt{-1})^{k_\Qx+k_\Qy+k_\Qz-3}\cdot \cE_p(\Fil^+\bfV^\dagger_{\ulQ}), \] 
where 
\begin{align*}
\bfV^\dagger&:=\calv\otimes\Om^a\Dmd{\cyc}_{\bfX_1}^{1/2}\Dmd{\cyc}_{\bfX_2}^{1/2}\Dmd{\cyc}_{\bfX_3}^{1/2}\cyc^{-1}, \\ 
\Fil^+\bfV^\dagger&=\Fil^+\calv\otimes\Om^a\Dmd{\cyc}_{\bfX_1}^{1/2}\Dmd{\cyc}_{\bfX_2}^{1/2}\Dmd{\cyc}_{\bfX_3}^{1/2}\cyc^{-1}.  
\end{align*} 
We remark that $\det V_\bdsf=(\chi_1\circ\cyc)^{-1}\Dmd{\cyc}_\bfI^{-1}\cyc$. By the interpolation formulae, we find that
\[L_{\bdsF,(a-1)}^*(X_1,X_2,X_3,\bfu^{-1}\{(1+X_1)(1+X_2)(1+X_3)\}^{1/2}-1)=\call^\bal_{\bdsF}(X_1,X_2,X_3)^2. \] 
This shows that the compatibility between $p$-adic $L$-functions constructed by different methods.
\end{Remark}

Without the assumption $p>3$ and Hypothesis (CR), our method yields the construction of the $p$-adic $L$-function with denominators. For each $p$-stabilized newform $f$ of weight $k$, define the $p$-modified period by \[\Omega^\flat_f:=(-2\sqrt{-1})^{k+1}\cdot\norm{f^\circ}_{\Gamma_0(N_{f^\circ})}^2\cdot \cE_p(f,\Ad).\]
By definition, $\Omega^\flat_f\cdot \eta_f$ is equal to Hida's canonical period $\Omega_f$ up to $p$-adic units.
\begin{cor}\label{C:78}Let $p$ be an arbitrary rational prime. There exists an element
\[L_{\bdsF,(a)}^{**}\in\bfI_4\otimes_{\bfI_3}(\Frac\bfI\otimes\Frac\bfI\otimes\Frac\bfI)\] 
such that 
\begin{itemize}
\item for each balanced critical $(\ulQ,P)=(\Qx,\Qy,\Qz,P)\in\frakX^\bal_{\bfI_4}$,
\begin{align*} 
L_{\bdsF,(a)}^{**}(\ulQ,P)=&\frac{\Gamma_{\bfV_{(\ulQ,P)}}(0)L(\bfV_{(\ulQ,P)},0)}{\Omega^\flat_{\bdsf_\Qx}\Omega^\flat_{\bdsg_\Qy}\Omega^\flat_{\bdsh_\Qz}}\cdot (\sqrt{-1})^{k_\Qx+k_\Qy+k_\Qz-3}\cdot \cE_p(\Fil^+\bfV_{(\ulQ,P)});
\end{align*}
\item for any $H_1$, $H_2$ and $H_3$ in the congruence ideals of $\bdsf,\bdsg$ and $\bdsh$, 
 \[H_1H_2H_3\cdot L_{\bdsF,(a)}^{**}\in \bfI_4.\]
 \end{itemize}\end{cor}
\begin{proof}For any $H_1$, $H_2$ and $H_3$ in the congruence ideals of $\bdsf,\bdsg$ and $\bdsh$, we let $L_{H}\in \bfI_3\powerseries{T}$ be the first Fourier coefficient of 
\[H_1\bfone_\bdsf\ot H_2\bfone_\bdsg\ot H_3\bfone_\bdsh\left( \Tr_{N/N_1}\ot\Tr_{N/N_2}\ot\Tr_{N/N_3}(\cG_{\ul{\chi}}^{(a)})\right)\in\bfI_3\powerseries{T}\powerseries{q_1,q_2,q_3}.\]
Then $L_{\bdsF,(a)}^{**}:=L_{H}\cdot (H_1H_2H_3)^{-1}\cdot \frakf_{\ul{\chi},a,N_1,N_2,N_3}^{-1}$ enjoys the desired properties.
\end{proof}
This $p$-adic $L$-function $L_{\bdsF,(a)}^{**}$ is more canonical in the sense that it does not depend on any particular choice of generators of the congruence ideal of $\bdsf$, $\bdsg$ and $\bdsh$. 


\subsection{The functional equation}\label{ssec:functeq}
We introduce the $\bfI_4$-adic epsilon factor and the functional equation of our $p$-adic $L$-functions. For each $(\ulQ,P)\in\frakX^\bal_{\bfI_4}$, the local epsilon factor of the triple tensor product representation $\bfV_{(\ulQ,P)}$ at $\ell\neq p$ is defined by 
\[\vep_\ell(\bfV_{(\ulQ,P)},s)=\vep(s+k_P-w_\ulQ/2,\WD_\ell(V_{\bdsf_{Q_1}})\otimes\WD_\ell(V_{\bdsg_{Q_2}})\otimes\WD_\ell(V_{\bdsh_{Q_3}})\otimes\Om^{a-k_P}\eps_P,\addchar_\ell). \]
Note that with the assumption \eqref{sf}, $\WD_\ell(V_{\bdsf_{Q_1}})$, $\WD_\ell(V_{\bdsg_{Q_2}})$, $\WD_\ell(V_{\bdsh_{Q_3}})$ are either unramified or the Steinberg representation twisted by an unramified character. We define the $\bfI_4$-adic epsilon factor $\vep^{(p\infty)}(\bfV)\in\bfI_4^\times$ by 
\beq\label{E:root}
\vep^{(p\infty)}(\bfV)=(-1)^{t_\bfV}\prod_{\ell|N}\Dmd{\ell^2}^{(a)}_{X_1,X_2,X_3,T}\prod_{\ell|N^-}\sqrt{\Dmd{\ell}^{(a)}_{X_1,X_2,X_3,T}} . \eeq
\begin{lm}[Interpolation of the epsilon factors]
\label{lem:rootnumber}Notation being as above, we get 
\[\vep^{(p\infty)}(\bfV_{(\ulQ,P)})=\prod_{\ell\neq p}\vep_\ell(\bfV_{(\ulQ,P)},0). \]
\end{lm}
\begin{proof}
We retain the notation of the proof of Proposition \ref{P:GLam.E}. 
Remark \ref{rem:localfactor} gives 
\[\vep_\ell(\bfV_{(\ulQ,P)},0)
=\chi_\ell(\ell)^4\hat\ome(\ell)^2\ell^{-4k_P+2(k_{Q_1}+k_{Q_2}+k_{Q_3})-4}
=\Dmd{\ell^2}_{X_1,X_2,X_3,T}^{(a)}(\ulQ,P)
\]
if $\ell$ divides $N/N^-$. If $\ell$ divides $N^-$,
putting $\xi_\ell=\xi_{\bdsf,\ell}\xi_{\bdsg,\ell}\xi_{\bdsh,\ell}$, we have   
\begin{align*}
\vep_\ell(\bfV_{(\ulQ,P)},0)&=-\chi_\ell(\ell)^5\xi_\ell(\Frob_\ell)^5\ell^{5(-2k_P+k_{Q_1}+k_{Q_2}+k_{Q_3}-2)/2}\\
&=-\Dmd{\ell^2}_{X_1,X_2,X_3,T}^{(a)}(\ulQ,P)\sqrt{\Dmd{\ell}_{X_1,X_2,X_3,T}^{(a)}}(\ulQ,P). 
\end{align*}
We have thus completed our proof. 
\end{proof}

Recall that we have fixed the topological generator $\bfu=1+\bfp$ of $\Gamma=1+\bfp\bfZ_p$ as in \subsecref{SS:6.5}. 
\begin{prop}\label{P:funceq}With the hypotheses in \propref{P:interpolation}, we further assume that $\chi_1\chi_2\chi_3=\Om^{a_0}$. 
Then 
\[L_{\bdsF,(a)}^*(X_1,X_2,X_3,T)=(-\vep^{(p\infty)}(\bfV))\cdot L_{\bdsF,(a_0-a-2)}^*\biggl(X_1,X_2,X_3,\frac{(1+X_1)(1+X_2)(1+X_3)}{\bfu^2(1+T)}-1\biggl). \]
\end{prop}

\begin{proof}
Recall that $\chi=\eps_P\Om^{a-k_P}$ and $\ome_i=\chi_i^{-1}\eps_{Q_i}^{-1}\Om^{k_{Q_i}}$. Put 
\begin{align*} 
k_{\breve P}&=k_{Q_1}+k_{Q_2}+k_{Q_3}-k_P-2, & 
\eps_{\breve P}&=\eps_P^{-1}\eps_{Q_1}^{}\eps_{Q_2}^{}\eps_{Q_3}^{}, &
\breve\chi&
=\eps_{\breve P}\Om^{a_0-a-2-k_{\breve P}}
=\chi^{-1}\ome_1^{-1}\ome_2^{-1}\ome_3^{-1}. 
\end{align*}
Thus the left hand side specialized at $(\ulQ,\breve P)$ equals 
\[L_{\bdsF,(a_0-a-2)}^*(\ulQ,\breve P)=\frac{L(1-s_0,\pi_1^\vee\times\pi_2^\vee\times\pi_3^\vee\otimes\chi_\bfA^{-1})}{(\sqrt{-1})^{3-(k_{Q_1}+k_{Q_2}+k_{Q_3})}\Ome_{f_1}\Ome_{f_2}\Ome_{f_3}}E_p\left(1-s_0,\pi_{1,p}\times\pi_{2,p}\times\pi_{3,p}\otimes\chi_p\right), \]
where $s_0=k_P-\frac{k_{Q_1}+k_{Q_2}+k_{Q_3}-3}{2}=1-\bigl(k_{\breve P}-\frac{k_{Q_1}+k_{Q_2}+k_{Q_3}-3}{2}\bigl)$. 

Since $(k_{Q_1},k_{Q_2},k_{Q_3})$ is balanced, we know that 
\[\vep(s,\pi_{1,\infty}\times\pi_{2,\infty}\times\pi_{3,\infty}\otimes\chi_\infty)=(-1)^{k_{Q_1}+k_{Q_2}+k_{Q_3}+1}=-\hat\ome_\infty(-1)=-\hat\ome_p(-1). \]
By the global functional equation we get 
\[L_{\bdsF,(a_0-a-2)}^*(\ulQ,\breve P)=\frac{L(s_0,\pi_1\times\pi_2\times\pi_3\otimes\chi_\bfA)}{(\sqrt{-1})^{3-(k_{Q_1}+k_{Q_2}+k_{Q_3})}\Ome_{f_1}\Ome_{f_2}\Ome_{f_3}}\cdot\frac{-E_p\left(s_0,\pi_{1,p}\times\pi_{2,p}\times\pi_{3,p}\otimes\breve\chi_p\right)}{\prod_{\ell\neq p}\vep(s_0,\pi_{1,\ell}\times\pi_{2,\ell}\times\pi_{3,\ell}\otimes\chi_\ell,\addchar_\ell)}\] 
in view of Lemma \ref{lem:modfactor}. 
\end{proof}


\section{The trivial zero for the triple product of elliptic curves}

\def\VEp{V_p(\boldsymbol{E})}

\subsection{The cyclotomic $p$-adic triple product $L$-functions for elliptic curves} 

Let $\bdsE=E_1\times E_2\times E_3$ be the triple fiber product of rational elliptic curves $E_i$ of square-free conductor $M_i$ for $i=1,2,3$. 
We denote by $N_i$ the {prime-to-$p$} part of $M_i$. 
Recall the rank eight $p$-adic Galois representation $\bfV_{\bdsE}$ defined in (\ref{tag:8dimGalois}). 
We write $L(\bdsE\otimes\chi,s)$ for the complex $L$-series attached to $\bfV_{\bdsE}$ twisted by a Dirichlet character $\chi$. 
Let $M$ (resp. $N$) and $M^-$ (resp. $N^-$) be the least common multiple and the greatest common divisor of $M_1,M_2,M_3$ (resp. $N_1,N_2,N_3$). 
\begin{Remark}\label{rem:sign}
Let $\Sigma^-$ be the set of prime factors $\ell$ of $M^-$ such that $a_\ell(E_1)a_\ell(E_2)a_\ell(E_3)=1$.  
From Remark \ref{rem:localfactor}, $\vep(\bdsE)=-(-1)^{\#\Sigma^-}$ is the sign in the functional equation for $L(s,\bdsE)$. 
From the formula \eqref{E:root} for the $p$-adic root number the $p$-adic sign $\vep_p(\bdsE)=-\vep^{(p\infty)}(\bfV_{\bdsE}(2))$ differs from $\vep(\bdsE)$ if and only if $p\in\Sigma^-$. 
  \end{Remark}

Let $f_i^\circ=\sum_{n=1}^\infty a_n(E_i)q^n\in\cals_2(M_i,1;\Z)$ be the primitive Hecke eigenform associated with the $p$-adic Galois representation $\rmH^1_{\et}(E_{i/\Qbar},\Qp)$ by Wiles' modularity theorem. 
Hereafter, we assume that $E_i$ has either good ordinary reduction or multiplicative reduction at $p$. 
Let $f_i\in \cals_2(pM_i,1;\Zp)$ be the $p$-stabilization of $f_i^\circ$ (see (\ref{tag:51})). 
If $p$ and $M_i$ are coprime, then $\al_i=\alp_p(f_i)\in\Z_p^\times$ denotes the $p$-adic unit root of the Hecke polynomial $X^2-a_p(E_i)X+p$ while if $p$ divides $M_i$, then $\al_i=a_p(E_i)$. 
Define a period and a fudge factor by 
\begin{align*}
\Ome(\bdsE)&=\prod_{i=1}^3\Lam(1,E_i,\Ad), &
c_p&=\prod_{i=1}^3\cE_p(f_i,\Ad), 
\end{align*}
where $\Lam(s,E_i,\Ad)$ denotes the complete adjoint $L$-function for $f_i$

Let $\bfT_i=\bfT(N_i,\Lam)$ be the big cuspidal ordinary Hecke algebra over $\Lam=\Zp\powerseries{X}$ with $X=[\bfu]-1$. Each $f_i$ induces a surjective homomorphism $\lam_{f_i}:\bfT_i\surjto \Zp$. Let $\frakm_i$ be the maximal ideal of $\bfT_i$ containing $\ker\lam_{f_i}$ and $\bfI_i=(\bfT_i)_{\frakm_i}$ be the localization at $\frakm_i$. Let $\bdsf_i=\sum_{n=1}^\infty \bfa(n,\bdsf_i)q^n\in \bfS^\ord(N_i,\Om^{2},\bfI_i)$ be the primitive Hida family of tame level $N_i$ such that $f_i$ is the specialization $\bdsf_{i,Q^o_i}$ at some arithmetic point $Q^o_i$ with $k_{Q^o_i}=2$ and $\eps_{Q^o_i}=1$ {(\cf\cite[Theorem 1.4.1]{Wiles88}). }
Now we consider the four-variable $p$-adic $L$-function $L_{\bdsF,(2)}^{**}$ in \corref{C:78} with  $\bdsF=(\bdsf_1,\bdsf_2,\bdsf_3)$ and $a=2$. Define the cyclotomic $p$-adic $L$-function by 
\[L_p(\bdsE,T):=c_p\cdot L_{\bdsF,(2)}^{**}(Q^o_1,Q^o_2,Q^o_3,\bfu^2(1+T)-1)\in\Zp\powerseries{\Gal(\Q_\infty/\Q)}\ot\Q_p. \]

\begin{prop}\label{P:cycpadicL.7}
The element $L_p(\bdsE)\in\Z_p\powerseries{\Gal(\Q_\infty/\Q)}\ot\Qp$ satisfies the following interpolation property
\[\hat\chi(L_p(\bdsE))=\frac{L(\bdsE\otimes\hat\chi,2)}{2^4\pi^5\Ome(\bdsE)}\cale_p(\Fil^+\bfV_{\bdsE}\otimes\hat\chi) 
\]
for all finite-order characters $\hat\chi$ of $\Gal(\Q_\infty/\Q)$. 
Moreover, it satisfies the functional equation 
\[L_p(\bdsE,T)=\vep_p(\bdsE)\Dmd{N^-N^4}_T^{-1}L_p(\bdsE,(1+T)^{-1}-1). \]
\end{prop}

\begin{proof}
Define $(\ulQ^o,P)=(Q^o_1,Q^o_2,Q^o_3,P)\in\frakX^\bal_{\bfI_4}$ with $Q_i^o$ as above, $k_P=2$ and $\eps_P=\hat\chi$.
 Then $\bfV_{(\ulQ^o,P)}=\bfV_{\bdsE}(2)\ot\hat\chi$ and $\hat\chi(L_p(\bdsE))=c_p\cdot L_{\bdsF,(2)}^{**}(\ulQ^o,P)$. The assertions follows from \corref{C:78}, \propref{P:funceq} and the equation $2^2\norm{f_i^\circ}=M_i\Lam(1,E_i,\Ad)$ by \cite[Proposition 3.3]{CH17Crelle}.  
\end{proof}

\subsection{The trivial zero conjecture for the triple product of elliptic curves}

We prove the trivial zero conjecture for the cyclotomic $p$-adic triple product $L$-function. 
We define a function on $\Zp$ by 
\[L_p(\bdsE,s):=L_p(\bdsE,\bfu^{s-2}-1).\]
We consider the case where $L_p(\bdsE,s)$ has a trivial zero at the critical value $s=2$. 
By \remref{R:82} below we essentially only need to consider the following two cases: 
\begin{itemize}
\item[(i)] all $E_1$, $E_2$ and $E_3$ have multiplicative reduction at $p$ such that $\al_1\al_2\al_3=1$.
\item[(ii)] $E_1$ has multiplicative reduction at $p$; $E_2$ and $E_3$ have good ordinary reduction at $p$ such that $\al_2=\al_1\al_3$.
\end{itemize}

\begin{Remark}\label{R:82}
Let $\beta_i=p\al_i^{-1}$. 
Then $\cale_p(\Fil^+\bfV_{\bdsE}(2))=0$ if and only if $L_p((\Fil^+\bfV_{\bdsE}(2))^\vee,1)^{-1}=0$ if and only if 
one of the following equations holds: 
\begin{align*}
&\beta_1\beta_2\beta_3=p^2, & 
&\beta_1\beta_2\al_3=p^2, &
&\beta_1\al_2\beta_3=p^2, & 
&\al_1\beta_2\beta_3=p^2. 
\end{align*}
The ordinality hypothesis rules out the first equation. 
The Ramanujan conjecture forces one or all of $E_i$ to have multiplicative reduction at $p$. 
When $E_1$ is multiplicative at $p$, we will have $\alp_1\in\{\pm 1\}$ and $\alp_2=\alp_1\alp_3$. 
\end{Remark}

In the above cases (i) and (ii), the trivial zero conjecture predicts that the leading coefficient of the Taylor expansion of $L_p(\bdsE,s)$ at $s=2$ should be essentially the product of Greenberg's $\mathscr L$-invariant for $\bdsE$ and the central value $L(\bdsE,2)$. Note that the localization of $\bfI_i$ at $Q^o_i$ is that of $\Lam$ at ${P_2}$, where $P_2$ is the principal ideal generated by $(1+X)\bfu^{-2}-1$, so $\bfI_i$ is contained in $\Lam[t_i^{-1}]$ with some $t_i(\bfu^2-1)\neq 0$. In what follows, we shall replace $\bfI_i$ by $\Lam[t_i^{-1}]$ with some $t_i(\bfu^2-1)\neq 0$. Let $\cU\subset \Zp$ be a neighborhood around $0$ such that $(t_1t_2t_3)(\bfu^{s+2}-1)\neq 0$ for any $s\in\cU$. To introduce Greenberg's $\mathscr L$-invariants, we let 
\begin{align*}
\bfa_i(s)&:=\bfa(p,\bdsf_i)(\bfu^{s+2}-1); & 
\ell_i&:=\al_i^{-1}\cdot\frac{\rmd\bfa_i(s)}{\rmd s}\Big|_{s=0}\quad (s\in\cU). 
\end{align*}
Note that $\bfa_i(0)=\al_i$ by definition. 
If $\alp_i=1$, then $-2\ell_i=\frac{\log_pq_{E_i}}{\ord_pq_{E_i}}$ by \cite[Theorem 3.18]{GS93} (the assumption $p\geq 5$ therein is not necessary). According to the discussion in \cite[\S 3]{Greenberg94Trivial}, Greenberg's $\mathscr L$-invariant for the Galois representation (\ref{tag:8dimGalois}) is given by 
\[\mathscr L_p(\bdsE):=\begin{cases}
-8\ell_1\ell_2\ell_3&\text{ in Case (i)};\\
4\ell_1^2&\text{ in Case (ii).}
\end{cases}\]
The non-vanishing of these $\mathscr L$-invariants is known, thanks to the work \cite{BDGP96Inv}. 
The aim of this section is to prove the following:

\begin{thm}[Trivial zero conjecture]\label{T:trivial}
\begin{enumerate}
\item In Case (i), $\ord_{s=2}L_p(\bdsE,s)\geq 3$, and 
\[\frac{L_p(\bdsE,s)}{(s-2)^3}\Big|_{s=2}=-p\mathscr L_p(\bdsE)\cdot \frac{
L(\bdsE,2)}{2^4\pi^5\Ome(\bdsE)}. \]
\item In Case (ii), $\ord_{s=2}L_p(\bdsE,s)\geq 2$ and 
\[\frac{L_p(\bdsE,s)}{(s-2)^2}\Big|_{s=2}=\mathscr L_p(\bdsE)(-p\al_2^{-2})(1-\al_2^{-2})^2\cdot \frac{L(\bdsE,2)}{2^4\pi^5\Ome(\bdsE)}. \]
\end{enumerate}
\end{thm}

\subsection{Improved $p$-adic $L$-functions}

We define an analytic function on $\cU^3\times\Zp\subset \Zp^4$ by 
\[L_p(x,y,z,s):=c_p\cdot\Dmd{N^-N^4}^{\frac{2s-(x+y+z)}{4}}L^{**}_{\bdsF,(2)}(\bfu^{x+2}-1,\bfu^{y+2}-1,\bfu^{z+2}-1,\bfu^{s+2}-1), \]
which satisfies 
\begin{align}
\label{E:fcneq}
L_p(0,0,0,s)&=\Dmd{N^-N^4}^{s/2}L_p(\bdsE,s+2), &
L_p(x,y,z,s)&=\vep_p(\bdsE)\cdot L_p(x,y,z,x+y+z-s). 
\end{align}
To apply the idea in \cite{GS93} (\cf\cite{BDJ17arXiv}), we will introduce several \emph{improved} \padic $L$-functions in \lmref{L:01.t} and \lmref{L:02}. The construction of these improved $p$-adic $L$-functions is merely a simple modification of the previous ones, but the details are quite tedious.  \def\imp{\dagger}
 
\begin{lm}[Improved $p$-adic $L$-functions]\label{L:01.t}
Suppose that $f_1^\circ$ is special at $p$, \ie $\al_1=\bfa_1(0)=\pm 1$. 
\begin{enumerate}
\item\label{L:01.t1} There exist a two-variable improved $p$-adic $L$-function $L_p^\imp(x,s)$ and a one-variable improved $p$-adic $L$-function $L_p^{\imp\imp}(s)$ such that 
\begin{align*}
L_p(x,s,s,s)=&\Big(1-\frac{\bfa_2(s)}{\bfa_1(x)\bfa_3(s)}\Big)\Big(1-\frac{\bfa_3(s)}{\bfa_1(x)\bfa_2(s)}\Big)L_p^\imp(x,s), &
L_p^\imp(s,s)=&\Big(1-\frac{\bfa_1(s)}{\bfa_2(s)\bfa_3(s)}\Big)L_p^{\imp\imp}(s).
\end{align*} 
\item\label{L:01.t2} For any positive integer $k$ with $k\con 2\pmod{p-1}$ and $k-2\in\cU$, we have the interpolation formula
\[L_p^\imp(0,k-2)=\cE^\imp(k-2)\cdot \frac{\Gamma(k-1)\Gamma(k)}{2^{2k-3}(\pi\sqrt{-1})^{2k+1}}\cdot \frac{L\bigl(\frac{1}{2},\pi_{f_1}\times\pi_{\bdsf_{2,k}}\times\pi_{\bdsf_{3,k}}\bigl)}{c_p^{-1}\Omega^\flat_{f_1}\Omega^\flat_{\bdsf_{2,k}}\Omega^\flat_{\bdsf_{3,k}}},\]
where $\pi_{\bdsf_{i,k}}$ is the automorphic representation generated by $\bdsf_{i,k}=\bdsf_i(\bfu^k-1)\in \cS_k(N_ip,1;\overline{\Q})$, and 
\begin{align*}\cE^\imp(s)&=(-\al_1)\bfa_2(s)^{-1}\bfa_3(s)^{-1}p^{s+1}(1-\al_1\cdot \bfa_2(s)^{-1}\bfa_3(s)^{-1}p^{s})^2.
\end{align*}
\item\label{L:01.t3} If $\vep_p(\bdsE)=-1$, then 
\begin{align*}
L_p^\imp(0,s)&=0, & 
\frac{\partial L_p^\imp}{\partial x}(0,0)&=(\ell_2+\ell_3-\ell_1)L_p^{\imp\imp}(0), & \ord_{s=2}L_p(\bdsE,s)&\geq 3. 
\end{align*} 
\item\label{L:01.t4} In Case (i), $L_p^{\imp\imp}(0)=-p\frac{L(\bdsE,2)}{2^4\pi^5\Omega(\bdsE)}$. 
\end{enumerate}
\end{lm}

\begin{proof}
The construction of these improved $L$-functions are similar to that of $L_{\bdsF,(a)}$ except that we need to replace the $\Lam_4$-adic modular form $\cG^{(a)}_{\ul{\chi}}$ in \subsecref{SS:6.5} with \emph{improved} ones. To do so, we have to go back to \subsecref{ssec:61} and modify the $p$-adic section $f_{\cD,s,p}$ used in the construction of the Siegel Eisenstein series $E_\A(g,f^{[k,\lam]}_{\cD,s,N})$.  In the notation of  \defref{def:psection}, for a datum $\cD=(\chi,\om_1,\om_2,\om_3)$ of characters of $\Zp^\x$ and a Bruhat-Schwartz function $\varphi_1\in\cS(\Qp)$, we modify the definition of Bruhat-Schwartz functions defined in \eqref{E:Phi} by 
\[\Phi_\cD(\varphi_1)\left(\begin{pmatrix}
u_1&x_3&x_2\\
x_3&u_2&x_1\\
x_2&x_1&u_3
\end{pmatrix}\right)=\prod_{i=1}^3\phi_i(u_i)\varphi_i(x_i), \]
where  
\begin{align*}
&\phi_1=\phi_2=\phi_3=\wh\bbI_{p\Zp}, & 
&\varphi_2=\varphi_3=\bbI_{\Zp}.
\end{align*}
Define the modified Bruhat-Schwartz functions by 
\begin{align*}
\Phi^\imp_\cD&=\Phi^{}_\cD(\wh\varphi_{\chi\om_1}), & 
\Phi^{\imp\imp}_\cD&=\Phi^{}_\cD(\bbI_{\Zp}).
\end{align*}
Following {Definition \ref{def:cell}}, we define the modified $p$-adic section $f^\bullet_{\cD,s,p}:=f_{\Phi^\bullet_\cD}(\chi\hat\om\boldsymbol{\al}_{\Qp}^s)$ for $\bullet\in\stt{\imp,\imp\imp}$. 
Then the local degenerate Whittaker functions for these modified $p$-adic sections are given by 
\begin{align*}
\cW_B(f^\imp_{\cD,s,p})&=(\chi\om_1)(2b_{23})\bbI_{\Xi_p^\imp}(B), & 
\cW_B(f^{\imp\imp}_{\cD,s,p})&=\bbI_{\Xi_p^{\imp\imp}}(B),
\end{align*}
for $B=(b_{ij})\in\Sym_3(\Q_p)$, where
\begin{align*}
\Xi_p^{\imp\imp}&:=\stt{(b_{ij})\in {\Sym_3^*(\Zp)}\;|\; b_{11},b_{22},b_{33}\in p\Zp}, &
\Xi_p^\imp&:=\stt{(b_{ij})\in \Xi_p^{\imp\imp}\;|\; 2b_{23}\in \Zp^\x}. 
\end{align*}
Notation is as in \subsecref{ssec:61}.  With the table of local Whittaker coefficients outside $p$ in \cite[page 210]{LR20MAnn}, one can apply the same argument in \cite[\S 5.2]{LiuZheng20JIMJ} to obtain two power series 
$\cG^\imp(T,X)\in\Zp\powerseries{T,X}\powerseries{q_1,q_2,q_3}$ and $\cG^{\imp\imp}(T)\in \Zp\powerseries{T}\powerseries{q_1,q_2,q_3}$ such that
for arithmetic points $(Q,P)$ with $k_Q=2$ we have 
\begin{align*}
\cG^\imp(Q,P)&=e_\ord\bfE_{\cD^\imp,N}^{[k_P,r,\lam]}(\tau,f^\imp_{\cD^\imp,s,N})|_{s=0}, &
\cG^{\imp\imp}(P)&=e_\ord\bfE_{\cD^{\imp\imp},N}^{[k_P,r,\lam]}(\tau,f^{\imp\imp}_{\cD^{\imp\imp},s,N})|_{s=0}
\end{align*}
with $\lam=(0,0,0)$ and $r=\frac{k_P}{2}-1$, where we have written 
\begin{align*}
\cD^\imp&:=(\ep_P\Om^{2-k_P},\ep_Q^{-1}\Om^{k_Q-2},\ep_P^{-1}\Om^{k_P-2},\ep_P^{-1}\Om^{k_P-2}), \\ 
\cD^{\imp\imp}&:=(\ep_P\Om^{2-k_P},\ep_P^{-1}\Om^{k_P-2},\ep_P^{-1}\Om^{k_P-2},\ep_P^{-1}\Om^{k_P-2}), \\ 
f^\bullet_{\cD^\bullet,s,N}&:=f^{[k_P,\lam]}_{s,\infty}\ot f^\bullet_{\cD^\bullet,s,p}\ot f_{s,N}^*\ot _{\ell\ndivides Np} f^0_{s,\ell}. 
\end{align*}
Let $\Lam_T:=\Zp\powerseries{T}$. As in Proposition \ref{P:GLam.E} we see that
\begin{align*}\cG^\imp(T,X)&\in \bfM^\ord(N,\Om^{2},\Lam_X)\wh\ot_{\Zp}  \bfM^\ord(N,\Om^{2},\Lam_T)\ot_{\Lam_T}  \bfM^\ord(N,\Om^{2},\Lam_T);\\
\cG^{\imp\imp}(T)&\in \bfM^\ord(N,\Om^{2},\Lam_T)\ot_{\Lam_T}  \bfM^\ord(N,\Om^{2},\Lam_T)\ot_{\Lam_T}  \bfM^\ord(N,\Om^{2},\Lam_T),\end{align*}
where $\bfM^\Ord(N,\Om^2,\Lam_T)$ is the space of ordinary $\Lam_T$-adic modular forms of level $N$ and character $\Om^2$. Let $\wtd\bfone_{\bdsf_i}$ be idempotent associated with $\bdsf_i$ in the Hecke algebra acting on the ordinary $\Lam$-adic module forms (not only cusp froms). Choose an element $H_i$ with $H_i(\bfu^2-1)\neq 0$ in the congruence ideal of $\bdsf_i$ among ordinary $\Lam$-adic modula forms (or rather the ideal generated by the denominators of $\wtd\bfone_{\bdsf_i}$. We define the improved $p$-adic $L$-functions $L^\imp_{\bdsF,(2)}(X,T)$ and $L^{\imp\imp}_{\bdsF,(2)}(T)$ as the first Fourier coefficients of  
\begin{align*}
&\wtd\bfone_{\bdsf_1}\ot\wtd\bfone_{\bdsf_2}\ot\wtd\bfone_{\bdsf_3}(\Tr_{N/N_1}\ot\Tr_{N/N_2}\ot \Tr_{N/N_3}(\cG^\imp))\in\Zp\powerseries{X,T}[\frac{1}{H^{\imp}}]; \\
&\wtd\bfone_{\bdsf_1}\ot\wtd\bfone_{\bdsf_2}\ot\wtd\bfone_{\bdsf_3}(\Tr_{N/N_1}\ot\Tr_{N/N_2}\ot \Tr_{N/N_3}(\cG^{\imp\imp}))\in\Zp\powerseries{T}[\frac{1}{H^{\imp\imp}}]
\end{align*}
respectively, where $H^{\imp}=t_1H_1(X)t_2H_2t_3H_3(T)$ and $H^{\imp\imp}=t_1H_1t_2H_2t_3H_3(T)$.
Define 
\begin{align*}
L_p^\imp(x,s)&:=c_p\cdot\Dmd{N^-N^4}^{\frac{-x}{4}}L^\imp_{\bdsF,(2)}(\bfu^{x+2}-1,\bfu^{s+2}-1), & 
L^{\imp\imp}_p(s)&:=c_p\cdot\Dmd{N^-N^4}^{\frac{-s}{4}}L^{\imp\imp}_{\bdsF,(2)}(\bfu^{s+2}-1). 
\end{align*}

In view of the proof of \lmref{lem:Garrett}, to prove the interpolation formulae for $L_p^\imp(x,s)$ and $L_p^{\imp\imp}(s)$, we need to compute the quantity $Z_p^*(f^\bullet_{\cD,s,p})$ defined in \eqref{E:padiczeta} attached to our modified $p$-adic sections $f_{\cD,s,p}^\bullet$ as well as  a subrepresentation $\pi_i$ of the induced representation $I(\mu_i,\nu_i)$ of $\GL_2(\Qp)$ with $\mu_i$ unramified for $i=1,2,3$. 
{Put $\ome_i=\mu_i\nu_i$. Applying Corollary \ref{cor:general2}}, we find that whenever $\chi\om_2$ and $\chi\om_3$ are unramified, 
\begin{align*}
Z^*_p(f^\imp_{\cD,s,p})&=Z^*_p({f_{\cD,s,p}})\prod_{i=2,3}L\Big(\frac{1}{2}-s,\chi^{-1}\mu_1^{-1}\mu_i^{-1}\nu_{5-i}^{-1}\Big)\\
\intertext{and that when $\chi\om_i$ are unramified for $i=1,2,3$, }
Z_p^*(f^{\imp\imp}_{\cD,s,p})&=Z^*_p(f^\imp_{\cD,s,p})L\Big(\frac{1}{2}-s,\chi^{-1}\nu_1^{-1}\mu_2^{-1}\mu_3^{-1}\Big). 
\end{align*}
From the proof of \thmref{P:interpolation} we can deduce the interpolation formulae for the improved $L$-functions. 
The formula for $\cale^+(s)$ follows from that for $Z^*_p(f_{\cD,s,p})$ proved in Proposition \ref{prop:13} and Remark \ref{rem:localfactor}. 
Observe that if $\pi_i\simeq\St\otimes\chi_i$ for $i=1,2,3$, i.e., $\mu_i=\chi_i\Abs_{\Q_p}^{1/2}$, then 
\[E_p\Big(\frac{1}{2}+s,\pi_1\times \pi_2\times \pi_3\Big)=-(\chi_1\chi_2\chi_3)(p)p^{1+s}(1-(\chi_1\chi_2\chi_3)(p)p^s)^3 , \]
from which (\ref{L:01.t4}) follows. 
 
Whenever $k>2$, the central sign for $L(s,\pi_{f_1}\times\pi_{\bdsf_{2,k}}\times\pi_{\bdsf_{3,k}})$ is $\vep_p(\bdsE)$. 
Therefore if $\vep_p(\bdsE)=-1$, then $L_p^\imp(0,s)=0$ by (\ref{L:01.t2}), which implies that $\frac{\partial L_p^\imp}{\partial x}(0,0)={\displaystyle\lim_{s\to 0}}\frac{L_p^\imp(s,s)}{s}$. 
The second equality of (\ref{L:01.t1}) gives the expression of ${\displaystyle\lim_{s\to 0}}\frac{L_p^\imp(s,s)}{s}$. 
We write 
\[L_p(x,y,z,s)=\sum_{j=0}^\infty A_j(x,y,z)\Big(s-\frac{x+y+z}{2}\Big)^j. \]
If $i\leq r:=\ord_{s=2}L_p(\bdsE,s)$, then 
\begin{align}
r&=\min\{j\;|\;A_j(0,0,0)\neq 0\}, & 
\lim_{s\to 2}\frac{L_p(\bdsE,s)}{(s-2)^i}&=A_i(0,0,0). \label{tag:order}
\end{align}
Letting $y=z=s=0$, we see by (\ref{L:01.t1}) that the power series 
\[\sum_{j=0}^\infty A_j(x,0,0)\Big(-\frac{x}{2}\Big)^j=(1-\al_1\bfa_1(x)^{-1})^2L_p^\imp(x,0) \]
has at least a double zero at $x=0$. 
If $\vep_p(\bdsE)=-1$, then since $A_{2n}(x,y,z)=0$ for all non-negative integers $n$ by the functional equation \eqref{E:fcneq}, we get $A_1(0,0,0)=0$ and $r\geq 3$. 
\end{proof}

\subsection{The proof of \thmref{T:trivial}(1)}

We discuss Case (i). 
Then $\vep_p(\bdsE)=-\vep(\bdsE)$ by Remark \ref{rem:sign}. 
First suppose that $\vep(\bdsE)=1$. 
The functional equation \eqref{E:fcneq} allows us to write
\[L_p(x,y,z,s)=A_1(x,y,z)\Big(s-\frac{x+y+z}{2}\Big)+A_3(x,y,z)\Big(s-\frac{x+y+z}{2}\Big)^3+\cdots \]
The proof of \lmref{L:01.t}(\ref{L:01.t3}) gives $A_1(0,0,0)=0$.  
From (\ref{tag:order}) and \lmref{L:01.t}(\ref{L:01.t4}) the formula boils down to 
\[A_3(0,0,0)=-8\ell_1\ell_2\ell_3 L^{\imp\imp}_p(0). \]

If we denote the degree two term of $A_1(x,y,z)$ by $ax^2+by^2+cz^2+dxy+eyz+fxz$, then the degree three term of $L_p(x,s,s,s)$ is given by 
\[L^{(3)}(x,s)=\{ax^2+(b+c+e)s^2+(d+f)xs\}(-x/2)+A_3(0,0,0)(-x/2)^3.\]
On the other hand, from \lmref{L:01.t}(\ref{L:01.t1}), (\ref{L:01.t3}) we find that 
\begin{align*}
L^{(3)}(x,s)&=(\ell_1 x+(\ell_3-\ell_2)s)\cdot (\ell_1 x+(\ell_2-\ell_3)s)x\cdot \lim_{x\to 0}x^{-1}L_p^\imp(x,0)\\
&=(\ell_1^2x^2-(\ell_2-\ell_3)^2s^2)x\cdot (\ell_2+\ell_3-\ell_1)L^{\imp\imp}_p(0).
\end{align*}
Comparing the coefficients of $x^2s,\,xs^2$ and $x^3$, we obtain the relations
\begin{align*}
d+f&=0, &
b+c+e&=2(\ell_2-\ell_3)^2(\ell_2+\ell_3-\ell_1)L^{\imp\imp}_p(0), &
4a+A_3(0,0,0)&=-8\ell_1^2(\ell_2+\ell_3-\ell_1)L_p^{\imp\imp}(0).
\end{align*}
By symmetry we get 
\begin{align*}
&d+e=0, & &e+f=0; \\ 
&a+c+f=2(\ell_1-\ell_3)^2(\ell_1+\ell_3-\ell_2)L^{\imp\imp}_p(0), & &a+b+d=2(\ell_1-\ell_2)^2(\ell_1+\ell_2-\ell_3)L^{\imp\imp}_p(0).
\end{align*}
From these equations we conclude that
$d=e=f=0$ and 
\begin{align*}
&a=\{(\ell_1-\ell_2)^2(\ell_1+\ell_2-\ell_3)+(\ell_1-\ell_3)^2(\ell_1+\ell_3-\ell_2)-(\ell_2-\ell_3)^2(\ell_2+\ell_3-\ell_1)\}L^{\imp\imp}_p(0), \\
&A_3(0,0,0)=-8\ell_1^2(\ell_2+\ell_3-\ell_1)L_p^{\imp\imp}(0)-4a=-8\ell_1\ell_2\ell_3L_p^{\imp\imp}(0).
\end{align*}

Next assume that $\vep(\bdsE)=-1$. 
Then $\vep_p(\bdsE)=1$. 
By \eqref{E:fcneq} and \lmref{L:01.t}(\ref{L:01.t1})  
\[\sum_{n=0}^\infty A_{2n}(x,s,s)\Big(\frac{s}{2}\Big)^{2n}=\biggl(1-\frac{\bfa_2(s)}{\bfa_1(x)\bfa_3(s)}\biggl)\biggl(1-\frac{\bfa_3(s)}{\bfa_1(x)\bfa_2(s)}\biggl)L_p^\dagger(x,s). \]
Since $L_p^\dagger(0,0)=0$, every term in the right hand side has degree at least three. 
In particular, the constant term $A_0(0,0,0)$ of the left hand side is zero. 
If we denote the degree two term of $A_0(x,y,z)$ by $\alp x^2+\bet y^2+\gam z^2+\xi xy+\eta yz+\zet xz$, then the degree two term of the left hand side is 
\[\alp x^2+(\bet+\gam+\eta)s^2+(\xi+\zet)xs+A_2(0,0,0)(x/2)^2. \]
It is zero, and so by symmetry we get  
\begin{align*}
A_2(0,0,0)&=-4\alp, & 
\bet+\gam+\eta&=0, & 
\xi+\zet&=0; \\
A_2(0,0,0)&=-4\bet, & 
\alp+\gam+\zet&=0, & 
\xi+\eta&=0; \\
A_2(0,0,0)&=-4\gam, & 
\alp+\bet+\xi&=0, & 
\eta+\zet&=0. 
\end{align*}
We arrive at $\xi=\eta=\zeta=\alp=\bet=\gam=A_2(0,0,0)=0$. 
Hence $\ord_{s=2}L_p(\bdsE,s)\geq 4$. 

\subsection{The proof of \thmref{T:trivial}(2)}

We discuss Case (ii). 
Then $\vep_p(\bdsE)=\vep(\bdsE)$ by Remark \ref{rem:sign}. 
If $\vep(\bdsE)=-1$, then $\ord_{s=2}L_p(\bdsE,s)\geq 3$ by \lmref{L:01.t}(\ref{L:01.t3}), and both sides of the declared identity are zero. 
We will consider the case $\vep(\bdsE)=1$, \ie $\Sigma^-$ has odd cardinality. 
Unlike Case (i) we cannot apply \lmref{L:01.t}(\ref{L:01.t3}). 
Our proof relies on the three-variable $p$-adic triple product $L$-function in the balanced case constructed in \cite[\S 4]{Hsieh_triple}. 

We will freely use the notation in \cite[\S 4]{Hsieh_triple}. Let $D$ be the definite quaternion algebra over $\Q$ of discriminant $N^-$ and $\bfS^D(N,\Lam)$ the space of $\Lam$-adic modular forms on $D^\x$ defined in \cite[Definition 4.1]{Hsieh_triple}. Let $\bdsf_i^D\in \bfS^D(N,\Lam[t_i^{-1}])$ be a Jacquet-Langlands lift of $\bdsf_i$ in the sense of \cite[\S 4.5]{Hsieh_triple}. 
Since we do not assume Hypothesis (CR,$\Sigma^-$) of \cite[\S 1.4]{Hsieh_triple}, we cannot choose $\bdsf_i^D$ to be a primitive Jacquet-Langlands lift as in \cite[Theorem 4.5]{Hsieh_triple}. 
Nonetheless, $\bdsf_i^D$ can be chosen so that $\bdsf_i^D(\bfu^2-1)$ is a non-zero Jacquet-Langlnads lift of $f_i$. Replacing the triple $\bdsF^D=(\bdsf_1^D,\bdsf_2^D,\bdsf_3^D)$ with the well-chosen test vectors in \cite[Definition 4.8]{Hsieh_triple}, we can associate to $\bdsF^D$ the three-variable \emph{theta element} $\Theta_{\bdsF^D}(X_1,X_2,X_3)$ in \emph{loc.cit}. 
 Define an analytic function on $\cU^3\subset \Z_p^3$ by 
 \[\Theta(x,y,z)=\Theta_{\bdsF^D}(\bfu^{x+2}-1,\bfu^{y+2}-1,\bfu^{z+2}-1). \]
By the interpolation formula for $\Theta_{\bdsF^D}$ in \cite[Theorem 7.1]{Hsieh_triple} (see Remark \ref{rem:11}), we can find an analytic function $H(x,y,z)$ with $H(0,0,0)\neq 0$ such that \[H(x,y,z)\cdot \Theta(x,y,z)^2=L_p\Big(x,y,z,\frac{x+y+z}{2}\Big). \]
To proceed, we introduce two-variable \emph{improved} theta elements.
\def\elt{h}
\begin{lm}[Improved theta elements]\label{L:02}
There exist analytic functions $\Theta_2^\ddagger(x,z)$, $\Theta_3^\ddagger(x,y)$ such that 
\begin{align*}
\Theta_2^\ddagger(0,0)&=-\Theta_3^\ddagger(0,0), \\
\Theta(x,x+z,z)&=\left(1-\frac{\bfa_2(x+z)}{\bfa_1(x)\bfa_3(z)}\right)\Theta^\ddagger_2(x,z), &
\Theta(x,y,x+y)&=\left(1-\frac{\bfa_3(x+y)}{\bfa_1(x)\bfa_2(y)}\right)\Theta^\ddagger_3(x,y).
\end{align*} 
\end{lm}
\begin{proof}
The idea of the construction of improved theta elements is close to \cite[Proposition 8.3]{Hsieh_triple}. The proofs are based on elementary but tedious calculations. We give a sketch here.  
For every integer $n$, let $R_n$ be the Eichler order of level $p^nN/N^-$ in $D$ and let $X_0(p^nN)=D^\x\bksl \wh D^\x/\wh R_n^\x$, where  $\wh D=D\otimes\wh\Q$ and $\wh R_n=R_n\otimes\wh\Z$. 
Through an isomorphism $R_0\ot\Zp\iso \Mat_2(\Zp)$ we define 
\[U_1(p^n):=\stt{g\in \wh R_n\;\biggl|\; g_p\con\pMX{*}{*}{0}{1}\pmod{p^n}}.\]
Recall that $\bfa_i(Q)=\bfa(p,\bdsf_{i,Q})$ and that $\uf_p\in\wh\Q^\x$ is the element with $\uf_{p,p}=p$ and $\uf_{p,\ell}=1$ for $\ell\neq p$. For all but finitely many arithmetic points $Q$ with $k_Q=2$, 
the specialization $\bdsf_{i,Q}^D:D^\x\bksl \wh D^\x/U_1(p^n)\to\Cp$ is a $p$-stabilized form on $\wh D^\x$ with the same Hecke eigenvalues with $\bdsf_{i,Q}$ and the central character $\ep_Q^{-1}:\Q^\x\bksl \wh \Q^\x/(1+p^n\wh\Z)^\x\to\mu_{p^\infty}$ for any sufficiently large $n$. In particular, $\bdsf^D_{i,Q}$ is a $\bfU_p$-eigenform with eigenvalue $\bfa_i(Q)$. Namely, 
\beq\label{E:Up.8}\bfU_p\bdsf^D_{i,Q}(\elt):=\sum_{b\in\Zp/p^n\Zp}\bdsf^D_{i,Q}\biggl(\elt\pMX{\uf_p^n}{b}{0}{1}\biggl)=\bfa_i(Q)^n\bdsf^D_{i,Q}(\elt)\quad (\elt\in\wh D^\x).\eeq
In what follows, we shall write $(\bdsf^D,\bdsg^D,\bdsh^D)=(\bdsf_1^D,\bdsf_2^D,\bdsf_3^D)$. Let $\rmN_D:D\to\Q$ be the reduced norm. 
Put $\tau_{p^n}=\pMX{0
}{1}{-\uf_p^n}{0}\in \GL_2(\Qp)\subset \wh  D^\x$. By definition,
\begin{multline}\label{E:9.t}
\Theta(\Qx,\Qy,\Qz)=\bfa_1(\Qx)^{-n}\bfa_2(\Qy)^{-n}\bfa_3(\Qz)^{-n}\\
\sum_{[\elt]\in X_0(p^nN)}\sum_{\begin{smallmatrix} b\in\Zp/p^n\Zp \\ c\in (\Zp/p^n\Zp)^\x\end{smallmatrix}}\bdsf_\Qx^D\biggl(\elt\pMX{\uf_p^n}{b}{0}{1}\biggl)\bdsg_{\Qy}^D\biggl(\elt\pMX{\uf_p^n}{b+c}{0}{1}\biggl)\bdsh_{\Qz}^D(\elt\tau_{p^n})\ep^\onehalf_{\Qx^{}\Qy^{} Q_3^{-1}}(c)\ep^\onehalf_{\Qx^{}\Qy^{}\Qz^{}}(\rmN_D(\elt)). 
\end{multline}
We replace the twisted diagonal cycle $\Delta_n$ in \cite[Definition 4.6]{Hsieh_triple} by the \emph{improved} diagonal cycle
\[\Delta^\ddagger_n:=\sum_{[\elt]\in X_0(Np^n)}\sum_{b\in\Zp/p^n\Zp}\biggl[\biggl({h}\pMX{\uf_p^n}{b}{0}{1},\elt\tau_{p^n},\elt\biggl)\biggl]. \]
We can define the regularized improved diagonal cycle by 
\[\Delta_\infty^\ddagger:=\prolim_{n\to\infty}(\bfU_p^{-n}\ot \bfU_p^{-n}\ot 1)e_E(\Delta^\ddagger_n),\]
and the improved theta element
\[\Theta_2^\ddagger(X_1,X_3):=(\bdsF^{D})^*(\Delta^\ddagger_\infty)(X_1,(1+X_1)(1+X_3)-1,X_3)\in\Zp\powerseries{X_1,X_3}[t^{-1}].\]
for $t=t_1\cdot t_2((1+X_1)(1+X_3)-1)\cdot t_3$. Put $\Theta_2^\ddagger(x,z):=\Theta_2^\ddagger(\bfu^{x+2}-1,\bfu^{z+2}-1)$ for $(x,z)\in\cU^2$.
By definition and \eqref{E:Up.8}, for all but finitely many arithmetic points $(\Qx,\Qz)$ with $k_\Qx=k_\Qz=2$ 
\[\Theta_2^\ddagger(\Qx,\Qz)=\bfa_2(\Qx\Qz)^{-n}\sum_{[\elt]\in X_0(Np^n)}\bdsf_\Qx^D(\elt)\bdsg_{\Qx\Qz}^D(\elt\tau_{p^n})\bdsh_{\Qz}^D(\elt)\ep_{\Qx\Qz}(\rmN_D(\elt)).\]
The above expression holds for any $n$ such that $p^n$ is bigger than the conductors of $\ep_{\Qx}$ and $\ep_{\Qy}$. Likewise we can define $\Theta_3^\ddagger\in\Zp\powerseries{X_1,X_2}$ and $\Theta_3^\ddagger(x,y)$ with the interpolation property:
\[\Theta_3^\ddagger(\Qx,\Qy)=\bfa_3(\Qx\Qy)^{-n}\sum_{[\elt]\in X_0(Np^n)}\bdsf_\Qx^D(\elt\tau_{p^n})\bdsg_{\Qy}^D(\elt)\bdsh_{\Qx\Qy}^D(\elt)\ep_{\Qx\Qy}(\rmN_D(\elt)).\]
To see the first equation in the lemma, we note that 
\begin{align*}
\Theta^\ddagger_2(0,0)=&\al_2^{-1}\sum_{[\elt]\in X_0(Np)}\bdsf^D_0(\elt)\bdsg^D_0(\elt\tau_p)\bdsh^D_0(\elt), &
\Theta_3^\ddagger(0,0)=&\al_3^{-1}\sum_{[a]\in X_0(Np)}\bdsf^D_0(\elt)\bdsg^D_0(\elt)\bdsh^D_0(\elt\tau_p).
\end{align*}
Since $\bdsf^D_0$ is a newform that is special at $p$, $\bdsf_0^D(\elt\tau_p)=(-\al_1)\bdsf_0^D(\elt)$, and hence $\Theta_2^\ddagger(0,0)=-\Theta^\ddagger_3(0,0)$. 

To prove the last equation in the lemma, it suffices to verify the following equation 
\begin{align}
\Theta(\Qx,\Qx\Qz,\Qz)&=\biggl(1-\frac{\bfa_2(\Qx\Qz)}{\bfa_1(\Qx)\bfa_3(\Qz)}\biggl)\Theta_2^\ddagger(\Qx,\Qz) \label{tag:81}
\end{align}
for all but finitely many arithmetic poitns $(\Qx,\Qz)$ with $k_\Qx=k_\Qz=2$. 
The formula for $\Theta^\ddagger_3$ can be done by a similar computation, so we leave it to the reader. Let $n$ be a sufficiently large integer. 
From \eqref{E:9.t}, we find that
\begin{align*}
&\bfa_1(\Qx)^n\bfa_2(\Qx\Qz)^n\bfa_3(\Qz)^np^{-n}\vol(\wh R_n^\x)\cdot \Theta(\Qx,\Qx\Qz,\Qz)\\
=&\int_{D^\x\bksl \wh D^\x}\rmd^\x\elt\,\sum_{c\in(\Zp/p^n\Zp)^\x}\bdsf_\Qx^D\biggl(\elt\pMX{1}{-c\uf_p^{-n}}{0}{1}\biggl)\bdsg^D_{\Qx\Qz}(\elt)\bdsh_\Qz^D\biggl(\elt\tau_{p^n}\pDII{1}{\uf_p^{-n}}\biggl)\ep_{\Qx}(c)\ep_{\Qx\Qz}(\rmN_D(\elt))\\
=&\int_{D^\x\bksl \wh D^\x}\rmd^\x \elt\,\sum_{c\in(\Zp/p^n\Zp)^\x}\bdsf_\Qx^D\biggl(\elt\tau_{p^n}\pMX{1}{-\uf_p^{-n}}{0}{c^{-1}}\biggl)\bdsg^D_{\Qx\Qz}(\elt\tau_{p^n})\bdsh_\Qz^D\biggl(\elt\pDII{1}{\uf_p^{-n}}\biggl)\ep_{\Qx\Qz}(\rmN_D(\elt)).
\end{align*} 
From \eqref{E:Up.8} and the equations \[\tau_{p^n}\pMX{1}{-\uf_p^{-n}}{0}{c^{-1}}=\pMX{\uf_p^n}{c^{-1}}{0}{1}\pMX{c^{-1}}{0}{-\uf_p^n}{1},\quad \ep_{Q_1}(\uf_p)=\ep_{Q_3}(\uf_p)=1,\] we find that the last integral equals 
\begin{align*}
&\int_{D^\x\bksl \wh D^\x}\rmd^\x \elt\,\sum_{c\in(\Zp/p^n\Zp)^\x}\bdsf_\Qx^D\biggl(\elt\pMX{\uf_p^n}{c}{0}{1}\biggl)\bdsg^D_{\Qx\Qz}(\elt\tau_{p^n})\bdsh_\Qz^D\biggl(\elt\pDII{\uf_p^n}{1}\biggl)\ep_{\Qx\Qz}(\rmN_D(\elt))\\
=&\int_{D^\x\bksl \wh D^\x}\rmd^\x \elt\, \bfa_1(Q_1)^n\cdot\biggl\{\bdsf_\Qx^D(\elt)-\bfa_1(Q_1)^{-1}\bdsf_\Qx^D\biggl(\elt\pDII{\uf_p}{1}\biggl)\biggl\}\bdsg^D_{\Qx\Qz}(\elt\tau_{p^n})\bdsh_\Qz^D\biggl(\elt\pDII{\uf_p^n}{1}\biggl)\ep_{\Qx\Qz}(\rmN_D(\elt))\\
=&\bfa_1(Q_1)^n\int_{D^\x\bksl \wh D^\x}\rmd^\x \elt\,\bdsf_\Qx^D(\elt)\bdsg^D_{\Qx\Qz}(\elt\tau_{p^n})p^{-n}\sum_{b\in\Zp/p^n\Zp}\bdsh_\Qz^D\biggl(\elt\pMX{\uf_p^n}{b}{0}{1}\biggl)\ep_{\Qx\Qz}(\rmN_D(\elt))\\
&-\bfa_1(Q_1)^{n-1}\int_{D^\x\bksl \wh D^\x}\rmd^\x \elt\,\bdsf_\Qx^D(\elt)\bdsg^D_{\Qx\Qz}(\elt\tau_{p^{n+1}})p^{-(n-1)}\sum_{b\in\Zp/p^{n-1}\Zp}\bdsh_\Qz^D\biggl(\elt\pMX{\uf_p^{n-1}}{b}{0}{1}\biggl)\ep_{\Qx\Qz}(\rmN_D(\elt))\\
=&\{(\bfa_1(\Qx)\bfa_3(\Qz)\bfa_2(\Qx\Qz)/p)^n\vol(\wh R_n^\x)-(\bfa_1(\Qx)\bfa_3(\Qz)/p)^{n-1}\bfa_2(\Qx\Qz)^{n+1}\vol(\wh R_{n+1}^\x)\}\Theta_2^\ddagger(\Qx,\Qz)\\
=&\bfa_1(\Qx)^n\bfa_3(\Qz)^n\bfa_2(\Qx\Qz)^n\biggl(1-\frac{\bfa_2(\Qx\Qz)}{\bfa_1(\Qx)\bfa_3(\Qz)}\biggl)p^{-n}\vol(\wh R_n^\x)\Theta_2^\ddagger(\Qx,\Qz).
\end{align*}
This verifies \eqref{tag:81}. 
\end{proof}

Now we return to the proof of \thmref{T:trivial}(2). 
Write $\Theta_x$ for the partial derivative $\frac{\partial\Theta}{\partial x}$. 
Put 
\begin{align*}
a&=\Theta_x(0,0,0), & 
b&=\Theta_y(0,0,0), & 
c&=\Theta_z(0,0,0). 
\end{align*}
Taking derivatives $\Theta(x,y,x+y)$ with respect to $x$ and $y$ at $(0,0)$ in \lmref{L:02}, we have 
\begin{align*}
a+c&=(\ell_1-\ell_3)\Theta_3^\ddagger(0,0), & 
b+c&=(\ell_2-\ell_3)\Theta_3^\ddagger(0,0). 
\end{align*} 
Similarly, we have 
\begin{align*}
a+b&=(\ell_1-\ell_2)\Theta^\ddagger_2(0,0)=(\ell_2-\ell_1)\Theta^\ddagger_3(0,0).
\end{align*}
These imply that
\begin{align*}
a&=0, & 
b&=(\ell_2-\ell_1)\Theta^\ddagger_3(0,0), &
c&=(\ell_1-\ell_3)\Theta^\ddagger_3(0,0). 
\end{align*}

On the other hand, by the functional equation \eqref{E:fcneq} we obtain the Taylor expansion 
\[L_p(x,y,z,s)=H(x,y,z)\Theta(x,y,z)^2+A_2(x,y,z)\cdot \Big(s-\frac{x+y+z}{2}\Big)^2+\cdots.\]
By \lmref{L:01.t}(\ref{L:01.t1}), we find 
\[(1-\al_1\bfa_1(x)^{-1})^2L_p^\imp(x,0)=H(x,0,0)\Theta(x,0,0)^2+A_2(x,0,0)\cdot x^2/4+\cdots. \]
From the vanishing of $\Theta_x(0,0,0)$
 we deduce that
\begin{align*}
&A_2(0,0,0)=4\ell_1^2L^{\imp}_p(0,0). 
\end{align*}
Lemma \ref{L:01.t}(\ref{L:01.t2}) and (\ref{tag:order}) complete the proof of Theorem \ref{T:trivial}(2). 

{
\subsection*{Acknowledgement} 
We especially thank the anonymous referee for a very careful reading and for pointing out an inaccuracy of our argument in the proof of Proposition 6.3. 
This work was partly supported by MEXT Promotion of Distinctive Joint Research Center Program JPMXP0723833165. 
Yamana thanks Max Planck Institut f\"{u}r Mathematik for an excellent working environment.} 

\bibliographystyle{amsalpha}
\bibliography{mybib_triple}
\end{document}

%% file: amssymbol.tex
\usepackage{amsmath}
\usepackage{amscd,amsthm,amssymb,amsfonts}
\usepackage{mathrsfs}
\usepackage{dsfont}
\usepackage{stmaryrd}
\usepackage{euscript}
\usepackage{expdlist}
\usepackage{enumerate}

\input xy
\xyoption{all}

\usepackage[OT2,T1]{fontenc}

\theoremstyle{plain}
\newtheorem{thm}{Theorem}[section]
\newtheorem{thmA}{Theorem}

\newtheorem*{thm*}{Theorem}

\newtheorem{lm}[thm]{Lemma}
\newtheorem{cor}[thm]{Corollary}
\newtheorem*{cor*}{Corollary}
\newtheorem{prop}[thm]{Proposition}
\newtheorem*{conj*}{Conjecture}

\theoremstyle{remark}
\newtheorem{remark}[thm]{Remark}

\theoremstyle{definition}
\newtheorem*{defn*}{Definition}
\newtheorem{Remark}[thm]{Remark}
\newtheorem{I_Remark*}{Remark}
\newtheorem{defn}[thm]{Definition}

\newtheorem*{hypothesis*}{Hypothesis}

\newcommand{\nc}{\newcommand}

\newcommand{\beq}{\begin{equation}}
\newcommand{\eeq}{\end{equation}}
\newcommand{\bpmx}{\begin{pmatrix}}
\newcommand{\epmx}{\end{pmatrix}}
\newcommand{\bbmx}{\begin{bmatrix}}
\newcommand{\ebmx}{\end{bmatrix}}
\newcommand{\wh}{\widehat}
\newcommand{\wtd}{\widetilde}

\newcommand{\beqcd}[1]{\begin{equation*}\label{#1}\tag{#1}}
\newcommand{\eeqcd}{\end{equation*}}

\numberwithin{equation}{section}

\def\parref#1{\ref{#1}}
\def\thmref#1{Theorem~\parref{#1}}

\def\propref#1{Proposition~\parref{#1}}
\def\corref#1{Corollary~\parref{#1}}     \def\remref#1{Remark~\parref{#1}}
\def\secref#1{\S\parref{#1}}

\def\lmref#1{Lemma~\parref{#1}}
\def\subsecref#1{\S\parref{#1}}
\def\defref#1{Definition~\parref{#1}}

\def\makeop#1{\expandafter\def\csname#1\endcsname
  {\mathop{\rm #1}\nolimits}\ignorespaces}

\makeop{Hom}   \makeop{End}   \makeop{Aut}   
\makeop{Pic} \makeop{Gal}       \makeop{Div} \makeop{Lie}
\makeop{PGL}   \makeop{Corr} \makeop{PSL} \makeop{sgn} \makeop{Spf}
 \makeop{Tr} \makeop{Nr} \makeop{Fr} \makeop{disc}
\makeop{Proj} \makeop{supp} \makeop{ker}   \makeop{Im} \makeop{dom}
\makeop{coker} \makeop{Stab} \makeop{SO} \makeop{SL} \makeop{SL}
\makeop{Cl}    \makeop{cond} \makeop{Br} \makeop{inv} \makeop{rank}
\makeop{id}    \makeop{Fil} \makeop{Frac}  \makeop{GL} \makeop{SU}
\makeop{Trd}   \makeop{Sp} \makeop{Tr}    \makeop{Trd} \makeop{Res}
\makeop{ind} \makeop{depth} \makeop{Tr} \makeop{st} \makeop{Ad}
\makeop{Int} \makeop{tr}    \makeop{Sym} \makeop{can} \makeop{SO}
\makeop{torsion} \makeop{GSp} \makeop{Tor}\makeop{Ker} \makeop{rec}
\makeop{Ind} \makeop{Coker}
 \makeop{vol} \makeop{Ext} \makeop{gr} \makeop{ad}
 \makeop{Gr}\makeop{corank} \makeop{Ann}
\makeop{Hol} 
\makeop{Fitt} \makeop{Mp} \makeop{CAP}

\def\ord{{ord}}

\def\Ord{{\mathrm{ord}}}

\def\Spec{\mathrm{Spec}\,}

\DeclareMathAlphabet{\mathpzc}{OT1}{pzc}{m}{it}
\DeclareSymbolFont{cyrletters}{OT2}{wncyr}{m}{n}
\DeclareMathSymbol{\SHA}{\mathalpha}{cyrletters}{"58}

\def\makebb#1{\expandafter\def
  \csname bb#1\endcsname{{\mathbb{#1}}}\ignorespaces}
\def\makebf#1{\expandafter\def\csname bf#1\endcsname{{\bf
      #1}}\ignorespaces}
\def\makegr#1{\expandafter\def
  \csname gr#1\endcsname{{\mathfrak{#1}}}\ignorespaces}
\def\makescr#1{\expandafter\def
  \csname scr#1\endcsname{{\EuScript{#1}}}\ignorespaces}
\def\makecal#1{\expandafter\def\csname cal#1\endcsname{{\mathcal
      #1}}\ignorespaces}

\def\doLetters#1{#1A #1B #1C #1D #1E #1F #1G #1H #1I #1J #1K #1L #1M
                 #1N #1O #1P #1Q #1R #1S #1T #1U #1V #1W #1X #1Y #1Z}
\def\doletters#1{#1a #1b #1c #1d #1e #1f #1g #1h #1i #1j #1k #1l #1m
                 #1n #1o #1p #1q #1r #1s #1t #1u #1v #1w #1x #1y #1z}
\doLetters\makebb   \doLetters\makecal  \doLetters\makebf
\doLetters\makescr
\doletters\makebf   \doLetters\makegr   \doletters\makegr

    \def\setminus{\smallsetminus}

\normalsize

\makeop{Ram} \makeop{Rep} \makeop{mass}

\makeop{Bl}
\def\abs#1{\left|#1\right|}
\def\norm#1{\lVert#1\rVert}

\def\Fp{{\mathbb F}_p}

\def\Qbarp{\C_p}
\def\Qp{\Q_p}
\def\Qbar{\ol{\Q}}

\def\Zp{\Z_p}

\def\diag#1{\mathrm{diag}(#1)}

\def\rmT{{\mathrm T}}
\def\rmN{{\mathrm N}}

\def\cA{{\mathcal A}}  

\def\cD{\mathcal D}
\def\cE{{\mathcal E}}
\def\cF{{\mathcal F}}  
\def\cG{{\mathcal G}}
\def\cL{{\mathcal L}}

\def\cJ{\mathcal J}
\def\cM{\mathcal M}

\def\cO{\mathcal O}
\def\cS{{\mathcal S}}
\def\cf{{\mathcal f}}
\def\cW{{\mathcal W}}

\def\cV{{\mathcal V}}
\def\cP{{\mathcal P}}

\def\cJ{\mathcal J}
\def\cN{\mathcal N}

\def\cY{\mathcal Y}
\def\cQ{\mathcal Q}
\def\cU{\mathcal U}

\def\bfc{\mathbf c}

\def\bfK{\mathbf K}

\def\bfM{\mathbf M}
\def\bfU{\mathbf U}
\def\bfT{{\mathbf T}}

\def\bda{\mathbf a}
\def\bff{\mathbf f}

\def\bfi{\mathbf i}
\def\bfu{\mathbf u}

\def\sL{\mathscr L}

\def\sD{\mathscr D}

\def\sB{\mathscr B}

\def\sS{\mathscr S}

\def\bbI{\mathbb I}

\newcommand{\Z}{\mathbf Z}
\newcommand{\Q}{\mathbf Q}
\newcommand{\R}{\mathbf R}
\newcommand{\C}{\mathbf C}
\newcommand{\A}{\mathbf A}    

\def\frakf{\mathfrak f}

\def\frakp{{\mathfrak p}}

\def\frakg{\mathfrak g}
\def\frakm{\mathfrak m}

\def\frakH{{\mathfrak H}}

\def\frakX{\mathfrak X}

\def\bfone{{\mathbf 1}}

\def\et{{\acute{e}t}}

\def\etale{{\'{e}tale }}

\def\padic{\text{$p$-adic }}

\def\BS{Bruhat-Schwartz }

\def\Frob{\mathrm{Frob}}

\def\surjto{\twoheadrightarrow}

\def\ot{\otimes}

\def\longto{\longrightarrow}
\def\ol{\overline}  \nc{\opp}{\mathrm{opp}} \nc{\ul}{\underline}

\newcommand{\pair}[2]{\langle #1, #2\rangle}

\newcommand{\pairing}{\pair{\,}{\,}}

\def\XYmatrix{\xymatrix@M=8pt} 
\def\ncmd{\newcommand}
\ncmd{\xysubset}[1][r]{\ar@<-2.5pt>@{^(-}[#1]\ar@<2.5pt>@{_(-}[#1]}
\ncmd{\XYmatrixc}[1]{\vcenter{\XYmatrix{#1}}}
\ncmd{\xyto}[1][r]{\ar@{->}[#1]}
\ncmd{\xyinj}[1][r]{\ar@{^(->}[#1]}
\ncmd{\xysurj}[1][r]{\ar@{->>}[#1]}
\ncmd{\xyline}[1][r]{\ar@{-}[#1]}
\ncmd{\xydotsto}[1][r]{\ar@{.>}[#1]}
\ncmd{\xydots}[1][r]{\ar@{.}[#1]}
\ncmd{\xyleadsto}[1][r]{\ar@{~>}[#1]}
\ncmd{\xyeq}[1][r]{\ar@{=}[#1]} \ncmd{\xyequal}[1][r]{\ar@{=}[#1]}
\ncmd{\xyequals}[1][r]{\ar@{=}[#1]}
\ncmd{\xymapsto}[1][r]{l\ar@{|->}[#1]}\ncmd{\xyimplies}[1][r]{\ar@{=>}[#1]}
\ncmd{\xyiso}{\ar[r]_-{\sim}}
\def\injxy{\ar@{^(->}}

\newcommand{\pMX}[4]{\begin{pmatrix}
{#1}& {#2}\\
{#3}&{#4}\end{pmatrix} }

 \newcommand{\pDII}[2]{\begin{pmatrix}{#1}&0
 \\0&{#2}\end{pmatrix}}

\newcommand{\seesaw}[4]{{#1}\ar@{-}[rd]\ar@{-}[d]&{#2}\ar@{-}[d]\\
{#3}\ar@{-}[ru]&{#4}}

\def\ie{i.e. }

\def\cf{\mbox{{\it cf.} }}

\def\uf{\varpi} 
\def\Abs{{\boldsymbol\alpha}} 

\def\ndivides{\nmid}

\def\x{{\times}}

\def\onehalf{{\frac{1}{2}}}
\def\al{\alpha}

\def\Lam{\Lambda}
\def\kap{\kappa}
\def\om{\omega}

\def\prolim{\varprojlim}
\def\iso{\simeq}
\def\con{\equiv}
\def\bksl{\backslash}
\newcommand\stt[1]{\left\{#1\right\}}
\def\ep{\epsilon}

\def\lam{\lambda}

\def\vep{\varepsilon}

\def\sg{\sigma}

\newcommand{\powerseries}[1]{\llbracket{#1}\rrbracket}

\renewcommand\pmod[1]{\,(\mbox{mod }{#1})}

\renewcommand\Re{\text{Re}\,}
\newcommand\Dmd[1]{\left<{#1}\right>} 

\def\Cp{\C_p}

\def\pmq{q}